\documentclass[reqno]{amsart}

\usepackage{amssymb}
\usepackage{stmaryrd}
\usepackage{hyperref}
\usepackage[all]{xy}
\usepackage{color}

\numberwithin{equation}{section}

\newtheorem{thm}[equation]{Theorem}
\newtheorem{cor}[equation]{Corollary}
\newtheorem{lem}[equation]{Lemma}
\newtheorem{prop}[equation]{Proposition}

\newtheorem{MainThm}{Theorem}

\theoremstyle{definition}

\newtheorem{defn}[equation]{Definition}
\newtheorem{notation}[equation]{Notation}

\theoremstyle{remark}

\newtheorem{rem}[equation]{Remark}

\newtheorem{assumption}[equation]{Assumption}
\newtheorem{construction}[equation]{Construction}

\newcommand{\colim}{\mathrm{colim}}
\newcommand{\scpr}[1]{\langle #1 \rangle}
\newcommand{\bQ}{\mathbb{Q}}
\newcommand{\bR}{\mathbb{R}}
\newcommand{\bZ}{\mathbb{Z}}

\newcommand{\bN}{\mathbb{N}}
\newcommand{\cD}{\mathcal{D}}
\newcommand{\cM}{\mathcal{M}}

\newcommand{\cE}{\mathcal{E}}
\newcommand{\cF}{\mathcal{F}}
\newcommand{\cG}{\mathcal{G}}

\newcommand{\Diff}{\mathrm{Diff}}

\newcommand{\GL}{\mathrm{GL}}
\newcommand{\SL}{\mathrm{SL}}
\newcommand{\Sp}{\mathrm{Sp}}
\newcommand{\SO}{\mathrm{SO}}

\newcommand{\id}{\mathrm{id}}

\newcommand{\bF}{\mathbb{F}}
\newcommand{\hAut}{\mathrm{hAut}}
\newcommand{\Aut}{\mathrm{Aut}}

\newcommand{\Fr}{\mathrm{Fr}}

\newcommand{\inter}[1]{\mathrm{int}(#1)}
\newcommand{\Hom}{\mathrm{Hom}}
\newcommand{\End}{\mathrm{End}}
\newcommand{\map}{\mathrm{map}}

\newcommand{\rank}{\mathrm{rank}}

\newcommand{\cA}{\mathcal{A}}
\newcommand{\Mod}{\mathrm{Mod}}
\newcommand{\op}{\mathrm{op}}
\newcommand{\blockdiff}{\widetilde{\mathrm{Diff}}}
\newcommand{\st}{\mathrm{st}}
\newcommand{\Hq}[2]{#1 \sslash #2}
\newcommand{\hq}[2]{  \ooalign{
    $\genfrac{}{}{1.2pt}1{#1}{#2}$\cr
    $\color{white}\genfrac{}{}{.4pt}1{\phantom{#1}}{\phantom{#2}}$}
}
\newcommand{\norm}[1]{\| #1 \|}

\newcommand{\sSet}{\mathcal{S}}
\newcommand{\inc}{\mathrm{inc}}
\newcommand{\hur}{\mathrm{hur}}
\newcommand{\bL}{\mathbb{L}}
\newcommand{\bK}{\mathbb{K}}
\newcommand{\ev}{\mathrm{ev}}

\newcommand{\cP}{\mathcal{P}}
\newcommand{\height}{\mathrm{ht}} 
\newcommand{\evrow}{\mathrm{er}}
\newcommand{\evcol}{\mathrm{ec}}
\newcommand{\coker}{\mathrm{coker}}
\newcommand{\im}{\mathrm{im}}
\newcommand{\sgn}{\mathrm{sgn}}

\newcommand{\hocolim}{\mathrm{hocolim}}

\newcommand{\MT}{\mathrm{MT}}

\title[Diffeomorphisms of some odd-dimensional manifolds]{Some rational homology computations for diffeomorphisms of odd-dimensional manifolds}

\author{Johannes Ebert}
\email{johannes.ebert@uni-muenster.de}
\address{
Mathematisches Institut\\
Universit\"at M{\"u}nster\\
Einsteinstr. 62\\
48149 M{\"u}nster\\
Germany
}

\author{Jens Reinhold}
\email{jens.reinhold@posteo.de}
\address{
}

\thanks{The authors were supported by the Deutsche Forschungsgemeinschaft (DFG, German Research Foundation) -- Project-ID 427320536 -- SFB 1442, as well as under Germany’s Excellence Strategy EXC 2044 -- 390685587, Mathematics M\"unster: Dynamics–Geometry–Structure.
}

\date{\today}

\begin{document}

\begin{abstract}
We calculate the rational cohomology of the classifying space of the diffeomorphism group of the manifolds $U_{g,1}^n:= \#^g(S^n \times S^{n+1})\setminus \inter{D^{2n+1}}$, for large $g$ and $n$, up to degree $n-3$. The answer is that it is a free graded commutative algebra on an appropriate set of Miller--Morita--Mumford classes. 

Our proof goes through the classical three-step procedure: (a) compute the cohomology of the homotopy automorphisms, (b) use surgery to compare this to block diffeomorphisms, (c) use pseudoisotopy theory and algebraic $K$-theory to get at actual diffeomorphism groups.
\end{abstract}

\maketitle

\tableofcontents

\section{Introduction}

\subsection{Context: Madsen--Weiss type theorems}

For a smooth compact manifold with boundary $M$, let $\Diff_\partial (M)$ denote the topological group of diffeomorphisms of a smooth compact manifold $M$ which are equal to the identity near $\partial M$. One of the success stories of differential topology in the 21st century was a (partial) computation of the cohomology of the classifying space $B \Diff_\partial (M)$ for some even--dimensional manifolds, by Madsen--Weiss \cite{MadsenWeiss} (for surfaces) and by Galatius--Randal-Williams \cite{GRW14} \cite{GRW18} \cite{GRW17} (in the higher dimensional case). The simplest case of these results concerns the manifolds 
\[
W_{g,1}^n:= \#^g (S^n \times S^n) \setminus \inter{D^{2n}}, 
\]
the connected sum of $g$ copies of $S^n \times S^n$, minus the interior of a disc, and are formulated in terms of the Madsen--Tillmann spectra\footnote{We use the notation from \cite{HebestreitPerlmutter} instead of that from \cite{GRW14}.} $\MT \theta^n_{2n}$, the Thom spectrum of the additive inverse of the universal $2n$-dimensional vector bundle over the $n$-connected cover $BO(2n) \langle n \rangle \to BO(2n)$. There is a natural map $\alpha_g^{2n}: B \Diff_\partial (W_{g,1}^n) \to \Omega^\infty_0 \MT \theta^n_{2n}$ to the unit component of the infinite loop space. These maps are compatible for varying value of $g$, and induce a map 
\[
\alpha_\infty^{2n}: \hocolim_{g \to \infty } B \Diff_\partial (W_{g,1}^n) \to \Omega^\infty_0 \MT \theta^n_{2n}
\]
in the limiting case, which is an integral homology equivalence (for $n=1$ by \cite{MadsenWeiss}, for $n\geq 3$ by \cite{GRW14} and for $n=2$ by \cite{GRW17}). This is complemented by homological stability theorems (unless $n=2$) due to \cite{Harer} and \cite{GRW18}, so that $\alpha_g^{2n}$ induces an isomorphism in homology in a range of degrees increasing with $g$. 

The rational cohomology of $\Omega^\infty_0 \MT \theta^n_{2n}$ (and more general Madsen--Tillmann spectra) is easily calculated using the standard tools from algebraic topology. The answer is that 
\[
H^* (\Omega^\infty_0 \MT \theta^n_{2n};\bQ) = \bF (s^{-2n} H^{*>2n} (BO(2n) \langle n \rangle;\bQ)),
\]
the free graded-commutative algebra generated by the desuspension of $H^{*>2n} (BO(2n) \langle n \rangle;\bQ)$. Let $\mu_c \in H^{k-2n} (\Omega^\infty_0 \MT \theta^n_{2n};\bQ)$ be the element corresponding to $c \in H^{k}  (BO(2n) \langle n \rangle;\bQ)$. The pullback $(\alpha_g^{2n})^* \mu_c\in H^{k-2n}(B \Diff_\partial (W_{g,1}^n);\bQ)$ is the tautological class $\kappa_c$ of the universal bundle over $B \Diff_\partial (W_{g,1}^n)$. Finally $H^* (BO(2n) \langle n \rangle;\bQ)$ is the polynomial algebra generated by the Pontrjagin classes $p_m$ with $\frac{n+1}{4} \leq m \leq n-1$ and the Euler class. So altogether, in a range of degrees, $H^*(B \Diff_\partial (W_{g,1}^n);\bQ)$ is a polynomial algebra in certain tautological classes. 

All these results are for \emph{even-dimensional} manifolds. The map $\alpha_g^{2n}$ is an instance of a more general construction: for an oriented $M$ of dimension $d$, there is a map $\alpha_M: B\Diff^+_\partial (M) \to \Omega^{\infty}_0 \MT SO(d)$. If the inclusion $\partial M \to M$ is $(k-1)$-connected, the map $\alpha_M$ can be refined to a map $B \Diff_\partial (M)=B \Diff_\partial^+ (M) \to \Omega^\infty_0 \MT \theta_M$, where $\theta_M: B \to BO(d)$ arises from the $k$th stage $M \to B \to BSO(d)$ of the Moore-Postnikov factorization of the classifying map $M \to BSO(d)$ of the tangent bundle $TM$. For $M=W_{g,1}^n$, we can take $k=n$ and get the map $\alpha_g^{2n}$.

It has been observed by the first named author \cite{Ebert13} that the classes $\alpha_M^* \mu_{L_m}$ associated to the components of the Hirzebruch $L$-class vanish, for each odd-dimensional $d$, though $\mu_{L_m} \neq 0 \in H^* (\Omega^\infty_0 \MT SO(d);\bQ)$. Except for small values of $m$, the pullback of $\mu_{L_m}$ to $H^* (\Omega^\infty_0 \MT \theta_M;\bQ)$ is also nonzero, hence the naive generalization of \cite{GRW14} will fail in odd dimensions.

\subsection{Main result}

Even though some substantial inroads into the odd-dimensional situation have been made recently \cite{PerlStab}, \cite{BotPerl}, \cite{Perlmutter}, \cite{HebestreitPerlmutter}, it remains a mystery and there does not seem to be a convincing conjectural odd-dimensional analogue of the main result of \cite{GRW14}. Our modest hope in this work is that our main result, Theorem \ref{mainthm:main} below, might eventually serve as a piece of evidence which helps to formulate an odd-dimensional version of these results. Let us consider the manifolds 
\[
U_{g,1}^n := \sharp^g (S^n \times S^{n+1})\setminus \inter{D^{2n+1}}
\]
which we view as an odd-dimensional variant of the manifolds $W_{g,1}^n$. Since $U_{g,1}^n$ is $(n-1)$-connected and $n$-parallelizable, one obtains a map 
\begin{equation}\label{eqn:madsentillmannmapug1n}
\beta_g^n: B \Diff_\partial (U_{g,1}^n) \to \Omega^\infty_0 \MT \theta_{2n+1}^n,
\end{equation}
where the target is the Madsen--Tillmann spectrum of $BO(2n+1) \langle n \rangle \to BO(2n+1)$. Note that 
\[
H^* (BO(2n+1) \langle n \rangle; \bQ) \cong \bQ [L_m \vert \tfrac{n+1}{4} \leq m \leq n]
\]
(there is no Euler class; and it is more convenient to use the components of the Hirzebruch $L$-class instead of the Pontrjagin classes as generators). Hence $H^* (\Omega^\infty_0 \MT \theta_{2n+1}^n;\bQ)$ is the exterior algebra generated by the set 
\[
\{ \mu_{L_{m_1}\cdots L_{m_r}}\vert \tfrac{n+1}{4} \leq m_1 \leq \ldots \leq m_r \leq n \}.
\]

\begin{MainThm}\label{mainthm:main}
The map 
\[
(\beta_g^n)^*: H^*(\Omega^\infty_0 \MT \theta_{2n+1}^n;\bQ) \to H^* (B \Diff_\partial (U_{g,1}^n);\bQ)
\]
is surjective in degrees $* \leq \min (\frac{g-4}{2}, n-3)$, and in that range of degrees, the kernel is the ideal generated by the classes $\mu_{L_m}$ (all $m$) and by the linear subspace $H^1(\Omega^\infty_0 \MT \theta_{2n+1}^n;\bQ)$.
\end{MainThm}

\begin{rem}
\begin{enumerate}
\item The theorem is vacuous for $n \leq 3$.
\item That $\mu_{L_m}$ lies in the kernel of $(\beta_g^n)^*$ is the main result of \cite{Ebert13}. 
\item The space $H^1(\Omega^\infty_0 \MT \theta_{2n+1}^n;\bQ)$ is zero if $n$ is even. If $n\equiv 1 \pmod 4$, say $n=4k+1$, one checks that 
\[
H^1 (\Omega^\infty_0 \MT \theta_{8k+3}^{4k+1};\bQ) = \bQ \{ \mu_{L_{2k+1}}\}, 
\]
and the triviality of $(\beta_g^n)^*$ in degree $1$ follows from \cite{Ebert13}.
If $n\equiv 3 \pmod 4$, say $n=4k-1$, we get a new relation. In that case, one checks that 
\[
H^1 (\Omega^\infty_0 \MT \theta_{8k-1}^{4k-1};\bQ) = \bQ \{ \mu_{L_{2k}}, \mu_{L_{k}^2}\},
\]
and the new relation is $ \kappa_{L_{k}^2}=0$. This relations holds more generally for all stably parallelizable manifolds of those dimensions. We give the fairly elementary proof in Proposition \ref{prop:vanishing-degree1} below; for the special manifolds $U_{g,1}^n$ the relation comes out of the proof of Theorem \ref{mainthm:main}. 
\item The bound in $g$ stems from a homological stability result due to Perlmutter \cite[Corollary 1.3.2]{Perlmutter}: the stabilization map $B \Diff_\partial (U_{g,1}^n) \to B \Diff_\partial (U_{g+1,1}^n)$ is homologically $\frac{g-2}{2}$-connected (with integral coefficients). 
\item The bound in $n$ comes from our method of proof which we describe informally in \S \ref{subsec:methodofproof} below.
\end{enumerate}
\end{rem}

\subsection{Relation to Hebestreit--Perlmutter's work}

Let us comment on the relationship of the present work with \cite{HebestreitPerlmutter}. The disjoint union 
\[
B \cD:= \coprod_{g} B \Diff_\partial (U_{g,1}^n)
\]
carries a natural structure of an algebra over the operad of little $(2n+1)$-discs. Hence we can form its group completion $\Omega B (B \cD)$ which is a $(2n+1)$-fold loop space. We clearly have $\pi_0 (\Omega B (B \cD)) \cong \bZ$, and an application of the group completion theorem shows that the homology of each of the components is $H_* (\Omega_0 B (B \cD)) \cong H_* (\hocolim_g B \Diff_\partial (U_{g,1}^n))$; hence Theorem \ref{mainthm:main} also evaluates the rational cohomology of $\Omega_0 B (B \cD)$ in a range of degrees. In \cite{HebestreitPerlmutter}, a larger $E_{2n+1}$-algebra is considered, namely 
\[
\cM_{2n+1}:= \coprod_{[W]} B \Diff_\partial (W),
\]
where $[W]$ runs through all diffeomorphism classes of $(n-1)$-connected and $n$-parallizable $(2n+1)$-manifolds $W$ with boundary $S^{2n}$. The main result of \cite{HebestreitPerlmutter} is that the group completion of $\cM_{2n+1}$ has the homotopy type of an infinite loop space if $n \geq 4$ and $n \neq 7$ (this is an odd-dimensional version of a theorem by Tillmann \cite{Tillmann} for surfaces). This is done by showing that $\Omega B \cM_{2n+1}$ is homotopy equivalent to the infinite loop space of a spectrum denoted $\mathrm{MT} \mathcal{L}_{2n+1}$. The latter is not a Madsen--Tillmann spectrum despite the notation, but rather obtained from a certain cobordism category of manifolds equipped with certain subspaces of their homology by using infinite loop space machinery. That cobordism category does not fit into the general theory of cobordism categories as in \cite{GMTW}; there is however a map $\mathrm{MT} \mathcal{L}_{2n+1} \to \MT \theta_{2n+1}^n$ of spectra. 

Clearly $B \cD \subset \cM_{2n+1}$ is a union of path components. However, while $\pi_0 (B \cD)\cong \bN_0$, $\pi_0 (\cM_{2n+1})$ is much larger; \cite[Proposition 3.2.5]{HebestreitPerlmutter} deduces a description of $\pi_0 (\cM_{2n+1})$ from \cite{Wall1967}. It is therefore not clear how to relate the group completions of $B \cD$ and of $\cM_{2n+1}$. As Fabian Hebestreit and Manuel Krannich pointed out to us, it seems conceivable that the map $\Omega_0 B (B\cD)\to \Omega_0 B \cM_{2n+1}$ is a rational homology equivalence. If that turns out to be true, Theorem \ref{mainthm:main} will also compute the rational homology of $\Omega^\infty_0 \MT \mathcal{L}_{2n+1}$ in a range of degrees. 

\subsection{Method of proof and outline of the paper}\label{subsec:methodofproof}

Having said that the methods of \cite{GRW14} must fail in the odd-dimensional case we need to say how we approach Theorem \ref{mainthm:main}. There is an established three-stage procedure to describe the topology of $B \Diff_\partial (M)$ for a high-dimensional manifold ($d=\dim (M) \geq 5$) in a range depending on $d$. The \emph{first step} is to get a hold on $B \hAut_\partial (M)$, the classifying space of the homotopy automorphisms of $M$, relative to the boundary. The \emph{second step} uses Quinn's space-level version of the surgery exact sequence \cite{Quinn} to compare $B \hAut_\partial (M)$ with the classifying space $B \blockdiff_\partial(M)$ of the \emph{block diffeomorphism} group; the difference is in terms of the $L$-theory of the group ring of $\pi_1 (M)$ and the normal invariants of $M$ in the form of the mapping space $\map_\partial (M;G/O)$. The \emph{third step} compares block diffeomorphisms to diffeomorphisms in a range of degrees, in terms of algebraic $K$-theory.

Our strategy in this paper is to first compute $H^* (\hocolim_g B \blockdiff_\partial (U_{g,1}^n);\bQ)$ for $* \leq n-3$, largely merging the first and second step, and then to use the comparison from the third step (which holds in a larger range of degrees) to arrive at $H^* (\hocolim_g B \Diff_\partial (U_{g,1}^n);\bQ)$, and finally to invoke \cite[Corollary 1.3.2]{Perlmutter}.

The map \eqref{eqn:madsentillmannmapug1n} does not extend to block diffeomorphisms, hence the spectrum $\MT \theta_{2n+1}^n$ does not play an important role in the calculation. However, the tautological classes can be extended by \cite{ERWblock}; this leads to a commutative diagram
\begin{equation}\label{maindiagram}
\xymatrix{
\bF (s^{-2n-1}H^{*>2n+1} (BSO(2n+1)\langle n \rangle;\bQ)) \ar[r]^-{\Psi}_-{\cong}  & H^* (\Omega^\infty_0 \MT \theta_{2n+1}^n;\bQ) \ar[d]^{\beta_g^\ast} \\
\bF (s^{-2n-1} H^{*>2n+1} (BO;\bQ)) \ar[dr]_-{\tilde{\Phi}_{U_{g,1}^n}} \ar[r]^-{\Phi_{U_{g,1}^n}} \ar[u] & H^* ( B \Diff_\partial (U_{g,1}^n);\bQ) \\
 & H^* ( B \blockdiff_\partial (U_{g,1}^n);\bQ) \ar[u].
}
\end{equation}
The map $\Psi$ sends the element $k_c$ corresponding to $c \in H^* (BO(2n+1)\langle n \rangle;\bQ)$ to $\mu_c$ and is an isomorphism. The left vertical map is induced from the map of spaces $BO(2n+1)\langle n \rangle \to BO$; it is surjective and the kernel is the ideal generated by the elements $k_{L_m c}$ with $4m \leq n$ and $c \in H^* (BO;\bQ)$ and by the elements $k_{p_m c}$ with $m >n$ and $c \in H^*(BO;\bQ)$. The lower right vertical map is induced by the natural comparison map from diffeomorphisms to block diffeomorphisms. The maps $\Phi_{U_{g,1}^n}$ and $\tilde{\Phi}_{U_{g,1}^n}$ send $k_c$ to $\kappa_c$ in the respective target algebras. 
Therefore Theorem \ref{mainthm:main} is equivalent to the following result.

\begin{thm}\label{mainthm:cohomological}
The map $\Phi_{U_{g,1}^n}$ is surjective in degrees $* \leq \min (\frac{g-4}{2}, n-3)$, and in that range of degrees, the kernel is the ideal generated by the following list of elements:
\begin{align}\label{eqn:mainthm2}
k_{L_m} & \; \; \text{all} \; \; m, \\
k_{L_m c} &\; \;  4m \leq n, \, c \in H^*(BO;\bQ),\\
k_{L_{m_0}L_{m_1}} &\; \;   4(m_0+m_1)=2n+2. 
\end{align}
\end{thm}

To state our result concerning $B \blockdiff_\partial (U_{g,1}^n)$, recall that 
\[
H^* (B \GL_\infty(\bZ);\bQ) \cong  \bF(\bigoplus_{k \geq 1} \bQ[4k+1])
\]
by Borel's famous calculation \cite{Borel}. The action of the block diffeomorphism group on $H_n (U_{g,1}^n;\bZ)\cong \bZ^g$ gives a map 
\[
H^* (B \GL_\infty(\bZ);\bQ) \to H^* (B \GL_g(\bZ);\bQ) \to H^*(B \blockdiff_\partial (U_{g,1}^n);\bQ). 
\]
Combining this with $\tilde{\Phi}_{U_{g,1}^n}$, we obtain an algebra map 
\begin{equation}\label{eqn:fromfreetocohblockdiff}
\Gamma: \bF (s^{-2n-1} H^{*>2n+1} (BO;\bQ)) \otimes H^* (B \GL_\infty(\bZ);\bQ) \to H^* (B \blockdiff_\partial (U_{g,1}^n);\bQ), 
\end{equation}
and we will prove the following result.
\begin{MainThm}\label{thm:blockdiffs}
The map $\Gamma$ is surjective in degree $* \leq n-3$ and for $g \gg n$, and in that range of degrees, the kernel is spanned by the same elements as given in Theorem \ref{mainthm:cohomological}. 
\end{MainThm}

\begin{rem}
Grey \cite[Theorem B]{Grey} has shown a homological stability result for block diffeomorphism groups which applies to the manifolds $U_{g,1}^n$; the result shows that Theorem \ref{thm:blockdiffs} holds in homological degrees $* \leq \min \{ n-3,\frac{g-2}{2}\}$. A similar bound can also deduced from our argument. The referee pointed us to the papers \cite{KupersMillerPatzt} and \cite{LiSun}. Using the results of these instead of the more classical \cite{VanderKallen} and \cite{Borel2}, one can improve the bound in \ref{thm:blockdiffs}. Since our main interest is in diffeomorphisms, not block diffeomorphisms, we do not spell out this here.
\end{rem}

In \S 2, we will show/review the vanishing theorems for the tautological classes that are entailed by Theorem \ref{mainthm:cohomological} and Theorem \ref{thm:blockdiffs}, namely that the listed elements lie in the kernel of $\Phi_{U_{g,1}^n}$ and $\Gamma$. We will actually only use the vanishing of $\kappa_{L_m}$ on $B \blockdiff_\partial (U_{g,1}^n)$ which we can just quote from \cite{ORWsignature}. The other two relations follow (for the manifolds $U_{g,1}^n$) from the subsequent calculations; we give the proper context for them in Propositions \ref{prop:vanishingobstruction} and \ref{prop:vanishing-degree1} for sake of completeness. In \S \ref{subsec:borelclasses}, we also show that the composition 
$H^{*>0} (B \GL_\infty;\bQ) \to H^* (B \blockdiff_\partial (U_{g,1}^n);\bQ) \to H^* (B \Diff_\partial (U_{g,1}^n;\bQ)$ is zero. This property is specific to the manifolds $U_{g,1}^n$ and follows from the Dwyer--Weiss--Williams index theorem \cite{DWW}; it plays a key role for deriving Theorem \ref{mainthm:cohomological} from Theorem \ref{thm:blockdiffs}.

\S \ref{sec:cohoblockdiffgeneral} is about the general theory behind the proof of Theorem \ref{thm:blockdiffs}. As already said, the proof of Theorem \ref{thm:blockdiffs} uses Quinn's theory, which roughly expresses the homotopy fibre of the forgetful map $B \blockdiff_\partial (M) \to B \hAut_\partial (M)$ in terms of $L$-theory and the normal invariants which is quite manageable for $1$-connected $M$. We refer to \cite{BerglundMadsen} for a more informative survey; more importantly, that paper contains a consequence of Quinn's theory \cite[Theorem 1.1]{BerglundMadsen} which allows us to use surgery theory completely as a black box. Some streamlining of their result leads to Theorem \ref{thm:blockdiffsversusmappingspace} which says that the calculation of $H^* (B \blockdiff_\partial (U_{g,1}^n);\bQ)$ is equivalent to the calculation of the cohomology of the homotopy quotient
\begin{equation}\label{eqn:mappingspacequients}
H^* (\hq{\map_\partial (U_{g,1}^n;BO_\bQ)^0 }{ \hAut_\partial (U_{g,1}^n)};\bQ), 
\end{equation}
up to some smallprint that we shall ignore for the moment ($\map_\partial (U_{g,1}^n;BO_\bQ)^0$ is the component of the mapping space containing the constant map). In \S \ref{subsec:cohmappingspace}, we also give a formula for the cohomology of $\map_\partial (U_{g,1}^n;BO_\bQ)^0$; this is mainly about bookkeeping elements.

The computation of \eqref{eqn:mappingspacequients} for large $g$ fills the largest portion of this paper. This is basically done by going through the Leray--Serre spectral sequences of the two fibrations 
\begin{equation}\label{eqn:firtfibresequintroduction}
\map_\partial (U_{g,1}^n;BO_\bQ)^0  \to \hq{\map_\partial (U_{g,1}^n;BO_\bQ)^0 }{ \hAut_\partial (U_{g,1}^n)^{\id}} \to B \hAut_\partial (U_{g,1}^n)^{\id}
\end{equation}
($\hAut_\partial (U_{g,1}^n)^{\id}$ is the unit component of $\hAut_\partial (U_{g,1}^n)$) and 
\begin{equation}\label{eqn:secondfibresequintroduction}
\hq{\map_\partial (U_{g,1}^n;BO_\bQ)^0 }{ \hAut_\partial (U_{g,1}^n)^{\id}} \to \hq{\map_\partial (U_{g,1}^n;BO_\bQ)^0 }{ \hAut_\partial (U_{g,1}^n)} \to B \pi_0 ( \hAut_\partial (U_{g,1}^n))
\end{equation}
(we actually consider some close variants which turn out to be more convenient for the computation, see \eqref{eqn:firstkeyfibresequence} and \eqref{eqn:secondkeyfibresequence} below).

In \S \ref{sec:homotopyautomorphisms}, we give the necessary calculations to compute the $E_2$-term of the spectral sequence of \eqref{eqn:firtfibresequintroduction} in a range of degrees, including the $d_2$-differential. These are mostly exercises in basic homotopy theory and are to a certain extent already included in \cite{Grey}. 
As we also explain in \S \ref{sec:homotopyautomorphisms}, the component group $\pi_0 (\hAut_\partial (U_{g,1}^n))$ is very close to $\GL_g (\bZ)$ so that naturally Borel's work on the cohomology of arithmetic group enters, which is one of the reasons why $g$ needs to be large in Theorem \ref{thm:blockdiffs}. For general reasons, the spectral sequence of \eqref{eqn:firtfibresequintroduction} is a sequence of $\GL_g (\bQ)$-representations, and we determine the $E_2$-term as such.

In order to compute \eqref{eqn:mappingspacequients}, we employ Borel's vanishing theorem \cite{Borel2} in a similar way to its use in \cite{ERWTorelli} or \cite{KupORW}; it shows that (approximately) 
\[
H^* (\hq{\map_\partial (U_{g,1}^n;BO_\bQ)^0 }{ \hAut_\partial (U_{g,1}^n)};\bQ) \cong H^* (\hq{\map_\partial (U_{g,1}^n;BO_\bQ)^0 }{ \hAut_\partial (U_{g,1}^n)^\id};\bQ)^{\GL_g (\bQ)}
\]
in a range of degrees. Since the category of rational $\GL_g (\bQ)$-representations is semisimple, the second algebra can be calculated by the $\GL_g (\bQ)$-invariant part of the spectral sequence of \eqref{eqn:firtfibresequintroduction}. This forces us to compute the invariant part of the $E_2$-term of the spectral sequence of \eqref{eqn:firtfibresequintroduction} which is a nontrivial calculation in invariant theory of the general linear group. The purely algebraic \S \ref{sec:representationtheory} is devoted to this computation.

In \S \ref{sec:coh:blockdiffspecial}, we put all these preliminaries together and finish the proof of Theorem \ref{thm:blockdiffs}. In degrees $* \leq n-3$, the evaluation of the invariant part of the spectral sequence of \eqref{eqn:firtfibresequintroduction} is, after the preliminary work, quite straightforward: the invariant part of the spectral sequence collapses (in the relevant range) at the $E_3$-stage, and the $E_3$-term is concentrated in the $0$th column; this solves all potential extension problems. The ultimate reason why this works out is that the only nonzero rational homotopy group of $B \hAut_\partial (U_{g,1}^n)^\id$ in the range $* \leq n-3$ is $\pi_2$.

\begin{rem}
In principle the general theory allows us to compute the cohomology of block diffeomorphisms in arbitrary degrees. We could not go beyond degree $n-3$, because $\pi_{n-2}(B \hAut_\partial (U_{g,1}^n)^\id) \otimes \bQ \neq 0$ and it is not clear to us how to understand the effect of this homotopy group in homology; here the rather naive technique we use (Leray--Serre spectral sequence) comes to its limits.

A more promising approach comes from recent work of Stoll \cite{Stoll}. He computed $H^* (B \hAut_\partial (U_{g,1}^n);\bQ)$ for $* \leq g-2$ and found contributions from graph homology, similar to Berglund-Madsen's results about $H^*(B \blockdiff_\partial (W_{g,1}^n);\bQ)$ in \cite{BerglundMadsen}. We have been informed by Stoll that the results of \cite{Stoll} can be used to compute $H^* (B \blockdiff_\partial (U_{g,1}^n);\bQ)$ in a range growing to infinity with $g$, but that work is still in progress.
\end{rem}

Let us now turn to the comparison of block diffeomorphisms and actual diffeomorphisms. The result is stated in detail as Theorem \ref{thm:comparisondiffblockdiff} below; we have formulated it in such a way that it is largely independent of our computation of $H^* (B \blockdiff_\partial (U_{g,1}^n);\bQ)$. 
In general, the comparison of diffeomorphisms and block diffeomorphisms is in terms of pseudo-isotopy theory and algebraic $K$-theory, with the stable h-cobordism theorem \cite{WJR} and Igusa's stability theorem \cite{IgusaStability} as the main points; the last one enforces a bound depending on the dimension of the manifold. An elaborate formulation of this step was given by Weiss--Williams in \cite{WeissWilliams}, but a simpler variant suffices for us. 
Let us describe briefly how the comparison is done. Recall the classical result by Farrell--Hsiang \cite{FarrellHsiang} stating that
\[
 \pi_k (B \Diff_\partial (D^{2n+1})) \otimes \bQ\cong 
\begin{cases}
\bQ & k \equiv 0 \pmod 4\\
0 & k \not \equiv 0 \pmod 4
\end{cases}
\]
holds in a range of degrees. This range is nowadays known to be roughly $2n$, by recent work of Krannich \cite{Krannich} and Krannich--Randal-Williams \cite{KrannRW}. An instance of Morlet's lemma of disjunction states that the homotopy fibre of 
\begin{equation}\label{eqn:fibrationblockdifffiff}
B \Diff_\partial (U_{g,1}^n) \to B \blockdiff_\partial (U_{g,1}^n)
\end{equation}
is rationally equivalent to $B \Diff_\partial (D^{2n+1})$ up to degree approximately $2n$, and the main result of \cite{EbertDisc} (which is a consequence of \cite{BotPerl} and \cite{PerlStab}) says that the inclusion of the homotopy fibre into the total space of \eqref{eqn:fibrationblockdifffiff} induces the trivial map on rational homology in a range of degrees. 

To combine those facts for a homological conclusion, we apply the Quillen plus construction to the spaces $\hocolim_g B \Diff_\partial (U_{g,1}^n)$ and $\hocolim_g B \blockdiff_\partial (U_{g,1}^n)$; after plus construction the latter two spaces are homotopy commutative $H$-spaces. It is crucial that the map from the homotopy fibre of \eqref{eqn:fibrationblockdifffiff} to the homotopy fibre of $\hocolim_g B \Diff_\partial (U_{g,1}^n)\to \hocolim_g B \blockdiff_\partial (U_{g,1}^n)$ is highly connected; for the proof of this fact, we borrow an argument from Krannich's paper \cite{Krannich}, using Morlet's lemma of disjunction a second time. After this maneuver, the deduction of Theorem \ref{mainthm:cohomological} from Theorem \ref{thm:blockdiffs} is straightforward.

\subsubsection*{Notations}

We use the following standard notations in this paper. For two $\bN_0$-graded algebras $A$ and $B$, $A \otimes B$ always denotes the \emph{graded} tensor product, to conform the conventions of homological algebra. The free graded-commutative algebra generated by a graded vector space $V$ is denoted $\bF (V)$. If $V$ is a vector space and $n \in \bN_0$, we let $V[n]$ be the graded vector space which is $V$ in degree $n$ and $0$ otherwise (and not the degree shift of a graded vector space). 
For a module $N$ over a commutative ring $R$, we write $N^\vee:= \Hom_R (N;R)$ for the dual.

\subsection{Acknowledgements}

The authors would like to thank Alexander Berglund, Fabian Hebestreit, Lutz Hille, Manuel Krannich, Sander Kupers, Ib Madsen, Thomas Nikolaus, Oscar Randal-Williams and Robin Stoll for helpful conversations about various aspects of this work. We also thank Andrea Bianchi and the anonymous referee for their careful reading of earlier versions of this paper. Last but not least, it is a pleasure to thank Jerzy Weyman. Without his help, we would not have been able to carry out the crucial invariant-theoretic calculation in Proposition \ref{prop:keylemma-invarianttheory}, and would not have been able to get this project to a conclusion. 

\section{Characteristic classes of smooth and block bundles}\label{subsec:characteristicclasses}

\subsection{Automorphism groups}

Let us first establish some notation. 
Let $M^d$ be a compact oriented smooth manifold with boundary. We write $\Diff(M)$ for the diffeomorphism group with its $C^\infty$-topology and $\Diff_\partial(M) \subset \Diff (M)$ for the subgroup of diffeomorphisms which are the identity near $\partial M$. We furthermore let $\Diff^+ (M)\subset \Diff (M)$ and $\Diff^+_\partial (M) \subset \Diff_\partial (M)$ the subgroups of orientation-preserving diffeomorphisms; note that $\Diff^+_\partial (M)=\Diff_\partial (M)$ if the inclusion $\partial M \to M$ is $0$-connected. 

In the present paper, we make heavy use of the \emph{block diffeomorphism} group $\blockdiff(M)$. We won't repeat the definition here and refer instead to \S 1 and 2 of \cite{Krannich} for an up-to-date exposition. There are block analogues $\blockdiff_\partial(M)$, $\blockdiff^+ (M)$, $\blockdiff^+_\partial(M)$ of respective diffeomorphism groups. Let us also note that the natural map
\[
I: \Diff(M) \to \blockdiff(M)
\]
is \emph{by definition} $0$-connected; the same holds for the decorated versions.

We shall need various flavours of homotopy automorphism groups. When forming mapping spaces, we secretly replace all spaces that occur by their singular simplicial set, and view the mapping space as a simplicial set. For a CW-pair $(X,A)$ and a pointed space $Y$, we let $\map_A (X;Y)$ be the space of maps $X \to Y$ whose restriction to $A$ is the constant map to the basepoint in $Y$ (or more formally the fibre of the restriction map $\map(X;Y) \to \map(A;Y)$ over the constant map). For a CW-pair $(X,A)$, we let $\hAut_A (X)$ be the monoid of all homotopy self-equivalences of $X$ which are the identity on $A$. For an oriented compact manifold, we let also $\hAut_\partial (M)^+ \subset \hAut_\partial (M)$ the submonoid of those self-equivalences which preserve the fundamental homology class. 
There are natural maps 
\[
\blockdiff_\partial (M) \to \hAut_\partial (M)\; \text{and} \; \blockdiff_\partial^+ (M) \to \hAut_\partial^+ (M)
\]
(or rather a zig-zag, see the discussion in \cite[p. 98 f]{BerglundMadsen} for more details).

We also need homotopy quotients by homotopy automorphisms; we use the following models. Suppose that a grouplike topological monoid $G$ acts from the left on a space $X$ and that $Z$ is a further space. Then $G$ acts from the right on $\map(X;Z)$, and by the notation $\Hq{\map(X;Z)}{G}$, we mean the two-sided bar construction $\norm{B_\bullet (\map(X;Z),G,*)}$. The two-sided bar construction 
\[
\Hq{(\map(X;Z) \times X)}{G} :=\norm{B_\bullet (\map(X;Z),G,X)} 
\]
comes with a map to $\Hq{\map(X;Z)}{G}$ with fibre $X$ and with a map 
\begin{equation}\label{augmentation-twosidedbarconstriction}
\Hq{(\map(X;Z) \times X)}{G} \to Z. 
\end{equation}
The latter arises from an augmentation $\epsilon_\bullet:B_\bullet (\map(X;Z),G,X)\to Z$; on the space of $p$-simplices $B_p (\map(X;Z),G,X)= \map(X;Z)\times G^p \times X$, the augmentation is the map $(f,g_1,\ldots,g_p,x) \mapsto f\circ g_1 \circ \ldots \circ g_p (x)$. This construction has an obvious variant for relative homotopy automorphisms and relative mapping spaces.

Assume now that $V \to X$ is a vector bundle and that $C \subset A$ is a subcomplex. We let $\hAut_A^C (V)$ be the monoid of all pairs $(f,\hat{f})$ where $f \in \hAut_A(X)$ and $\hat{f}: V \to V$ is a bundle map covering $f$ which is fibrewise an isomorphism, and such that $f|_{V|_C}$ is the identity, see \cite[p. 107f]{BerglundMadsen} for more details. There is a stable version of that construction given on p. 110 loc.cit.; we define 
\[
\hAut_A^C (V)^{\st}:= \colim_k \hAut_A^C (V \oplus \bR^k). 
\]
All the monoids we just introduced are grouplike $E_1$-spaces and therefore admit classifying spaces. The classifying space $B \hAut_A^C(V)^\st$ has a convenient description as follows. If $V|_C$ is stably trivial and a stable trivialization is chosen, Proposition 4.13 of \cite{BerglundMadsen} provides a weak equivalence
\begin{equation}\label{eqn:prop4.14BM}
B \hAut_A^C (V)^{\st} \simeq (\hq{\map_C (X;BO)}{\hAut_A(X)})_V ,
\end{equation}
where 
\[
( \hq{\map_C (X;BO)}{ \hAut_A(X)})_V \subset \hq{\map_C (X;BO)}{\hAut_A(X)}
\]
denotes the connected component determined by a fixed classifying map $\lambda:X \to BO$ of $V$ which extends the given trivialization of $V|_C$, viewed as a point in $\map_C (X;BO)$. The map in \eqref{eqn:prop4.14BM} arises as follows: the total space of the universal fibration over $B \hAut_A^C (V)^{\st}$ with fibre $X$ carries a stable vector bundle which is built from $V$. For a detailed construction on the point-set level, we refer to \cite{BerglundMadsen}. 
When $V$ is stably trivial, \eqref{eqn:prop4.14BM} can be reformulated as follows. 

\begin{lem}\label{lem:Bofbundlehomotopyautomorphisms}
Assume that $V\to X$ is stably trivial, and that a stable trivialization of $V|_C$ is fixed. Then there is a weak equivalence
\[
B \hAut_A^C (V)^{\st} \simeq \hq{\map_C (X;BO)^0}{\hAut_A(X)}, 
\]
where $\map_C (X;BO)^0 \subset \map_C (X;BO)$ denotes the component of the constant map. 
\end{lem}

\begin{proof}
According to \eqref{eqn:prop4.14BM}, we must prove that 
\[
(\hq{\map_C (X;BO) }{ \hAut_A(X)})_V \simeq \hq{\map_C(X;BO)^0}{\hAut_A(X)}
\]
if $V$ is stably trivial. Now $\pi_0 (\map_C (X;BO)) = KO^0 (X,C)$ is an abelian group and 
$\hAut_A(X)$ acts by group automorphisms, and therefore fixes the neutral element. So $\hq{\map_C(X;BO)^0 }{\hAut_A(X)}$ is the component of $\hq{\map_C(X;BO)}{\hAut_A(X)}$ containing the constant map, which is exactly $(\hq{\map_C (X;BO)}{\hAut_A(X)})_V$. 
\end{proof}

One important feature of block diffeomorphisms is the existence of the \emph{derivative map}
\begin{equation}\label{eqn:detrivativemap}
D: \blockdiff_\partial (M) \to \hAut_\partial^\partial (TM)^{\st},
\end{equation}
which is a map of $E_1$-monoids and can therefore be delooped to a map
\begin{equation}\label{eqn:detrivativemap2}
BD: B\blockdiff_\partial (M) \to B\hAut_\partial^\partial (TM)^{\st}. 
\end{equation}
The derivative map is constructed in \cite[\S 4.3]{BerglundMadsen}, see also \cite[\S 1.9]{Krannich}, and is an expression of the fact, first proven in \cite{ERWblock} and expanded on in \cite[\S 2.4]{HLLRW}, that block bundles have a stable vertical tangent bundle. By virtue of its definition, the derivative map is a map over $\hAut_\partial (M)$. 
If $M$ is stably parallelizable, we can use Lemma \ref{lem:Bofbundlehomotopyautomorphisms} to rewrite \eqref{eqn:detrivativemap2} in the form
\begin{equation}\label{eqn:detrivativemap3}
BD: B\blockdiff_\partial (M) \to \hq{\map_\partial (M;BO)^0}{\hAut_\partial(M)}.
\end{equation}

\subsection{Tautological classes}

Tautological classes (aka Miller--Morita--Mumford clas\-ses or $\kappa$-classes) for block bundles have been constructed in \cite{ERWblock} and more systematically in \cite{HLLRW}. The most streamlined construction can be given using the derivative map \eqref{eqn:detrivativemap2}, and we sketch the definition briefly, in a level of generality that will prove to be useful for us later on. 
\begin{construction}\label{construction-kappaclasses}
Assume that $M^d$ is a compact oriented smooth manifold with boundary. Consider the universal fibration pair 
\[
(E,\partial E):= (\hq{\map_\partial (M;BO) \times M}{\hAut^+_\partial (M)},\hq{\map_\partial (M;BO) \times \partial M}{\hAut^+_\partial (M)})\to B:=\hq{\map_\partial (M;BO)}{\hAut^+_\partial (M)}
\] 
with fibre $(M,\partial M)$. The construction \eqref{augmentation-twosidedbarconstriction} gives a map $\epsilon: E \to BO$ sending $\partial E$ to the basepoint, and the restriction of $\epsilon$ to the fibre over the basepoint in $B\hAut_\partial (M)$ can be identified with the evaluation map 
\[
\ev: \map_\partial (M;BO) \times M \to BO. 
\]
We may think of $\epsilon$ as a stable vector bundle on $E$, trivialized on $\partial E$. 
The Leray--Serre spectral sequence for the fibration pair $(E,\partial E) \to B$ yields the Gysin map
\[
\pi_!: H^k (E;\partial E) \to E_\infty^{k-d,d} \subset E_2^{k-d,d} \cong H^{k-d}(B;H^d (M;\partial M)) \to H^{k-d}(B)
\]
(with coefficients in an arbitrary ring). Given a class $c \in H^k (BO)$, we can therefore form 
\[
\kappa_c := \pi_! (\epsilon^* c) \in H^{k-d}(\hq{\map_\partial (M;BO)}{\hAut^+_\partial (M)}).
\]

Let $M^d$ be an oriented compact smooth manifold with boundary and assume that $TM|_{\partial M}$ is stably trivial. 
Combining the derivative map $BD$ with \eqref{eqn:prop4.14BM} yields a map
\[
B \blockdiff_\partial^+ (M) \to B \hAut_\partial^\partial (TM)^\st \to \hq{\map_\partial (M;BO)}{\hAut^+_\partial (M)},
\]
and we can pull back $\kappa_c$ to a class, also denoted $\kappa_c \in H^{k-d}(B \blockdiff_\partial^+ (M))$. Pulling this further back along $BI: B \Diff_\partial (M) \to B \blockdiff_\partial (M)$, we obtain the usual $\kappa$-classes on classifying spaces of diffeomorphism groups. 
\end{construction}

A map $\gamma$ from a space $X$ to one of the classifying spaces $B \Diff^+_\partial (M)$, $B \blockdiff_\partial^+ (M)$ or $\hq{\map_\partial (M;BO)}{\hAut^+_\partial (M)}$ classifies $E \to X$, which is a smooth fibre bundle / block bundle / fibration with a stable vector bundle on its total space. In such situations, we use the suggestive notation $\kappa_c (E):= \gamma^* \kappa_c$.

We now specialize to rational coefficients. Recall that $H^* (BO;\bQ)$ is the polynomial algebra in the Pontrjagin classes. For our purposes, it is useful to observe that one can also write 
\[
H^* (BO;\bQ) \cong \bQ[L_1,L_2, \ldots], 
\]
where $L_m \in H^{4m}(BO;\bQ)$ is the $m$th component of the Hirzebruch $L$-class. The above is true by the well-known fact \cite[p. 14]{Hirzebruch} that the coefficient $a_m$ of $p_m$ in $L_m$ is nonzero. 

\subsection{Some vanishing theorems for tautological classes}\label{sebsec:vanishingtheorems}

In this subsection, we review the three vanishing theorems that are entailed by Theorem \ref{mainthm:main}. Only one of them  (Theorem \ref{thm:vanishing}) is actually used in the proof of our main theorem; the other two (Proposition \ref{prop:vanishingobstruction} and Proposition \ref{prop:vanishing-degree1}) are only stated for sake of completeness; the fact that they are valid for the manifolds $U_{g,1}^n$ is a byproduct of our computations below. 

\subsubsection*{An additivity property}

Let us begin with a fact which seems to be well-known for diffeomorphism groups; the argument we give is essentially contained in \cite{Miller}, \cite{Morita}. 

\begin{lem}\label{lem:additivity-kappaclass}
Let $M$ and $N$ be compact oriented $d$-manifolds with boundary with a common (closed) part $\partial_0 \subset \partial M, \partial N$ of their boundary. Let 
\[
\mu: \hq{\map_\partial (M;BO)}{\hAut_\partial^+ (M)} \times \hq{\map_\partial (N;BO)}{\hAut_\partial^+ (N)} \to  \hq{\map_\partial (M\cup_{\partial_0} N;BO)}{\hAut_\partial^+ (M\cup_{\partial_0} N)}
\]
be the obvious gluing map, let 
\[
p_M: \hq{\map_\partial (M;BO)}{\hAut_\partial^+ (M)} \times \hq{\map_\partial (N;BO)}{\hAut_\partial^+ (N)} \to  \hq{\map_\partial (M;BO)}{\hAut_\partial^+ (M)}
\]
be the projection and define $p_N$ similarly.
Then for each $c \in H^k (BO)$ with $k>0$, we have 
\[
\mu^* \kappa_c = p_M^* \kappa_c+ p_N^* \kappa_c.
\]
\end{lem}

\begin{proof}
Assume that $\pi^E: E\to X$ and $\pi^F: F  \to X$ are two oriented fibrations with fibres $M$ and $N$ and trivialized boundaries, and that $E$ and $F$ contain a common part $A \subset \partial E, \partial F$ of their boundary, and let $V \to E$ and $W \to F$ be two (stable) vector bundles of the same rank $n$ which are trivialized over the respective boundaries and let $c \in H^k(BO)$ with $k>0$. The lemma amounts to the statement that in the above situation, we have
\begin{equation}\label{eqn:primitivity-kappa-class}
\pi^{E \cup_A F}_! (c(V \cup W))= \pi^E_! (c(V)) + \pi^F_! (c(W)) \in H^{k-d}(X).
\end{equation}
In order to prove \eqref{eqn:primitivity-kappa-class}, we may suppose that $X$ is a finite CW complex. It is an exercise in linear algebra to construct an isomorphism 
\[
(V \cup_A \underline{\bR^n}_F) \oplus (\underline{\bR^n}_E \cup_A W) \cong (V \cup_A W) \oplus (\underline{\bR^n}_E\cup_A \underline{\bR^n}_F)
\]
of clutched bundles on $E \cup_A F$ (hint: picking bundle maps $f: \underline{\bR^n} \to V$ and $g:\underline{\bR^n} \to W$ which are the identity on the boundary is a good start). Hence the classifying map $\gamma_{V \cup_A W}: E \cup_A F \to BO$ of $V \cup_A W$ can be factored as 
\[
E \cup_A F \to E /\partial \vee F / \partial \stackrel{\gamma_E \vee \gamma_F}{\to} BO \vee BO \stackrel{\mathrm{fold}}{\to} BO.  
\]
Therefore $c(V \cup_A E) = q_E^* c(V) + q_F^* c(W)$, where 
\[
q_E^*: H^* (E,\partial E) \cong H^* (E \cup_A F,\partial E \cup F) \to H^* (E \cup_A F,\partial (E \cup_A F))
\]
and $q_F^*$ is similarly defined. The formula \eqref{eqn:primitivity-kappa-class} follows from the observation that 
\[
H^* (E,\partial E) \stackrel{q_E^*}{\to} H^* (E \cup_A F,\partial (E \cup_A F)) \stackrel{\pi^{E \cup_A F}_!}{\to} H^{*-d}(X)
\]
is nothing else than $\pi^E_!$ (and the similar fact for $F$ in place of $E$). 
\end{proof}

The next result is fairly simple-minded; compare \cite[Lemma 7.16]{GRW14} for a more elaborate, but closely related result about diffeomorphism groups. 

\begin{prop}\label{prop:vanishingobstruction}
Let $M$ be a compact oriented $d$-manifold and fix a stable trivialization of $TM|_{\partial M}$. Assume that the inclusion map $\partial M \to M$ is $(k-1)$-connected and that all rational relative Pontrjagin classes of $TM$ up to degree $k$ are zero. Then 
\[
\kappa_c=0 \in H^{|c|-d} (B \blockdiff_\partial^+ (M);\bQ)
\]
whenever $c \in H^* (BO;\bQ)$ lies in the ideal generated by $\bigoplus_{0<j\leq k} H^j (BO;\bQ)$. 
\end{prop}

\begin{proof}
The cohomological Leray--Serre spectral sequence for the universal block bundle pair $(E,\partial E)$ over $(B\blockdiff_\partial^+(M),*)$ reads as follows:
\[
E_2^{p,q}= H^p (B\blockdiff_\partial^+(M),*;H^q (M,\partial M;\bQ))\Rightarrow H^{p+q} (E, M \cup \partial E;\bQ). 
\]
Hence the connectivity assumption on $M$ shows that $H^* (E,M \cup \partial E;\bQ)=0$ whenever $* \leq k$. 
By the long exact sequence for the triple $(E,M\cup\partial E,\partial E)$, the map $H^* (E,\partial E;\bQ) \to H^* (M \cup \partial E,\partial E;\bQ) \cong H^* (M,\partial M;\bQ)$ is injective when $* \leq k$. The latter map sends the relative Pontrjagin classes of $T_v E$ to those of $TM$; so $c(T_v E)=0 \in H^* (E,\partial E;\bQ)$; a fortiori, $\kappa_c =0$.
\end{proof}

\subsubsection*{The family signature theorem}

The classical family signature theorem for smooth fibre bundles (which uses families of elliptic operators in its proof) holds more generally for block bundles, as shown by Randal--Williams in \cite[Theorem 2.6]{ORWsignature}. It has two cases, the odd-dimensional and the even-dimensional case. The odd case which reads as follows.

\begin{thm}\label{thm:vanishing}
Let $M^d$ be an odd-dimensional oriented manifold, and assume for simplicity that $TM|_{\partial M}$ is stably trivial. Then
\[
\kappa_{L_m} =0 \in H^{4k-d} (B \blockdiff^+_\partial M;\bQ)
\]
for each $m\in \bN$ (and hence the same is true for diffeomorphisms). 
\end{thm}

\begin{proof}[References]
By Lemma \ref{lem:additivity-kappaclass}, it is enough to show the theorem for closed $M$, and this is done in \cite[Theorem 2.6]{ORWsignature} for block diffeomorphisms, and in \cite{Ebert13} for diffeomorphisms. 
\end{proof}

Let us state the even-dimensional case for sake of completeness and only for the case where $M$ is closed or $\partial M$ is a sphere (we use this in the proofs of Proposition \ref{prop:vanishing-degree1} and Theorem \ref{thm:blockdiffsversusmappingspace} which we believe to be of independent interest). 
The point is that if $\dim (M)=2n$, the action of $\hAut_\partial^+ (M)$ on $H_n (M;\bQ)$ preserves the (nondegenerate) intersection form $I_M$. This fact produces a map $h:  B\hAut_\partial^+ (M) \to B \Aut(H_n (M;\bQ);I_M)$; the latter is the classifying space of a symplectic or an orthogonal group, depending on the parity of $n$. Randal--Williams defines classes $\sigma_i \in H^i( B \Aut(H_n (M;\bQ);I_M);\bQ)$, which live in degrees $i\equiv 2 \pmod 4$ if $n$ is odd and $i \equiv 0 \pmod 4$ if $n$ is even, and shows in \cite[Theorem 2.6]{ORWsignature} that 
\begin{equation}\label{eqn:evenfamilysignautre}
\kappa_{L_m} = h^* \sigma_{4m-2n} \in H^{4m-2n}(B  \blockdiff_\partial^+ (M);\bQ).
\end{equation}

\begin{prop}\label{prop:vanishing-degree1}
Let $M$ be a compact oriented $d$-manifold, assume that $TM|_{\partial M}$ is stably trivial, and suppose that all rational Pontrjagin classes of $M$ are trivial. Then for each $c \in H^k (BO;\bQ)$, $k>d$, the homomorphism 
\[
\pi_{k-d}(B\blockdiff_\partial^+ (M)) \to \bQ
\]
given by 
\[
[f] \mapsto \scpr{f^* \kappa_c;[S^{k-d}]}
\]
is the zero map. In particular, if $k=d+1$, then $\kappa_c =0 \in H^1 (B \blockdiff_\partial^+ (M);\bQ)$.
\end{prop}
\begin{proof}
For diffeomorphisms, this is a well-known fact, see e.g. \cite[Proposition 13]{KreckOdd} or \cite[Proposition 1.9]{HankeSteimleSchick}. To see that the proof also applies to block diffeomorphisms, we review the argument. 

Firstly, the double $M\cup_{\partial M} M$ has trivial rational Pontrjagin classes, by an argument given in the proof of Theorem F of \cite{ERWpsc3}; hence by Lemma \ref{lem:additivity-kappaclass} (extend diffeomorphisms trivially over the second copy of $M$) it suffices to give the argument for closed $M$. 

Let $k=4m$ and let $\pi: E \to S^{4m-d}$ be an oriented block bundle with fibre $M$, classified by $f: S^{4m-d} \to B \blockdiff^+ (M)$. The restriction of the stable vertical tangent bundle $T_v E$ to the fibre $M$ over $*$ is stably isomorphic to $TM$. A brief inspection of the Leray--Serre spectral sequence of $\pi$ proves that $c(T_v E)=0$ if $c$ can be written as a product of classes in positive degrees. Hence we only have to prove that $\kappa_{L_m} (E)=0$. This follows from Theorem \ref{thm:vanishing} for odd $d$, and from \eqref{eqn:evenfamilysignautre} for even $d$ (in the latter case, we only need to consider bundles over spheres of even dimension which are $1$-connected, so that $h \circ f$ is nullhomotopic).

The last sentence follows by the Hurewicz theorem.
\end{proof}

\subsection{Borel classes}\label{subsec:borelclasses}

Let $K(\bZ)$ denote the algebraic $K$-theory spectrum of the integers; recall that $\pi_0 (K(\bZ))\cong \bZ$ and that $\Omega^\infty_0 K(\bZ)\simeq  B \GL_\infty(\bZ)^+$. A celebrated result of Borel describes the rational cohomology of the latter space. 

\begin{thm}[Borel]\label{thm:Borel}\mbox{}
\begin{enumerate}
\item The rational cohomology $H^* (\Omega_0^\infty K(\bZ);\bQ)=H^* (B \GL_\infty(\bZ);\bQ)$ is an exterior algebra with generators $\beta_{4k+1} \in H^{4k+1}(\Omega^\infty_0 K(\bZ);\bQ)$, $k \geq 1$. 
The classes $\beta_{4k+1}$ are primitive. 
\item The restriction maps 
\[
H^p (B\GL_\infty(\bZ);\bQ) \to H^p (B \GL_g (\bZ);\bQ)\to H^p (B \SL_g (\bZ);\bQ)
\]
are isomorphisms provided that $g \geq 2p+2$.
\item The group homomorphism $\kappa: \GL_g (\bZ) \to \GL_g (\bZ)$ given by the inverse transpose $\kappa(x):= (x^\top)^{-1}$ has the following effect on these classes:
\[
(B \kappa)^* \beta_{4k+1}=-\beta_{4k+1}. 
\]
\end{enumerate}
\end{thm}

\begin{proof}[References]
(1) is of course a famous theorem of Borel \cite{Borel} (he treats real cohomology which makes little difference as the cohomology spaces are all finite-dimensional by \cite{QuillenKfg}). (2) A range in which the map from the stable cohomology to the unstable cohomology is an isomorphism is also determined in Borel's paper; the range as stated follows from Maazen--Van der Kallen's homological stability theorem \cite{Maazen} \cite[Theorem 4.11]{VanderKallen}, using that the Bass stable rank of the integers is $2$; see \cite[\S 4.1.11]{HahnOMeara}. 
(3) can also easily be deduced from Borel's work. Since we do not know a reference, we shall indicate the proof here. It suffices to prove the statement for $\SL_g (\bZ)$ instead of $\GL_g (\bZ)$, since the covering map $B \SL_g (\bZ) \to B\GL_g (\bZ)$ induces an injection in rational cohomology. We need to recall how the Borel classes are constructed. Let $X$ be the symmetric space $\SL_g (\bR)/SO(g)$ and let $\cA^* (X)^{\SL_g (\bR)}$ be the cochain complex of invariant differential forms, which has trivial differential as each $\SL_g (\bR)$-invariant differential form on $X$ is closed, by a general fact about symmetric spaces. On the other hand, $X$ is contractible and the $\SL_g (\bZ)$-action is proper, so that there is a natural isomorphism
\[
 H^* (B \SL_g (\bZ); \bR) \cong H^* (\cA^* (X)^{\SL_g (\bZ)}). 
\]
On $X$, there is the Cartan involution $\tau: X \to X$, $\tau(x SO(g)) := (x^\top)^{-1} SO(g)$. It is easily verified that $\tau$ induces an involution on $\cA^* (X)^{\SL_g (\bR)}$, and that the diagram
\[
 \xymatrix{
\cA^* (X)^{\SL_g (\bR)} \ar[d]^{\tau^*} \ar[r] & H^* (B \SL_g (\bZ);\bR) \ar[d]^{B \kappa^*} \\
 \cA^* (X)^{\SL_g (\bR)} \ar[r] & H^* (B \SL_g (\bZ);\bR)
  }
\]
commutes. By definition, the Borel classes come from certain invariant forms on $X$. It is therefore enough to show that $\tau^*: \cA^p (X)^{\SL_g (\bR)}  \to \cA^p (X)^{\SL_g (\bR)}$ is multiplication by $(-1)^p$. 
The involution $\tau$ fixes the basepoint $o:=SO(g) \in X$, and since invariant forms on $X$ are determined by their values at $o$, it is enough to check that $T_o \tau = -1$. But this is easily verified by a direct calculation. 
\end{proof}

Now let $X$ be a finite CW-complex. For each $p$, the action on $H_p (X;\bZ)$ provides a map 
\[
B \rho_p :B \hAut(X)\to B\GL(H_p (X;\bZ)). 
\]
For each finitely generated abelian group $A$ with torsion subgroup $TA$, there is a map 
\[
\iota: B \GL(A) \to B\GL(A/TA) \to \Omega^\infty K(\bZ)
\]
well defined up to homotopy; it hits the component of $\rank (A) \in \bZ = \pi_0(K(\bZ))$. 
Composing $\iota$ and $B \rho_p$ gives classes 
\[
\beta_{4k+1}^p:= (\iota \circ B \rho_p)^* \beta_{4k+1} \in H^{4k+1}(B \hAut(X);\bQ).
\]
Using the infinite loop space structure on $\Omega^\infty K(\bZ)$, we can form the alternating sum
\[
\chi:= \sum_{p\geq 0} (-1)^p \iota \circ B \rho_p: B \hAut(X) \to \Omega_{\chi(X)}^\infty K(\bZ) \simeq \Omega^\infty_0 K(\bZ),
\]
the \emph{algebraic $K$-theory Euler characteristic}. Because the Borel classes are primitive, the $H$-space inversion on $\Omega^\infty K(\bZ)$ acts on them by a minus sign, and so the relation 
\[
\chi^* \beta_{4k+1}= \sum_{p \geq 0} (-1)^p \beta_{4k+1}^p \in H^{4k+1}(B \hAut(X);\bQ)
\]
holds. 
\begin{thm}[Dwyer--Weiss--Williams]\label{dww:vaninsihg}
Let $M$ be a smooth compact manifold, possibly with boundary. Then 
\[
\chi^* \beta_{4k+1} = 0 \in H^{4k+1} (B \Diff(M);\bQ).
\]
\end{thm}

Let us remark that the analogue of Theorem \ref{dww:vaninsihg} for block diffeomorphism groups is \emph{false}; in fact our computation of $H^* (B \blockdiff_\partial (U_{g,1}^n);\bQ)$ certifies its failure.

\begin{proof}[References]
The Dwyer--Weiss--Williams index theorem \cite[Corollary 8.12]{DWW} shows that the map $\chi$ factors through the Becker--Gottlieb transfer $B \Diff(M) \to Q S^0$, so that the result simply follows from Serre's finiteness theorem. See also \S 2 of \cite{EbertDisc} for a more detailed summary.
\end{proof}
We now use Poincar\'e duality to deduce a sharper vanishing theorem from Theorem \ref{dww:vaninsihg}. 
\begin{lem}\label{lem:borel-class-poincaredual}
Let $M^d$ be a connected smooth oriented manifold with boundary, and suppose that $\partial M = S^{d-1}$. Then 
\[
\beta_{4k+1}^p = - \beta_{4k+1}^{d-p}\in H^{4k+1}(B \Diff^+_\partial (M);\bQ)
\]
for all $p$. 
\end{lem}

\begin{proof}
We can consider $\hat{M}:= M \cup_{\partial M} D^d$ instead, without changing the Borel classes. Poincar\'e duality, the universal coefficient theorem and Theorem \ref{thm:Borel} (3) proves the claim.
\end{proof}

\begin{prop}\label{prop:vanishingdww}
Let $M$ be as in Lemma \ref{lem:borel-class-poincaredual}. If $d=2n$, then 
\[
\beta_{4k+1}^n =0 \in H^{4k+1}(B \Diff_\partial^+ (M);\bQ).
\] 
If $d=2n+1$, then 
\[
\sum_{p=0}^n (-1)^p \beta_{4k+1}^p =0 \in H^{4k+1}(B\Diff_\partial^+ (M);\bQ).
\]
\end{prop}

\begin{proof}
For $d=2n$, compute
\[
0 \stackrel{\eqref{dww:vaninsihg}}{=}  \sum_{p=0}^{n-1} \Bigl( (-1)^p \beta_{4k+1}^p + (-1)^{2n-p} \beta_{4k+1}^{2n-p}\Bigr) + (-1)^n \beta_{4k+1}^n= 
\]
\[
\stackrel{\eqref{lem:borel-class-poincaredual}}{=} \sum_{p=0}^{n-1} (-1)^p \Bigl(  \beta_{4k+1}^p -  \beta_{4k+1}^{p}\Bigr) +  (-1)^n \beta_{4k+1}^n =  (-1)^n \beta_{4k+1}^n. 
\]
For $d=2n+1$, compute 
\[
0 \stackrel{\eqref{dww:vaninsihg}}{=}  \sum_{p=0}^{n} \Bigl( (-1)^p \beta_{4k+1}^p + (-1)^{2n+1-p} \beta_{4k+1}^{2n+1-p}\Bigr) = 
\]
\[
\stackrel{\eqref{lem:borel-class-poincaredual}}{=} \sum_{p=0}^{n} \Bigl( (-1)^p \beta_{4k+1}^p + (-1)^{2n+1-p+1} \beta_{4k+1}^{p}\Bigr) = 
\]
\[
2 \sum_{p=0}^{n} (-1)^p \beta_{4k+1}^p.
\]
\end{proof}

\begin{rem}
The even-dimensional case of Proposition \ref{prop:vanishingdww} can be shown directly from Borel's work, without recourse to \cite{DWW}. The point is that by Poincar\'e duality, $\rho_n$ factors through the symplectic group or through an orthogonal group of some signature, depending on the parity of $n$, and Borel also computed the stable rational cohomologies of such groups: in a stable range, they are concentrated in even degrees.
\end{rem}

\section{Rational cohomology of block diffeomorphism spaces: general theory}\label{sec:cohoblockdiffgeneral}

We shall approach $H^* (B \blockdiff_\partial (U_{g,1}^n);\bQ)$ by the surgery-theoretic approach to the topology of diffeomorphism groups which is due to Quinn \cite{Quinn}; a detailed exposition is available in \cite{Nicas}. For our purposes, work of Berglund and Madsen \cite[\S 4]{BerglundMadsen} enables us to treat all the surgery theory as a black box. Our aim in this section is to reformulate the results of \cite[\S 4]{BerglundMadsen} in a way which is readily applicable to $U_{g,1}^n$. The main result of this section is Theorem \ref{thm:blockdiffsversusmappingspace} which gives essentially a formula for $H^* (B \blockdiff_\partial (M);\bQ)$ when $M$ is stably parallelizable, $1$-connected, of dimension $d\geq 5$ and has boundary $S^{d-1}$. The theorem is formulated in terms of the rationalization of spaces, which we recapitulate in \S \ref{subsec:rationalhomotopy}, and in terms of the cohomology of mapping spaces of the form $\map_\partial (M;BO)$, which we describe in detail in \S \ref{subsec:cohmappingspace}. In \S \ref{subsec:blcockdiffberglundmadsensrv}, we review the results of \cite[\S 4]{BerglundMadsen} and polish them slightly, leading to Proposition \ref{prop:rationalized-blockdiffs}. 
The final subsection \ref{subsec:finalformula-for-blockdiffs} combines all these ingredients to Theorem \ref{thm:blockdiffsversusmappingspace}.

\subsection{Some words about rational homotopy theory}\label{subsec:rationalhomotopy}

\subsubsection*{Generalities}
Let us recall some notions and results from rational homotopy theory. For us, a \emph{space} will be a Kan complex. The category of spaces is denoted $\sSet$, and the category of pointed spaces by $\sSet_*$. Recall that the category $\sSet$ is enriched over itself. 

We say that a space $X$ is \emph{finite} if the geometric realization $|X|$ is homotopy equivalent to a finite CW complex. 

A map $f:X \to Y$ is a \emph{$H\bQ$-equivalence} if the induced map $f_*: H_* (X,\bQ) \to H_* (Y,\bQ)$ is an isomorphism. When all path components of $X$ and $Y$ are nilpotent (e.g. $1$-connected or simple), this requirement is equivalent to saying that $f_*: \pi_0 (X) \to \pi_0 (Y)$ is bijective and that $f_*: \pi_k (X,x) \otimes \bQ \to \pi_k (Y,f(x)) \otimes \bQ$ is an isomorphism for all $k\geq 1$ and all $x \in X$ (for an arbitrary nilpotent group $G$, we use the notation $G \otimes \bQ$ for the $\bQ$-localization of $G$, see \cite[\S I]{HilMisRoi}). In this situation, we call an $H \bQ$-equivalence also a \emph{rational homotopy equivalence}. 

A space $Z$ is \emph{$\bQ$-local} if for each $H \bQ$-equivalence $f:X \to Y$ and all choices of basepoints, the map 
\[
\_ \circ f: \map_* (Y;Z) \to \map_* (X;Z)
\]
is a weak equivalence (equivalently the map induced by $\_ \circ f$ on $\pi_0$ is bijective for all such $f$). If $Z$ is nilpotent, this is equivalent to saying that all homotopy groups $\pi_k (Z,z)$ for $k  \geq 2$ are $\bQ$-vector spaces and that the fundamental groups $\pi_1 (Z,z)$ are $\bQ$-local nilpotent groups in the sense of \cite[p.4]{HilMisRoi}. 

A map $f:X \to Y$ is a \emph{$\bQ$-localization} if $f$ is an $H\bQ$-equivalence and $Y$ is $\bQ$-local. Such a map, if it exists, is unique up to weak equivalence. It was proven by Sullivan \cite{Sullivan} that each nilpotent space admits a $\bQ$-localization. 
For our purposes, it will be convenient to have a strictly functorial $\bQ$-localization. 

\begin{thm}\label{thm:rationalhomotopyinput}
There is an enriched functor $(\_)_\bQ: \sSet \to \sSet$, together with an enriched natural transformation $\eta: \id \to (\_)_\bQ$, such that for each $X \in \sSet$, the map $\eta_X:X \to X_\bQ$ is a $\bQ$-localization.
\end{thm}

We refer to \cite[\S 3]{RiehlCatHomThy} for the vocabulary of enriched category theory. The statement that $(\_)_\bQ$ is enriched means that it comes along with natural maps 
\[
i_{X,Y}: \map (X;Y) \to \map(X_\bQ;Y_\bQ), 
\]
and the enriched natural transformation $\eta$ is given by maps $\eta_X:X \to X_\bQ$ such that the composition 
\[
\map (X;Y) \stackrel{i_{X,Y}}{\to} \map(X_\bQ;Y_\bQ) \stackrel{\_ \circ \eta_X}{\to} \map (X;Y_\bQ)
\]
agrees with the map $\eta_Y \circ \_$. There is an induced enriched functor $\sSet_* \to \sSet_*$ of pointed spaces. 

\begin{proof}[References for Theorem \ref{thm:rationalhomotopyinput}]
This is due to Bousfield; first in \cite{Bousfield} without the word ``enriched'', the enrichment is constructed in \cite[\S 5]{BousfieldK}. 
\end{proof}

\subsubsection*{Rationalization of mapping spaces}

\begin{lem}\label{rationalization-mappingspace}
Let $X$ and $Z$ be connected pointed spaces and let $g: X \to Z$ be a pointed map. We write $\map_* (X;Z)^g \subset \map_*(X;Z)$ for the component containing $g$; and we write $g_\bQ: X_\bQ \to Z_\bQ$ for the rationalization of $g$, i.e. $g_\bQ:= i_{X,Z}(g)$. The map $i_{X,Z}$ restricts to 
\[
i_{X,Z}^g: \map_* (X;Z)^g \to \map_* (X_\bQ;Z_\bQ)^{g_\bQ}.
\]
Then if $X$ is finite, the spaces $\map_* (X;Z)^g$ and $\map_* (X_\bQ;Z_\bQ)^{g_\bQ}$ are nilpotent, $\map_* (X_\bQ;Z_\bQ)^{g_\bQ}$ is $\bQ$-local and $i_{X,Z}^g$ is a rational homotopy equivalence.
\end{lem}

\begin{proof}
The composition
\[
\map_* (X;Z)^g \stackrel{i_{X,Z}^g}{\to} \map_* (X_\bQ; Z_\bQ)^{g_\bQ}\stackrel{\_ \circ \eta_X}{\to} \map_* (X;Z_\bQ)^{ g_\bQ \circ \eta_X}
\]
is equal to $\eta_Z \circ \_$. The map $\_ \circ \eta_X$ is a weak homotopy equivalence as $\eta_X$ is an $H\bQ$-equivalence and $Z_\bQ$ is $\bQ$-local. 

The spaces $\map_* (X;Z)^g$ and $\map_* (X;Z_\bQ)^{ g_\bQ \circ \eta_X}$ are nilpotent, by Theorem 2.5 on p.64 of \cite{HilMisRoi}, and therefore so is $\map_* (X_\bQ;Z_\bQ)^{g_\bQ}$. The map 
\[
\eta_Z \circ\_ : \map_* (X;Z)^g \to \map_* (X;Z_\bQ)^{\eta_Z \circ g}
\]
is a $\bQ$-localization by Theorem 3.11 on p. 77 of loc.cit., which entails the other claims.
\end{proof}

There is a similar fact for homotopy automorphisms. The map $i_{X,X}$ restricts to a monoid homomorphism
\[
j_X: \hAut_* (X) \to \hAut_* (X_\bQ),
\]
and 
\[
j_{X,A}: \hAut_A(X) \to \hAut_{A_\bQ}(X_\bQ)
\]
when $X$ is a pointed space or $(X,A)$ is a space pair. We write $\hAut_A(X)^\id$ for the unit component of $\hAut_A(X)$, and let moreover 
\[
\hAut_{A_\bQ} (X_\bQ)_\bZ \subset \hAut_{A_\bQ} (X_\bQ)
\]
be the union of all path components which are hit by $j_{X,A}$; this are clearly grouplike submonoids. 

\begin{lem}\label{lem:rationalization-homotopyautomorphisms}
Suppose that $(X,A)$ is a pair of finite spaces, with $X$ nilpotent and $A$ connected. Then the natural map 
\[
j_{X,A}: \hAut_A (X)^{\id} \to \hAut_{A_\bQ} (X_\bQ)^{\id}
\]
is $\bQ$-localization. 
\end{lem}

\begin{proof}
The case $A=*$ is a special case of Lemma \ref{rationalization-mappingspace}. 
In the general case, consider the commutative diagram
\[
\xymatrix{
\hAut_{\ast} (X)^{\id} \ar[r]^{j_X}\ar[d] & \hAut_{\ast} (X_\bQ)^{\id}\ar[d] \\
\map_{\ast} (A;X)^{\inc} \ar[r]^{i_{A,X}^\inc} & \map_{\ast} (A_\bQ;X_\bQ)^{\inc},
}
\]
where the symbol $\inc$ is used to denote both inclusions $A \to X$ and $A_\bQ \to X_\bQ$.
The horizontal maps are rational homotopy equivalences by Lemma \ref{rationalization-mappingspace}. It follows that the induced map on vertical homotopy fibres induces an isomorphism on all rational homotopy groups in degrees $\geq 2$. The same is true on fundamental groups, using \cite[Proposition 1.10]{HilMisRoi}. On the other hand, restricting the map on vertical fibres to the path component of the identity gives the map $j_{A,X}$. 
\end{proof}

Having understood these matters, we usually abuse notation and write
\[
\hAut_A (X_\bQ):= \hAut_{A_\bQ} (X_\bQ)
\]
and use the notation $\hAut_A(X_\bQ)_\bZ$, $\hAut_A(X_\bQ)^\id$ similarly.

\subsection{Block diffeomorphisms versus tangential homotopy automorphisms}\label{subsec:blcockdiffberglundmadsensrv}

We need to introduce some more notation. 

\begin{notation}
\begin{enumerate}
\item We write $\Diff_\partial (M)^{\sim \id} \subset \Diff_\partial (M)$ for the subgroup of diffeomorphisms which are homotopic to the identity (relative boundary), and define $\blockdiff_\partial (M)^{\sim \id} \subset \blockdiff_\partial (M)$ analogously.
\item Assume that $C \subset A\subset X$ are subcomplexes and that $V \to X$ is a vector bundle. In that situation, we denote by 
\[
\hAut_A^C (V)^{\sim \id} \subset \hAut_A^C (V)
\]
the preimage of $\hAut_A (X)^\id$ under the forgetful map $\hAut_A^C (V) \to \hAut_A (X)$; in other words the space of pairs $(f,\hat{f})$ with $f \sim \id $ (relative $A$). We define $\hAut_A^C (V)^{\st,\sim \id} \subset \hAut_A^C (V)^\st$ similarly. 
\item We let 
\[
\hAut_\partial (M)^{\cong} \subset\hAut_\partial (M)
\]
be the union of the components which are hit by the forgetful map $\blockdiff_\partial (M) \to \hAut_\partial (M)$ (or equivalently the forgetful map $\Diff_\partial (M) \to \hAut_\partial (M)$, as $\Diff_\partial (M) \to \blockdiff_\partial (M)$ is by definition $0$-connected). 
\item For a subcomplex $C \subset \partial M$, we define
\[
\hAut_\partial^C (TM)^{\st,\cong} \subset \hAut_\partial^C (TM)^\st
\]
as the preimage of $\hAut_\partial (M)^{\cong}$ under the forgetful map.
\end{enumerate}
\end{notation}

The derivative map \eqref{eqn:detrivativemap} is, by virtue of its definition, a map over $\hAut_\partial (M)$, and hence it restricts to a map 
\[
D: \blockdiff_\partial (M)^{\sim \id} \to \hAut_\partial^\partial (TM)^{\st,\sim \id}.
\]
If $* \in \partial M$ is a basepoint, we can furthermore compose the derivative map with the forgetful map $\hAut_\partial^\partial (TM)^{\st} \to \hAut_\partial^* (TM)^{\st}$. Hence by restriction and taking classifying spaces, we obtain a map 
\begin{equation}\label{eqn:derivative-map-finalform}
BD: B \blockdiff_\partial (M)^{\sim \id} \to B \hAut_\partial^* (TM)^{\st,\sim \id}. 
\end{equation}

All the surgery theory we need enters the proof of the following result. 

\begin{thm}\cite[Theorem 1.1]{BerglundMadsen}\label{thm:berglundmadsen}
Assume that $M^d $ is $1$-connected, $\partial M=S^{d-1}$ and that $d \geq 5$. Then the spaces $B \blockdiff_\partial (M)^{\sim \id}$ and $B \hAut_\partial^* (TM)^{\st,\sim\id}$ are nilpotent, and the map $BD$ from \eqref{eqn:derivative-map-finalform} is a rational homotopy equivalence. 
\end{thm}

In the rest of this subsection, we derive a version of Theorem \ref{thm:berglundmadsen} (namely Proposition \ref{prop:rationalized-blockdiffs} below) which involves the full block diffeomorphism group and not just $\blockdiff_\partial (M)^{\sim\id}$ and which is directly applicable to the manifolds $U_{g,1}^n$. 
The derivative map yields
\begin{equation}\label{eqn:derivative-map-finalform2}
BD: B \blockdiff_\partial (M) \to B \hAut_\partial^\partial (TM)^{\st,\cong}\to B \hAut_\partial^* (TM)^{\st,\cong}. 
\end{equation}

\begin{cor}\label{cor:toberglundmadsen}
Let $M$ be as in Theorem \ref{thm:berglundmadsen}. Then the composition \ref{eqn:derivative-map-finalform2} is an $H \bQ$-equivalence.
\end{cor}

\begin{proof}
There is a diagram
\[
\xymatrix{
B \blockdiff_\partial (M)^{\sim\id} \ar[r]^-{BD} \ar[d] & B \hAut_\partial^\ast (TM)^{\st,\sim\id}\ar[d]\\
B \blockdiff_\partial (M) \ar[r]^-{BD} \ar[d] & B \hAut_\partial^\ast (TM)^{\st,\cong}\ar[d]\\
B\pi_0 (\hAut_\partial (M)^{\cong}) \ar@{=}[r] & B\pi_0 (\hAut_\partial (M)^{\cong})
}
\]
whose columns are fibre sequences, so the claim follows immediately from Theorem \ref{thm:berglundmadsen} and an application of the Leray--Serre spectral sequence.
\end{proof}

\begin{cor}\label{cor:toberglundmadsenfinal}
Assume that $M$ is a $1$-connected and stably parallelizable manifold of dimension $d \geq 5$, and that $\partial M = S^{d-1}$. Then the composition 
\[
B \blockdiff_\partial (M) \to \hq{\map_\partial (M;BO)^0}{\hAut_\partial(M)^{\cong}} \to \hq{\map_* (M;BO)^0}{\hAut_\partial(M)^{\cong}}
\]
of the derivative map with the forgetful map is an $H \bQ$-equivalence.
\end{cor}

\begin{proof}
Combine Corollary \ref{cor:toberglundmadsen} with Lemma \ref{lem:Bofbundlehomotopyautomorphisms}. 
\end{proof}

Using the rationalization functor, we now give the variant of Theorem \ref{thm:berglundmadsen} which we shall eventually use. We let 
\[
\hAut_\partial (M_\bQ)^{\cong} \subset \hAut_\partial (M_\bQ)_\bZ
\]
the union of components which are hit by the rationalization map 
\[
\hAut_\partial (M)^{\cong} \to \hAut_\partial (M_\bQ)_\bZ.
\]

\begin{prop}\label{prop:rationalized-blockdiffs}
Let $M$ be a $1$-connected manifold of dimension $d \geq 5$, assume that $M$ is stably parallelizable and that $\partial M = S^{d-1}$.  Then the composition 
\[
B \blockdiff_\partial (M) \stackrel{BD}{\to} \hq{\map_\partial (M;BO)^0}{\hAut_\partial(M)^{\cong}} \to \hq{\map_\partial (M_\bQ;BO_\bQ)^0}{\hAut_\partial(M_\bQ)^{\cong}}  \to \hq{\map_* (M_\bQ;BO_\bQ)^0}{\hAut_\partial(M_\bQ)^{\cong}}
\]
of the derivative map, the rationalization and the forgetful map is an $H \bQ$-equivalence.
\end{prop}

\begin{proof}
The composition in question agrees with 
\[
B \blockdiff_\partial (M) \to \hq{\map_* (M;BO)^0}{\hAut_\partial(M)^{\cong}} \to \hq{\map_* (M_\bQ;BO_\bQ)^0}{\hAut_\partial(M_\bQ)^{\cong}},
\]
so Corollary \ref{cor:toberglundmadsenfinal} leaves us with the task of proving that the second map is an $H \bQ$-equivalence. This second map can clearly be factored as 
\[
\hq{\map_* (M;BO)^0 }{ \hAut_\partial(M)^{\cong}} \to \hq{\map_* (M_\bQ;BO_\bQ)^0}{\hAut_\partial(M)^{\cong}} \to \hq{\map_* (M_\bQ;BO_\bQ)^0}{\hAut_\partial(M_\bQ)^{\cong}}. 
\]
The first of those maps is a $H \bQ$-equivalence, by Lemma \ref{rationalization-mappingspace} and a straightfoward spectral sequence argument.
To prove that the second map is also an $H \bQ$-equivalence, observe that there is a homotopy cartesian diagram
\[
\xymatrix{
\hq{\map_* (M_\bQ;BO_\bQ)^0}{\hAut_\partial(M)^{\cong}}\ar[r] \ar[d] & B \hAut_\partial(M)^{\cong}\ar[d]\\
\hq{\map_* (M_\bQ;BO_\bQ)^0}{\hAut_\partial(M_\bQ)^{\cong}}  \ar[r] & B\hAut_\partial(M_\bQ)^{\cong}. 
}
\]
Therefore, it is sufficient to prove that the homotopy fibre $F$ of the right vertical map is rationally acyclic. 

The definition of $B\hAut_\partial(M_\bQ)^{\cong}$ shows that $F$ is connected, and Lemma \ref{lem:rationalization-homotopyautomorphisms} shows that $\pi_k(F) \otimes \bQ=0$ for all $k \geq 2$. Hence the universal cover $\tilde{F}$ is rationally acyclic, and the Leray--Serre spectral sequence of the fibre sequence $\tilde{F} \to F \to B \pi_1 (F)$ shows that 
\begin{equation}\label{eqn:lerayserreunivcover}
H_k (F;\bQ)\cong H_k (B \pi_1 (F);\bQ). 
\end{equation}
The exact sequence 
\[
\pi_1 (\hAut_\partial(M)^{\cong}) \to \pi_1 (\hAut_\partial(M_\bQ)^{\cong} )\to \pi_1 (F) \to \pi_0 (\hAut_\partial(M)^{\cong}) \to \pi_0 (\hAut_\partial(M_\bQ)^{\cong})  
\]
yields a short exact sequence 
\begin{equation}\label{eqn:groupextension}
0 \to T \to \pi_1 (F) \to Q \to 1
\end{equation}
of groups, where 
\[
T:= \coker \Bigl(\pi_1 (\hAut_\partial(M)^{\cong}) \to \pi_1 (\hAut_\partial(M_\bQ)^{\cong} )\Bigr)
\]
is an abelian torsion group by Lemma \ref{lem:rationalization-homotopyautomorphisms} and 
\[
Q:= \ker \Bigr(\pi_0 (\hAut_\partial (M)^{\cong}) \to \pi_0 (\hAut_\partial (M_\bQ)^{\cong})\Bigr)
\]
is finite by \cite[Theorem 1.1]{EspicSaleh}. The latter is a relative version of a theorem of Sullivan \cite[Theorem 10.2]{Sullivan1977} which asserts that $\pi_0 (\hAut_*(X)) \to \pi_0 (\hAut_* (X_\bQ))$ has finite kernel whenever $X$ is a finite and $1$-connected CW complex. As $Q$ is finite, we get
\[
H_k (\pi_1 (F);\bQ)\cong H_k (T;\bQ)_Q
\]
from the Lyndon--Hochschild--Serre spectral sequence of the group extension \eqref{eqn:groupextension}.
Since $H_k (T;\bQ)=0$ for $k \geq 1$, being the colimit of $H_k (H;\bQ)$, where $H$ runs through the finitely generated (and hence finite, as $T$ is torsion) subgroups of $T$, $B \pi_1 (F)$ is rationally acyclic. By \eqref{eqn:lerayserreunivcover}, $F$ is rationally acyclic as claimed. 
\end{proof}

\subsection{Cohomology of mapping spaces}\label{subsec:cohmappingspace}

Proposition \ref{prop:rationalized-blockdiffs} shows that we need to understand the cohomology of mapping spaces $\map_* (M;BO_\bQ)^0$. It turns out to be more convenient for our purposes to consider $\map_\partial (M;BO_\bQ)^0$ instead. Since $\map_\partial (M;BO_\bQ)^0 = \map_* (M/\partial M;BO_\bQ)^0$, it suffices to consider pointed mapping spaces. 

In what follows, all homology and cohomology groups are taken with coefficients in $\bQ$. Though this is not needed for large parts of the section, it is all we shall need later on. To ease notation, we often write $\alpha \beta:= \alpha \cup \beta$ for the cup product of two cohomology classes. 

\subsubsection*{The slant product}
Let us first recall from \cite[Chapter 6.1]{Spanier} the slant product
\[
H^n (Y \times X) \otimes H_k (X) \to H^{n-k} (Y), \; (\xi,x) \mapsto \xi / x,
\]
which is related to the cohomology cross product and the Kronecker product by the formula
\begin{equation}\label{eqn:slant-identity}
(\eta \times \zeta) / x = \eta \scpr{\zeta,x}
\end{equation}
for $\eta \in H^* (Y)$, $\zeta \in H^* (X)$ and $x \in H_* (X)$. Assume that $X$ has finite type over $\bQ$ (i.e. each $H_k (X)$ is finite-dimensional), pick a homogeneous basis $(b_i)_i$ of $H_* (X)$ and let $(\beta_i)_i$ be the dual basis of $H^* (X)$. An arbitrary $\gamma \in H^* (Y \times X)$ can be written in the form $\gamma = \sum_i \gamma_i \times \beta_i$ by the K\"unneth formula, and we get 
\[
\gamma / b_j = \sum_i (\gamma_i \times \beta_i)/b_j = \sum_i \gamma_i \scpr{\beta_i,b_j} = \gamma_j. 
\]
or 
\begin{equation}\label{eqn:slantvsbasis}
\gamma = \sum_i (\gamma/b_i) \times \beta_i
\end{equation}

\subsubsection*{The $\lambda$-classes}

\begin{defn}\label{defn:lambdaclass}
Let $X$ and $Z$ be pointed spaces. The evaluation map $\ev: \map_* (X;Z) \times X \to Z$ and the slant product yield
\[
\lambda: H^n (Z) \otimes H_k (X) \to H^{n-k} (\map_* (X;Z)), \; \xi \otimes x \mapsto \lambda_{x,\xi}:= (\ev^* \xi)/x. 
\]
\end{defn}
This construction enjoys a naturality property, which is most concisely expressed by saying that $\lambda$ is a natural transformation of functors 
\[
\sSet_* \times \sSet_*^{\op} \to \bQ-\Mod
\]
from $(X,Z) \mapsto H^n (Z) \otimes H_k (X)$ to $(X,Z) \mapsto H^{n-k} (\map_* (X;Z))$. 

If $X$ has finite type over $\bQ$, formula \eqref{eqn:slantvsbasis} leads to the equation 
\begin{equation}\label{eqn:formulaevz}
\ev^* z = \sum_i (\ev^* z/b_i)\times \beta_i = \sum_i \lambda_{b_i,z}\times \beta_i \in H^* (\map_* (X;Z) \times X)
\end{equation}
for $z \in H^*(Z)$. 

\begin{lem}\label{lem:product-lambdaclasses}
Let $c^k_{i,j}$ be the structure constants of the algebra $H^* (X)$ with respect to the basis $(\beta_i)$, i.e.
\[
\beta_i \beta_j = \sum_k c^k_{i,j} \beta_k. 
\]
Then 
\[
\lambda_{b_k,zy} =  \sum_{i,j}(-1)^{|b_i|(|y|-|b_j|)}c_{i,j}^k \lambda_{b_i,z}\lambda_{b_j,y}
\]
for all $z,y \in H^* (Z)$. 
\end{lem}

\begin{proof}
Using \eqref{eqn:formulaevz}, one checks that firstly
\[
\ev^* (z) \ev^* (y)= \ev^* (zy)= \sum_k \lambda_{b_k,zy}\times \beta_k  ,
\]
and secondly
\[
\begin{split}
\ev^* (z) \ev^* (y) & = (\sum_i \lambda_{b_i,z}\times \beta_i) (\sum_j \lambda_{b_j,y}\times \beta_j) = \\
& = \sum_{i,j}(-1)^{|b_i|(|y|-|b_j|)} (\lambda_{b_i,z}\lambda_{b_j,y})\times (\beta_i\beta_j)=\\
& =  \sum_{i,j,k}(-1)^{|b_i|(|y|-|b_j|)}c_{i,j}^k (\lambda_{b_i,z}\lambda_{b_j,y})\times \beta_k. 
\end{split}
\]
Comparing coefficients yields the result.
\end{proof}

\begin{lem}\label{lem:cuplength}
Assume that $X$ is connected and that the cup length of $X$ is $\leq r-1$ (i.e. cup products of $r$ elements of $\tilde{H}^* (X)$ vanish). Then for $z_1, \ldots,z_r \in H^* (Z)$, we have 
\[
\lambda_{b_k,z_1 \cdots z_r}=0. 
\]
\end{lem}

\begin{proof}
Using \eqref{eqn:formulaevz}, we compute 
\[
\begin{split}
\sum_k \lambda_{b_k,z_1 \cdots z_r} \times \beta_k  & = \ev^* (z_1 \cdots z_r) = \ev^* (z_1) \cdots \ev^*(z_r) = \\
& = \sum_{j_1, \ldots,j_r} \epsilon_{j_1, \ldots ,j_r} (\lambda_{b_{j_1,z_1}} \cdots \lambda_{b_{j_r,z_r}}) \times (\beta_{j_1} \cdots \beta_{j_r}) =0
\end{split}
\]
for some signs $\epsilon_{j_1, \ldots ,j_r} \in \{\pm 1 \}$. The claim follows immediately. 
\end{proof}

\subsubsection*{The case of an Eilenberg--Mac-Lane space}

Now we consider the case $Z=K(\bQ,m)$ and let $u_m \in H^m (K(\bQ,m);\bQ)$ be the fundamental cohomology class. 

\begin{prop}\label{prop:cohomologyofmappingspace1}
Assume that $X$ is connected and of finite type over $\bQ$. Then the natural map 
\[
\bF(\bigoplus_{k=1}^{m-1} H_k (X)[m-k]) \to  H^* (\map_* (X;K(\bQ,m)^0))
\]
induced by the maps
\[
H_k (X) \to H^{m-k}(\map_* (X;K(\bQ,m)^0)), \; a \mapsto \lambda_{a,u_m}
\]
is an isomorphism. Hence $H^* (\map_* (X;K(\bQ,m)^0))$ is the free graded commutative algebra generated by the elements $\lambda_{a_i,u_m}$, where $a_i$ runs through a homogeneous basis for $\tilde{H}_{*<m} (X)$. 
\end{prop}

For the proof, we need a lemma. 

\begin{lem}\label{lem:bilinearform}
Let $Y$ be a connected simple space, and $V_*$ be an $\bN_0$-graded, degreewise finite-dimensional $\bQ$-vector space with $V_0=0$ and let $\sigma_*: V_* \to H^* (Y;\bQ)$ be a graded linear map. Assume that the bilinear form 
\[
B_k: \pi_k (Y)\otimes \bQ \times V_k \to \bQ, \; (y,v) \mapsto \scpr{\sigma(v) , \hur_k^\bQ(y)}
\]
is nondegenerate, for each $k$. Then the induced algebra map
\[
\bF(\sigma): \bF(V_*) \to H^* (Y;\bQ)
\]
from the free graded-commutative algebra on $V$ to the cohomology of $Y$ is an isomorphism. 
\end{lem}

\begin{proof}
From the assumption that $B_k$ is nondegenerate, it follows that $\pi_k (Y) \otimes \bQ$ is finite-dimensional, so that $Y$ is of finite type over $\bQ$.
The $k$th component $\sigma_k$ dualizes to a map $\sigma_k^\vee: H_k (Y;\bQ) \to V_k^\vee$, and there is a unique (up to homotopy) map $F_k: Y \to K(V_k^\vee,k)$ which induces $\sigma_k^\vee$ on $k$th rational homology. The maps $F_k$ together give a map 
\[
F: Y \to \prod_{k\geq 1 } K(V_k^\vee,k)
\]
The induced map $F^*: H^* (\prod_{k\geq 1 } K(V_k^\vee,k);\bQ)\to H^* (Y;\bQ)$ can be identified with $\bF(\sigma)$. Hence it must be verified that the nondegeneracy of all $B_k$'s implies that $F$ is a rational homotopy equivalence. 
The induced map 
\[
F_*: \pi_k (Y) \otimes \bQ \to \pi_k (\prod_{k\geq 1 } K(V_k^\vee,k)) \otimes \bQ= V_k^\vee
\]
is adjoint to a bilinear map
\[
C_k:\pi_k (Y) \otimes \bQ \times V_k \to \bQ,
\]
and we claim that $C_k =B_k$. Because of this identity, nondegeneracy of $B_k$ implies that $F_*$ is an isomorphism (in degree $k$). 

To check that $C_k = B_k$, let $v \in V_k \subset H^k (\prod_{k\geq 1 } K(V_k^\vee,k);\bQ)$ and $y \in \pi_k (Y) \otimes \bQ$ and calculate
\[
C_k (y,v)= \langle v, \hur (F_* (y)) \rangle  = \langle v, F_* \hur (y) \rangle  = \langle F^* v, \hur (y) \rangle = \langle \sigma (v),\hur(y)\rangle = B_k (y,v).
\]
\end{proof}

\begin{proof}[Proof of Proposition \ref{prop:cohomologyofmappingspace1}]
The space $\map_* (X;K(\bQ,m))^0$ is a connected infinite loop space; in fact it is a generalized Eilenberg--Mac Lane space with homotopy groups 
\begin{equation}\label{eqn:homotopymappingspace}
 \pi_k (\map_* (X;K(\bQ,m))) \otimes \bQ = [X,\Omega^k K(\bQ,m)]_* =   \tilde{H}^{m-k} (X) 
\end{equation}
in positive degrees. It follows from the Milnor--Moore theorem \cite{MilnorMoore} and the assumption that $X$ is of finite type that $H^* (\map_* (X;K(\bQ,m)^0))$ is isomorphic (as a graded algebra) to $\bF(\bigoplus_{k=1}^{m-1} H_k (X)[m-k])$. The only issue is to verify that the map in question is indeed an isomorphism. 

We shall show that the bilinear form 
\[
B_{X,k}: \pi_{m-k}  (\map_* (X;K(\bQ,m)^0)) \times \tilde{H}_k (X) \to \bQ
\]
given by 
\[
([f], a) \mapsto \scpr{\lambda_{a,u},\hur_{m-k}^\bQ([f])}
\]
is nondegenerate, and invoke Lemma \ref{lem:bilinearform}. 

If $F: X \to Y$ is a map of pointed spaces, we denote by $F^\sharp : \map_* (Y;K(\bQ,m))^0 \to \map_* (X;K(\bQ,m))^0$ the induced map. By the naturality of the slant product and the Kronecker product, we have 
\[
B_{Y,k} ([f],F_* a) = B_{X,k} ((F^\sharp)_* [f],a). 
\]
If $F_*: H_k (X) \to H_k (Y)$ is an isomorphism then so is the induced map $(F^\sharp)_*:\pi_{m-k}(\map_* (Y;K(\bQ,m))^0) \to \pi_{m-k}(\map_* (X;K(\bQ,m))^0)$ by \eqref{eqn:homotopymappingspace}, and it follows that nondegeneracy of $B_{Y,k}$ is equivalent to nondegeneracy of $B_{X,k} $. Now there are maps $X \to K(H_k (X;\bQ),k)$ and $\bigvee^g S^k \to K(H_k (X;\bQ),k)$ inducing isomorphisms on $H_k (\_;\bQ)$. This argument proves that it suffices to consider the case where $X$ is a wedge of finitely many $k$-spheres. 
Using the naturality again, the form $B_{\bigvee^g S^k,k}$ decomposes as the direct sum of $k$ copies of the form $B_{S^k,k}$, and so we are left with the case $X=S^k$. 

Both, $\pi_{m-k}  (\map_* (S^k;K(\bQ,m)^0))\otimes \bQ$ and $\tilde{H}_k (S^k;\bQ)$ are $1$-dimensional and $\map_* (S^k;K(\bQ,m))^0 \simeq K(\bQ,m-k)$ is $(m-k-1)$-connected. Hence (by the Hurewicz theorem) it is left to be proven that $\lambda_{[S^k],u_m} \in H^{m-k} (\map_* (S^k;K(\bQ,m)))$ is nonzero. On the other hand, the map 
\[
H^m (K(\bQ,m);\bQ) \to H^{m-k} (\map_* (S^k;K(\bQ,m))) = H^{m-k} (\Omega^k K(\bQ,m)), \; u \mapsto \ev^* (u)/[S^k]
\]
can be identified with the ''transgression`` map 
\[
H^m (K(\bQ,m);\bQ) = [K(\bQ,m);K(\bQ,m)] \stackrel{\Omega^k}{\to} [\Omega^k K(\bQ,m);\Omega^k K(\bQ,m)] = H^{m-k} (\Omega^k K(\bQ,m))
\]
which is well-known to be an isomorphism. 
\end{proof}

Using that $BO_\bQ\simeq \prod_{m \geq 1} K(\bQ;4m)$ via the classes $L_m$, we deduce:

\begin{cor}\label{cor:cohomology-mappingspace}
Let $X$ be a connected space of finite type over $\bQ$. Then the map 
\[
\bF \Bigl( \bigoplus_{m\geq 1, 0<k<4m} H_k (X)[4m-k]\Bigr)\to H^* (\map_* (X;BO_\bQ)^0;\bQ)
\]
which on the generators is the direct sum of the maps 
\[
H_k (X)[4m-k] \to H^{4m-k}(H^* (\map_* (X;BO_\bQ)^0;\bQ)); \; a \mapsto \lambda_{a,L_m}
\]
is an isomorphism. Thus $H^* (\map_* (X;BO_\bQ)^0;\bQ)$ is the free graded-commutative algebra with generators 
\[
\lambda_{a_i,L_m} \in H^{4m-|a_i|} (\map_* (X;BO_\bQ)^0;\bQ)
\]
where $a_i$ runs through a homogeneous basis for $\tilde{H}_*(X)$ and $m$ through the natural numbers and $4m-|a_i|>0$. The classes $\lambda_{a_i,L_{m_1} \cdots L_{m_r}}$ with $m_1 \leq \ldots \leq m_r$ are determined by the relation stated in Lemma \ref{lem:product-lambdaclasses}. \qed
\end{cor}

Let us collect a useful consequence of the naturality of the $\lambda$-classes. Let $(X,A)$ be a pair of finite spaces, with $X$ nilpotent and $A$ connected. The fibration
\[
\hq{\map_A (X_\bQ;BO_\bQ)^0}{\hAut_A (X_\bQ)^\id} \to B \hAut_A (X_\bQ)^\id
\]
is $\pi_0 (\hAut_A(X_\bQ))$-equivariant. Hence its Leray--Serre spectral sequence $E_*^{*,*}$ is a spectral sequence of $\pi_0 (\hAut_A(X_\bQ))$-modules. On the other hand, the $E_2$-term is identified as
\begin{equation}\label{eqn:idetificationE_2term}
\begin{split}
E_2^{*,*}= H^* (B \hAut_A (X_\bQ)^\id;\bQ) \otimes H^* (\map_A (X_\bQ;BO_\bQ)^0;\bQ) \cong \\
H^* (B \hAut_A (X_\bQ)^\id;\bQ) \otimes \bF \Bigl( \bigoplus_{m\geq 1, 0<k<4m} H_k (X/A;\bQ)[4m-k]\Bigr).
\end{split}
\end{equation}
On the other hand, $\pi_0 (\hAut_A(X_\bQ))=\pi_1 (B \hAut_A(X_\bQ))$ acts on the universal cover $B \hAut_A (X_\bQ)^\id$ in the usual way and hence on its cohomology; and it acts on $\bF \Bigl( \bigoplus_{m\geq 1, 0<k<4m} H_k (X/A)[4m-k]\Bigr)$ via the action of $\pi_0 (\hAut_A(X_\bQ))$ on $H_* (X/A;\bQ)= H_* (X_\bQ/A;\bQ)$. 

\begin{cor}\label{cor:actiononcohomologymappingspace}
Let $(X,A)$ be as above. Then the isomorphism \eqref{eqn:idetificationE_2term} is $\pi_0 (\hAut_A(X_\bQ))$-equivariant, where the actions are as just described.
\end{cor}

\subsection{Consequences for block diffeomorphism spaces}\label{subsec:finalformula-for-blockdiffs}

We now collect some consequences of the above calculations. Let $M^d$ be a smooth oriented compact manifold with boundary. In that case, there is a fundamental class $[M] \in H_d (M/\partial M)\cong H_d (M,\partial M)$. 

\begin{lem}\label{lem:kappaclassaslambdaclass}
Let $M$ be a smooth oriented compact manifold with boundary, and denote by 
\[
q: \map_\partial (M;BO)  \to \hq{\map_\partial (M;BO)}{\hAut^+_\partial (M)}
\]
the fibre inclusion. Then 
\[
q^* \kappa_c = \lambda_{[M],c}
\]
for each $c \in H^* (BO;\bQ)$. 
\end{lem}

\begin{proof}
This is a consequence of the definitions and the fact that for a space $Y$, the Gysin map of the trivial bundle $\pi: Y \times (M/\partial M) \to Y$ is given by $\pi_! (\xi)=\xi/[M]$.  

The latter identity holds for classes of the form $\xi=\alpha \times \beta$, by the naturality and linearity of $\pi_!$ over the cohomology of the base, together with formula \eqref{eqn:slant-identity}. The K\"unneth formula shows that it is enough to consider such $\xi$.
\end{proof}

\begin{thm}\label{thm:blockdiffsversusmappingspace}
Assume that $M$ is a stably parallelizable $1$-connected smooth manifold of dimension $d \geq 5$ with $\partial M   = S^{d-1}$. 
Then the algebra map
\[
H^* (\hq{\map_\partial (M_\bQ;BO_\bQ)^0}{\hAut_\partial(M_\bQ)^{\cong}}) \to H^* (B \blockdiff_\partial (M))
\]
is surjective. Moreover
\begin{enumerate}
\item If $d$ is odd, the kernel is the ideal generated by the classes $\kappa_{L_m}$, $4m-d>0$.  
\item If $d=2n$ is even, the kernel is the ideal generated by the classes $\kappa_{L_m}-\sigma_{4m-d}$, $4m-d>0$.
\end{enumerate}
\end{thm}

\begin{proof}
The composition 
\[
B \blockdiff_\partial (M) \to \hq{\map_\partial (M_\bQ;BO_\bQ)^0}{\hAut_\partial(M_\bQ)^{\cong}} \to \hq{\map_* (M_\bQ;BO_\bQ)^0}{\hAut_\partial(M_\bQ)^{\cong}}
\]
is an $H \bQ$-equivalence by Proposition \ref{prop:rationalized-blockdiffs}, and so surjectivity is immediate. Theorem \ref{thm:vanishing} and \eqref{eqn:evenfamilysignautre} prove that the classes listed in (1) and (2) lie in the kernel, and so it remains for us to show that the kernel is not larger. Consider the map 
\begin{equation}\label{eqn:map-in-blockdiffsversusmappingspace}
\Pi: \hq{\map_\partial (M_\bQ;BO_\bQ)^0}{\hAut^+_\partial(M_\bQ)^{\cong}}  \to \hq{\map_* (M_\bQ;BO_\bQ)^0 }{ \hAut^+_\partial(M_\bQ)^{\cong}} \times \prod_{4m-d>0} K(\bQ,4m-d)
\end{equation}
made out of the forgetful map and the classes $\kappa_{L_m}$ or $\kappa_{L_m}-\sigma_{4m-d}$ (for odd or even $d$). If we can prove that $\Pi$ is a weak equivalence, the claim follows. 

To prove this, note first that $\Pi$ is a map over $B \hAut^+_\partial (M_\bQ)^{\cong}$. The induced map on the homotopy fibres of the respective maps to $B \hAut^+_\partial (M_\bQ)^{\cong}$ is the map 
\begin{equation}\label{eqn:map-in-blockdiffsversusmappingspace23434}
 \map_\partial (M_\bQ;BO_\bQ)^0 \to  \map_* (M_\bQ;BO_\bQ)^0 \times \prod_{4m-d>0} K(\bQ,4m-d);
\end{equation}
made out of the forgetful map and the classes $\lambda_{L_m}$. This follows from Lemma \ref{lem:kappaclassaslambdaclass} and the fact that the classes $\sigma_{4m-d}$ are pulled back from $B \hAut^+_\partial (M_\bQ)_\bZ$.

To see that \eqref{eqn:map-in-blockdiffsversusmappingspace23434} is a weak equivalence, one uses Corollary \ref{cor:cohomology-mappingspace} for both mapping spaces and the naturality of the $\lambda$-classes.
\end{proof}

\begin{rem}
The proof suggests that one might extract a proof of Theorem \ref{thm:vanishing} out of the arguments in \cite[\S 4]{BerglundMadsen}. While surely true, it seems simpler to us to use the independent proof of Theorem \ref{thm:vanishing}, as it allows us to treat \cite[\S 4]{BerglundMadsen} as a black box. 
\end{rem}

\section{Homotopy calculations for the manifolds \texorpdfstring{$U_{g,1}^n$}{U}}\label{sec:homotopyautomorphisms}

We wish to apply Theorem \ref{thm:blockdiffsversusmappingspace} to the manifolds $U_{g,1}^n$. We shall approach the required calculation of $H^* (\hq{\map_\partial (M_\bQ;BO_\bQ)^0}{\hAut_\partial(M_\bQ)^{\cong}};\bQ)$ through the spectral sequences of the fibre sequences
\begin{equation}\label{eqn:firstfibresequence-for-blockdiffs}
\map_\partial ((U_{g,1}^n)_\bQ;BO_\bQ)^0 \to \hq{\map_\partial ((U_{g,1}^n)_\bQ;BO_\bQ)^0}{\hAut_\partial((U_{g,1}^n)_\bQ)^{\id}} \to B \hAut_\partial((U_{g,1}^n)_\bQ)^{\id}
\end{equation}
(which is a spectral sequence of $\pi_0 (\hAut_\partial ((U_{g,1}^n)_\bQ))$-modules) and 
\begin{equation}\label{eqn:secondfibresequence-for-blockdiffs}
\hq{\map_\partial ((U_{g,1}^n)_\bQ;BO_\bQ)^0}{\hAut_\partial((U_{g,1}^n)_\bQ)^{\id}} \to \hq{\map_\partial ((U_{g,1}^n)_\bQ;BO_\bQ)^0}{\hAut_\partial((U_{g,1}^n)_\bQ)^{\cong}} \to B \pi_0 \hAut_\partial((U_{g,1}^n)_\bQ)^{\cong}. 
\end{equation}
In this section, we shall collect the homotopy-theoretic input needed for these computations. 
Sections \S \ref{subsec:lowdimpistar}, \ref{subsec:mappingclassgroupsug1n} collect some standard homotopy-theoretic facts. This leads among other things to Lemma \ref{lem:pinullhautrelbound} which identifies $\pi_0 \hAut_\partial((U_{g,1}^n)_\bQ)$ with $\GL_g (\bQ)$ and $\pi_0 \hAut_\partial((U_{g,1}^n)_\bQ)^{\cong}$ with $\GL_g (\bZ)$. In \S \ref{subsec:higherhomotopy-hautug1n}, we compute enough of the homotopy groups of $B \hAut_\partial((U_{g,1}^n)_\bQ)^{\id}$ to identify the $E_2$-term of the spectral sequence for \eqref{eqn:firstfibresequence-for-blockdiffs} up to the $n-3$rd column, as $\GL_g (\bQ)$-modules. We also identify the $d_2$-differential of that spectral sequence, which is based on Lemma \ref{lem:lowdimhomotopyhomoauto2} and carried out in \S \ref{subsec:computationdifferential}. 
Also in this section, we give the relatively easy proof of Theorem \ref{thm:blockdiffs} in the case $n=4$.
\subsection{Homotopy groups}\label{subsec:lowdimpistar}

Let us now focus our attention to the manifolds we are actually interested in, i.e. 
\begin{equation}\label{eqn:Ugnmanifolds}
U_{g,1}^n := \sharp^g (S^n \times S^{n+1}) \setminus \inter{D^{2n+1}}. 
\end{equation}
Note that 
\[
U_{g,1}^n \simeq \bigvee^g S^n \vee \bigvee^g S^{n+1}. 
\]
Let us denote
\[
N(g)_\bZ := H_n (U_{g,1}^n;\bZ) \cong \bZ^g
\]
and 
\[
N(g):= H_n (U_{g,1}^n;\bQ) = N(g)_\bZ \otimes \bQ.
\] 
Poincar\'e duality and the universal coefficient theorem show that
\[
H_{n+1} (U_{g,1}^n;\bZ) \cong H^n (U_{g,1}^n;\bZ) \cong N(g)_\bZ^\vee. 
\]
These isomorphisms are natural with respect to the action of $\hAut_\partial (U_{g,1}^n)$, and hold similarly for rational coefficients. The group $\hAut_\partial (U_{g,1}^n)$ acts trivially on $H_{2n+1}(U_{g,1}^n,\partial;\bZ)$, $H_{2n+1}(U_{g,1}^n,\partial;\bQ)$, $H^{2n+1}(U_{g,1}^n,\partial;\bZ)$ and $H^{2n+1}(U_{g,1}^n,\partial;\bQ)$.

Let $x_1, \ldots,x_g \in \pi_n (U_{g,1}^n)$ be the elements represented by the inclusion of the $g$ different $S^n$'s, and similarly let $y_1, \ldots, y_g \in \pi_{n+1}(U_{g,1}^n)$ be represented by the $g$ copies of $S^{n+1}$. 
The inclusion of the boundary $S^{2n}= \partial U_{g,1}^n \to U_{g,1}^n$ represents an element $\omega \in \pi_{2n}(U_{g,1}^n)$ which agrees with the sum 
\begin{equation}\label{eqn:inclusionaswhitenehadprofduct}
\omega=\sum_{j=1}^g [x_j,y_j]
\end{equation}
of Whitehead products of the generators, if the numbering and the signs of the generators are chosen appropriately. We denote by 
\[
a_j := \hur_n^\bZ(x_j) \in H_n (U_{g,1}^n;\bZ)\; \text{and}\; b_j := \hur_{n+1}^\bZ(y_j) \in H_{n+1} (U_{g,1}^n;\bZ)
\]
the images under the Hurewicz homomorphism, and use the same symbol for the images in rational homology.

We let $(\alpha_1, \ldots,\alpha_g)$ be the basis of $H^{n}(U_{g,1}^n;\bZ) = H^{n}(U_{g,1}^n,\partial;\bZ)$ dual to $(a_1, \ldots,a_g)$ and let $(\beta_1 ,\ldots,\beta_g)$ be the basis of $H^{n+1}(U_{g,1}^n;\bZ)=H^{n+1}(U_{g,1}^n,\partial;\bZ)$ dual to $(b_1, \ldots,b_g)$. Let $\nu \in H^{2n+1}(U_{g,1}^n,\partial;\bZ)$ be dual to the fundamental class $[U]=[U_{g,1}^n]$. The cup product structure of $H^* (U_{g,1}^n;\bZ)$ is given by 
\begin{equation}\label{eqn:cupproductsug1n}
\alpha_i  \beta_j  =\delta_{ij} \nu;
\end{equation}
all other cup products are zero for degree reasons. The cup length of $U_{g,1}^n/\partial$ is therefore $2$, and it follows from Lemma \ref{lem:product-lambdaclasses} that
\begin{equation}\label{eqn:lambdaproductug1n}
\lambda_{[U],L_m L_k}= \sum_{j=1}^g \lambda_{a_j,L_m} \lambda_{b_j,L_k}+ \sum_{j=1}^g \lambda_{b_j,L_m} \lambda_{a_j,L_k},
\end{equation}
and from Lemma \ref{lem:cuplength} that 
\begin{equation}\label{eqn:lambdaproductug1n12234345}
\lambda_{[U],L_{m_1} \cdots L_{m_r}}= 0
\end{equation}
if $r \geq 3$ and $1 \leq m_1 \leq \ldots \leq m_r$. 

Having understood the homological structure, let us turn to homotopy groups. The Hurewicz theorem shows that 
\[
\hur_n^\bZ: \pi_n (U_{g,1}^n) \to H_n (U_{g,1}^n;\bZ)
\]
is an isomorphism, and that 
\[
\hur_{n+1}^\bZ: \pi_{n+1} (U_{g,1}^n) \to H_{n+1} (U_{g,1}^n;\bZ)
\]
is surjective, if $n \geq 2$. The rational Hurewicz theorem implies that $\hur_{n+1}^\bZ$ has finite kernel if $n\geq 3$. 

The rational homotopy groups of $U_{g,1}^n$ are known by a computation in rational homotopy theory. For a graded vector space $V$, we denote the desuspension by $(s^{-1} V)_k := V_{k+1}$. For a ($1$-connected) space $X$, $s^{-1} \pi_* (X_\bQ)$ is a graded Lie algebra under the Whitehead product, and since $U_{g,1}^n$ is a wedge of spheres, we have 
\[
s^{-1} \pi_* ((U_{g,1}^n)_\bQ) \cong \bL (\pi_n((U_{g,1}^n)_\bQ) \oplus \pi_{n+1}((U_{g,1}^n)_\bQ)),
\]
the free Lie algebra on the graded vector space which is $\pi_n((U_{g,1}^n)_\bQ)$ in degree $n-1$ and $\pi_{n+1}((U_{g,1}^n)_\bQ)$ in degree $n$, by \cite[Theorem 24.5]{FHT}. It follows that
\begin{equation}\label{eqn:vanishingrationalhomotopy}
\pi_{n+k} (U_{g,1}^n) \otimes \bQ =0  \;\text{if} \; 2\leq k \leq n-2 \, \text{or}\, n+2 \leq k \leq 2n-3. 
\end{equation}

\subsection{Mapping class groups}\label{subsec:mappingclassgroupsug1n}

\begin{lem}\label{lem:pinullautvgn}
The map
\[
\pi_0 (\hAut_*(U_{g,1}^n)) \to \GL (N(g)_\bZ) \times \GL(N(g)_\bZ^\vee)
\]
given by the action on integral homology is surjective when $n \geq 1$, and has finite kernel, provided that $n \geq 3$. 

The analogous map 
\[
\pi_0 (\hAut_*((U_{g,1}^n)_\bQ)) \to \GL (N(g)) \times \GL(N(g)^\vee)
\]
is an isomorphism if $n \geq 3$, and it maps the subgroup $\pi_0 (\hAut_*((U_{g,1}^n)_\bQ)_\bZ) \subset \pi_0 (\hAut_*((U_{g,1}^n)_\bQ))$ onto the subgroup $\GL (N(g)_\bZ) \times \GL(N(g)_\bZ^\vee)$. 
\end{lem}

\begin{proof}
Since $\hAut_* (\bigvee^g S^n) \to \GL (H_n  (\bigvee^g S^n;\bZ))$ is surjective, and similarly for a 
wedge of $(n+1)$-spheres, surjectivity follows. 

To see that the kernel is finite, we note that since $n \geq 3$, $\hur_{n+1}^\bZ$ identifies $N(g)^\vee_\bZ=H_{n+1} (U_{g,1}^n;\bZ)$ with the torsionfree quotient of $\pi_{n+1}(U_{g,1}^n)$, which is why we can factor the map in question as
\begin{equation}\label{eqn:facotriztationspinullahhaut}
\pi_0 (\hAut(U_{g,1}^n)) \to \GL (\pi_n(U_{g,1}^n)) \times \GL (\pi_{n+1}(U_{g,1}^n)) \to \GL (N(g)_\bZ) \times \GL(N(g)_\bZ^\vee).
\end{equation}
The kernel of the second map can be identified with 
\[
\GL (\ker (\hur_{n+1}^\bZ)) \ltimes \Hom (H_{n+1}(U_{g,1}^n;\bZ);\ker (\hur_{n+1}^\bZ))
\]
and is therefore finite.

The first map in \eqref{eqn:facotriztationspinullahhaut} is injective: a pointed map $f: U_{g,1}^n\to U_{g,1}^n$ which induces the identity on both, $\pi_n$ and $\pi_{n+1}$, must be homotopic to the identity, as $U_{g,1}^n$ is a wedge of spheres.

The rational case is analogous. The proof of surjectivity is similar (using elementary matrices to generate $\GL_g (\bQ)$), and in the factorization analogous to \eqref{eqn:facotriztationspinullahhaut}, the second map is an isomorphism since $\hur_{n+1}^\bQ$ is an isomorphism.

The last statement follows easily from the others.  
\end{proof}

We are interested in $\hAut_\partial (U_{g,1}^n)$ rather than $\hAut_*(U_{g,1}^n)$. These two spaces are related by a fibre sequence 
\begin{equation}\label{fibresequence}
\hAut_\partial (U_{g,1}^n) \to \hAut_* (U_{g,1}^n) \to \Omega^{2n} U_{g,1}^n. 
\end{equation}

\begin{lem}\label{lem:pinullhautrelbound}
Let $n \geq 3$. 
\begin{enumerate}
\item The map 
\[
\pi_0 (\hAut_\partial ((U_{g,1}^n)_\bQ)) \to \GL(N(g))
\]
induced by the action on $H_n(U_{g,1}^n;\bQ)$ is an isomorphism.
\item It maps the subgroup $\pi_0 (\hAut_\partial ((U_{g,1}^n)_\bQ)_\bZ)$ onto $\GL(N(g)_\bZ)$. 
\item Moreover $\pi_0 (\hAut_\partial ((U_{g,1}^n)_\bQ)^{\cong}) = \pi_0 (\hAut_\partial ((U_{g,1}^n)_\bQ)_\bZ)$. 
\end{enumerate}
\end{lem}

\begin{proof}
(1) Consider the diagram
\begin{equation}\label{diag:asfkjdafafgkjlpolzu}
\xymatrix{
\pi_0 (\hAut_\partial ((U_{g,1}^n)_\bQ)) \ar[r] \ar[d] & \GL (N(g)) \ar[d]^{\Delta} \\
\pi_0 (\hAut_*((U_{g,1}^n)_\bQ)) \ar[r] & \GL (N(g)) \times \GL(N(g)^\vee);
}
\end{equation}
the left vertical map is the obvious one, and the map $\Delta$ sends $h \in \GL(N(g)_\bZ)$ to $(h,(h^{-1})^\vee)$. The horizontal maps are induced by the action on homology, and the diagram commutes by Poincar\'e--Lefschetz duality.

The bottom horizontal map is an isomorphism by Lemma \ref{lem:pinullautvgn}, and our task is to prove that the top horizontal map is an isomorphism as well.  
To prove that the top horizontal map in \eqref{diag:asfkjdafafgkjlpolzu} is injective, it suffices to show that the left vertical map is injective, but that fits into an exact sequence 
\[
\pi_{2n+1} ((U_{g,1}^n)_\bQ) \stackrel{\delta}{\to} \pi_0 (\hAut_\partial ((U_{g,1}^n)_\bQ)) \to \pi_0 (\hAut_*((U_{g,1}^n)_\bQ)) 
\]
coming from the fibre sequence \eqref{fibresequence} or the rationalized version thereof. It is proven in \cite{Grey} that the connecting homomorphism $\delta$ is zero (see the proof of Proposition 6.6 and Remark 6.7 of the quoted paper). 

To prove that the top horizontal map of \eqref{diag:asfkjdafafgkjlpolzu} is surjective, we start with $h \in \GL (N(g))$ and use Lemma \ref{lem:pinullautvgn} to find $\varphi \in \hAut_* ((U_{g,1}^n)_\bQ)$ which realizes $\Delta (h)$ on $H_n(U_{g,1}^n;\bQ) \oplus H_{n+1}(U_{g,1}^n;\bQ)$. 

Let $f=(f_{ij})$ and $g=(g_{ij})$ be the matrices which describe the effect of $\varphi$ on $H_n(U_{g,1}^n;\bQ)$ and $H_{n+1}(U_{g,1}^n;\bQ)$, respectively, in terms of the bases $(a_1, \ldots,a_g)$ and $(b_1, \ldots,b_g)$. By construction $f^\top g=1$ or $\sum_i f_{ij} g_{ik}=\delta_{jk}$. Using \eqref{eqn:inclusionaswhitenehadprofduct}, we deduce that 
\[
\varphi_* (\omega)=\omega \in \pi_{2n}((U_{g,1}^n)_\bQ). 
\]
So $\varphi$ is homotopic relative to the basepoint to a map which is the identity on $(\partial U_{g,1}^n)_\bQ$. So $\varphi$ can be viewed as an element of $\hAut_\partial ((U_{g,1}^n)_\bQ)$ which maps to $h$ under the top horizontal map of \eqref{diag:asfkjdafafgkjlpolzu}. 

To show (2) and (3), we consider the commutative diagram
\begin{equation}\label{asdfjasdlklsakjdffghqqqqqqqqq}
\xymatrix{
\pi_0 (\hAut_\partial ((U_{g,1}^n)_\bQ))^{\cong} \ar[dr] \ar[d] & \\
\pi_0 (\hAut_\partial ((U_{g,1}^n)_\bQ))_\bZ \ar[r] \ar[d] & \GL(N(g)_\bZ)\ar[d]\\
\pi_0 (\hAut_\partial ((U_{g,1}^n)_\bQ)) \ar[r]^-{\cong}  & \GL(N(g)); \\
}
\end{equation}
the bottom map is an isomorphism by item (1), and all maps pointing downwards are by definition injective. 

Now every linear automorphism of $H_n (U_{g,1}^n;\bZ)$ can be realized by a diffeomorphism fixing the boundary pointwise: it is explained in the proof of \cite[Proposition 5.3]{Grey} how to deduce this fact from \cite[Lemma 17]{Wall1963}. It follows that the diagonal map in \eqref{asdfjasdlklsakjdffghqqqqqqqqq} is surjective, hence both maps 
\[
\pi_0 (\hAut_\partial ((U_{g,1}^n)_\bQ))^{\cong} \to \pi_0 (\hAut_\partial ((U_{g,1}^n)_\bQ))_\bZ \to \GL(N(g)_\bZ)
\]
are isomorphisms; which shows (2) and (3). 
\end{proof}

\subsection{Some higher homotopy groups of the homotopy automorphisms}\label{subsec:higherhomotopy-hautug1n}

Let us turn to a description of some of the higher homotopy groups of $\hAut_* (U_{g,1}^n)$ and $\hAut_\partial (U_{g,1}^n)$. Lemma \ref{lem:rationalization-homotopyautomorphisms} shows that the rationalization map 
\[
\pi_k (B \hAut_* (U_{g,1}^n)) \otimes \bQ \to \pi_k (B \hAut_* ((U_{g,1}^n)_\bQ))\otimes \bQ= \pi_k (B \hAut_* ((U_{g,1}^n)_\bQ))
\]
is an isomorphism when $k \geq 2$. The fundamental group 
\[
\pi_1 (B\hAut_* ((U_{g,1}^n)_\bQ))  \cong \pi_0 (\hAut_* ((U_{g,1}^n)_\bQ)) \cong \GL(N(g)) \times \GL(N(g)^\vee)
\]
acts on the higher homotopy groups of $B \hAut_* ((U_{g,1}^n)_\bQ)$ in the usual way. 

\begin{lem}\label{lem:lowdimhomotopyhomoauto1}
We have 
\[
\dim_\bQ (\pi_k (B\hAut_* (U_{g,1}^n))\otimes \bQ)=
\begin{cases}
g^2 & k=2 \;\text{and} \; n \geq 3,\\
0 & 3 \leq k \leq n-2.
\end{cases}
\]
\end{lem}

\begin{proof}
For $k \geq 1$, we have $ \pi_k (\hAut_* (U_{g,1}^n)) \cong \pi_k (\map_* (U_{g,1}^n;U_{g,1}^n);\id)$. The inclusion 
\[
U_{g,1}^n \simeq \bigvee^g S^n \vee \bigvee^g S^{n+1} \to (S^n)^g \times (S^{n+1})^g
\]
is the inclusion of the $(2n-1)$-skeleton and hence $(2n-1)$-connected. So the induced map 
\[
\map_* (U_{g,1}^n;U_{g,1}^n) \to \map_* (U_{g,1}^n;(S^n\times S^{n+1})^g)
\]
is $(2n-1)-(n+1)=(n-2)$-connected. Because
\[
\map_* (U_{g,1}^n;(S^n\times S^{n+1})^g) \simeq (\Omega^n (S^n\times S^{n+1})^g) )^g \times (\Omega^{n+1} (S^n\times S^{n+1})^g) )^g,
\]
we obtain an isomorphism
\[
\begin{split}
& \pi_k (\map_* (U_{g,1}^n;U_{g,1}^n))  \cong\\
& (\pi_{n+k}(S^n))^{\oplus g^2}\oplus  (\pi_{n+k}(S^{n+1}))^{\oplus g^2} \oplus (\pi_{n+1+k}(S^n))^{\oplus g^2} \oplus (\pi_{n+k+1}(S^{n+1}))^{\oplus g^2}
\end{split}
\]
for $1 \leq k \leq n-3$. The claim follows by using the knowledge of the rational homotopy groups of the spheres.
\end{proof}

We need to know a precise description of $\pi_1 (\hAut_* (U_{g,1}^n))\otimes \bQ \cong \pi_1 (\hAut_* ((U_{g,1}^n)_\bQ))$, not merely its dimension which we just computed. We identify 
\[
\pi_1 (\hAut_* (U_{g,1}^n))\cong \pi_2 (B\hAut_* (U_{g,1}^n))= H_2 ((B \hAut_* (U_{g,1}^n))^\id;\bZ). 
\]
An element $\gamma \in \pi_2 (B \hAut_* (U_{g,1}^n))$ classifies a fibration $E_\gamma \to S^2$, together with a cross-section and a homotopy equivalence of the fibre over the basepoint with $U_{g,1}^n$. 

The only potentially nonzero differential of the homological Leray--Serre spectral sequence of $E_\gamma \to S^2$ is the map 
\[
d^2_{2,n}: E^2_{2,n} \to E^2_{0,n+1},
\]
which can be rewritten as 
\[
d(\gamma):  N(g)_\bZ \cong H_2 (S^2;N(g)_\bZ)  \to H_0 (S^2;N(g)^\vee_\bZ) \cong N(g)_\bZ^\vee. 
\]
Therefore, assigning $\gamma \mapsto d(\gamma)$ gives a map 
\begin{equation}\label{eqn:isomorphism-pitwo-homotopyauto}
d: \pi_2 (B\hAut_* (U_{g,1}^n)) \to \Hom (N(g)_\bZ;N(g)_\bZ^\vee). 
\end{equation}
Similarly, we obtain 
\begin{equation}\label{eqn:isomorphism-pitwo-homotopyautoQ}
d_\bQ: \pi_2 (B\hAut_* ((U_{g,1}^n)_\bQ) \to \Hom (N(g);N(g)^\vee). 
\end{equation}
Source and target of \eqref{eqn:isomorphism-pitwo-homotopyauto} are $\bZ[\pi_0 (\hAut_* (U_{g,1}^n))]$-modules; and source and target of \eqref{eqn:isomorphism-pitwo-homotopyautoQ} are $\bQ[\pi_0 (\hAut_* (U_{g,1}^n)_\bQ)]$-modules. 

\begin{lem}\label{lem:lowdimhomotopyhomoauto2}
The map $d$ is a homomorphism of $\bZ[\pi_0 (\hAut_* (U_{g,1}^n))]$-modules, and similarly $d_\bQ$ is a homomorphism of $\bQ[\pi_0 (\hAut_* (U_{g,1}^n)_\bQ)]$-modules. When $n  \geq 4$, $d_\bQ$ is an isomorphism. 
\end{lem}

\begin{proof}
The group $\pi_0 (\hAut_* (U_{g,1}^n))$ acts on $\pi_2 (B\hAut_* (U_{g,1}^n))$ by changing the identification of the fibres. Therefore, it is clear that $d$ is $\pi_0 (\hAut_* (U_{g,1}^n))$-equivariant; and the same argument applies to $d_\bQ$. 

That $d$ and $d_\bQ$ are additive (the latter is then automatically $\bQ$-linear) requires a more detailed argument, which is the same in both cases, so we concentrate on $d$. Let $\gamma_0,\gamma_1$ be two elements of $\pi_2 (B\hAut_* (U_{g,1}^n))$. Identification of the fibres over the basepoint gives a fibration over $S^2 \vee S^2$ which we denote by $E_{\gamma_0} \vee E_{\gamma_1} \to S^2 \vee S^2$, abusing notation for simplicity. 

The pinch map $p: S^2 \to S^2 \vee S^2 $ is covered by a fibrewise homotopy equivalence $E_{\gamma_0+\gamma_1} \to E_{\gamma_0} \vee E_{\gamma_1 } $. We obtain a commutative diagram
\begin{equation}\label{eqn:proofadditiviry1}
\xymatrix{
H_2 (S^2;N(g)_\bZ) \ar[r]^-{p_*} \ar[d]^{d(\gamma_0+\gamma_1)} & H_2 (S^2 \vee S^2;N(g)_\bZ)\ar[d]^{d}\\
H_0 (S^2;N(g)^\vee_\bZ) \ar[r]^-{p_*} & H_0 (S^2\vee S^2;N(g)^\vee_\bZ)
}
\end{equation}
where the right horizontal map is the differential for the spectral sequence of $E_{\gamma_0}\vee E_{\gamma_1} \to S^2 \vee S^2$. Moreover, the two inclusions $S^2 \to S^2 \vee S^2$ are covered by fibrewise homotopy equivalences $E_{\gamma_j} \to E_{\gamma_0} \vee E_{\gamma_1}$, from which we obtain the upper half of the commutative diagram
\begin{equation}\label{eqn:proofadditiviry2}
\xymatrix{
H_2 (S^2;N(g)_\bZ)\oplus H_2 (S^2;N(g)_\bZ) \ar[r]^-{\phi} \ar[d]^{d(\gamma_0)\oplus d(\gamma_1)} & H_2 (S^2 \vee S^2;N(g)_\bZ)\ar[d]^{d}\\
H_0 (S^2;N(g)^\vee_\bZ)\oplus H_0 (S^2;N(g)^\vee_\bZ) \ar[r]^-{\psi} \ar[dr]_{\Sigma} & H_0 (S^2\vee S^2;N(g)^\vee_\bZ) \ar[d]^{f_*}\\
& H_0 (S^2;N(g)^\vee_\bZ);
}
\end{equation}
the map $f_*$ is induced by the fold map $S^2  \vee S^2 \to S^2$, $\Sigma$ is the sum map, and $\phi$ and $\psi$ are just names for the obvious maps.

Now let $a \in H_2 (S^2;N(g)_\bZ)$; then $p_*(a)=\phi(a,a)$. Moreover $f_* \circ p_*: H_0 (S^2;N(g)^\vee_\bZ) \to H_0 (S^2;N(g)^\vee_\bZ)$ is the identity. So 
\[
d(\gamma_0+\gamma_1) (a) = f_* \circ p_* \circ d(\gamma_0+\gamma_1) (a) = f_* \circ d \circ p_* (a) =  
\]
\[
=f_* \circ d \circ \phi(a,a) = \Sigma \circ (d(\gamma_0)\oplus d (\gamma_1))(a) = d(\gamma_0)(a) + d(\gamma_1)(a),
\]
as claimed. 

For the second part of the proof, consider the diagram
\begin{equation}\label{eqn:diagramidentificationpitwobhaut}
\xymatrix{
\pi_2 (B\hAut_* (U_{g,1}^n)) \ar[r]^{d}\ar[d] & \Hom (N(g)_\bZ;N(g)_\bZ^\vee)\ar[d]\\
\pi_2 (B\hAut_* ((U_{g,1}^n)_\bQ) \ar[r]^{d_\bQ} & \Hom (N(g);N(g)^\vee)
}
\end{equation}
which obviously commutes. We will show that 
\begin{equation}\label{eqn:kerneldinkernelrational}
\ker (d) \subset  \ker \Bigl( \pi_2 (B\hAut_* (U_{g,1}^n)) \to \pi_2 (B\hAut_* ((U_{g,1}^n)_\bQ))\Bigl).
\end{equation}
Both vertical maps in \eqref{eqn:diagramidentificationpitwobhaut} are rationalizations (the left one by Lemma \ref{lem:rationalization-homotopyautomorphisms}), and so \eqref{eqn:kerneldinkernelrational} shows that $\ker (d)$ is finite. Hence $d_\bQ$ is injective, and Lemma \ref{lem:lowdimhomotopyhomoauto1}, together with a dimension count, proves that $d_\bQ$ is an isomorphism. 

To show \eqref{eqn:kerneldinkernelrational}, we must show that a fibration $\pi: E \to S^2 $ with fibre $U_{g,1}^n$ and a cross-section whose Leray--Serre spectral sequence collapses at the $E^2$-stage is rationally fibre-homotopy equivalent to $S^2 \times (U_{g,1}^n)_\bQ$. 
Since the spectral sequence collapses and since $n \geq 3$, the inclusion of the fibre $U_{g,1}^n\to E$ induces isomorphisms on $H_n$ and $H_{n+1}$ and so gives isomorphisms $N(g)_\bZ \to H_n (E)$ and $N(g)_\bZ^\vee \to H_{n+1}(E)$. From the inverses of those isomorphisms, we obtain a map 
\[
f: E \to K(N(g)_\bZ,n) \times K(N(g)_\bZ^\vee,n+1) \to K(N(g),n) \times K(N(g)^\vee,n+1)
\]
inducing an isomorphism on rational homology in degrees $n$ and $n+1$.
The natural map $(U_{g,1}^n)_\bQ\to K(N(g),n) \times K(N(g)^\vee,n+1)$, which induces the identity on homology in degrees $n$ and $n+1$, is $(2n-1) \geq (n+3)$-connected, and since the homotopy dimension of $E$ is $n+3$, we can deform $f$ to a map $g:E \to (U_{g,1}^n)_\bQ$, such that the composition with the inclusion of the fibre is the rationalization map $U_{g,1}^n \to (U_{g,1}^n)_\bQ$. The map 
\[
h:= (\pi,g): E \to S^2 \times (U_{g,1}^n)_\bQ
\]
over $S^2$ induces an isomorphism in rational homology up to degree $n+1$. Using a spectral sequence comparison argument, we obtain that $h$ is a rational homology equivalence, and hence also a rational homology equivalence on the fibres.  
\end{proof}

\begin{prop}\label{prop:rationalhomotopyhautug1n}
We have 
\[
\pi_k (B \hAut_\partial (U_{g,1}^n)) \otimes \bQ =0
\]
for $3 \leq k \leq n-3$. If $n \geq 5$, the map 
\[
\pi_2 (B \hAut_\partial (U_{g,1}^n)) \otimes \bQ \to \pi_2 (B \hAut_* (U_{g,1}^n)) \otimes \bQ
\]
is injective and can be identified, $\pi_0 ( \hAut_\partial ((U_{g,1}^n)_\bQ)$-equivariantly, with the inclusion of 
\[
\begin{cases}
S^2 (N(g)^\vee) & n \; \text{even}\\
\Lambda^2 (N(g)^\vee) & n \; \text{odd}
\end{cases}
\]
into $N(g)^\vee \otimes N(g)^\vee = \Hom(N(g);N(g)^\vee)$. 
\end{prop}

\begin{proof}
From the fibre sequence \eqref{fibresequence}, we obtain an exact sequence 
\[
\pi_{2n+k} (U_{g,1}^n) \otimes \bQ \to \pi_k (B \hAut_\partial (U_{g,1}^n))\otimes \bQ \to \pi_k (B \hAut_* (U_{g,1}^n))\otimes \bQ . 
\]
By \eqref{eqn:vanishingrationalhomotopy}, the leftmost group vanishes when $2 \leq k \leq n-3$. Together with Lemma \ref{lem:lowdimhomotopyhomoauto1}, the first claim follows, as well as the injectivity on $\pi_2$. 

For the remaining claim, consider the exact sequence 
\[
\begin{split}
0  \to \pi_2 (B \hAut_\partial ((U_{g,1}^n)_\bQ)) & \to \pi_2 (B \hAut_* ((U_{g,1}^n)_\bQ)) \to \pi_{2n+1} ((U_{g,1}^n)_\bQ) \to\\
\to  \pi_1 (B \hAut_\partial ((U_{g,1}^n)_\bQ)) & \to \pi_1 (B \hAut_* ((U_{g,1}^n)_\bQ)) \to \pi_{2n}((U_{g,}^n)_\bQ) 
\end{split}
\]
derived from \eqref{fibresequence}. The proof of Lemma \ref{lem:pinullhautrelbound} shows that $\pi_1 (B \hAut_\partial ((U_{g,1}^n)_\bQ))  \to \pi_1 (B \hAut_* ((U_{g,1}^n)_\bQ))$ is injective, so we obtain a short exact sequence 
\begin{equation}\label{eqn:pi2BhautVg1n}
0 \to \pi_2 (B \hAut_\partial ((U_{g,1}^n)_\bQ))\to \pi_2 (B \hAut_* ((U_{g,1}^n)_\bQ)) \to \pi_{2n+1} ((U_{g,1}^n)_\bQ) \to 0. 
\end{equation}
The vector space $\pi_{2n+1} ((U_{g,1}^n)_\bQ)$ is generated by the Whitehead products $[y_i,y_j]$, modulo the relations that are universally satisfied by Whitehead products of two elements of degree $n+1$, that is
\[
[y_i,y_j]= (-1)^{(n+1)^2} [y_j,y_i].
\]
It follows that
\begin{equation}\label{eqn:relationwhiteheadprofutvz}
\pi_{2n+1} ((U_{g,1}^n)_\bQ) \cong
\begin{cases}
S^2 (N(g)^\vee) & n \; \text{odd}\\
\Lambda^2 (N(g)^\vee) & n \; \text{even}. 
\end{cases}
\end{equation}
The sequence \eqref{eqn:pi2BhautVg1n} is a sequence of $\pi_0 (\hAut_\partial ((U_{g,1}^n)_\bQ))=\GL(N(g))$-modules, and as $S^2 (N(g)^\vee)$, as well as $\Lambda^2 (N(g)^\vee)$ are irreducible $\GL(N(g))$-modules, the second map in \eqref{eqn:pi2BhautVg1n} must, up to multiplication by a nonzero constant, agree with the natural projection from $ \Hom (N(g);N(g)^\vee) = (N(g)^\vee)^{\otimes 2}$ to $S^2 (N(g)^\vee)$ or $\Lambda^2 (N(g)^\vee)$. Therefore, the kernel of these maps are as asserted.
\end{proof}

\subsection{The spectral sequence for tangential homotopy automorphisms}\label{subsec:computationdifferential}

Lemma \ref{lem:pinullhautrelbound} (3) and Theorem \ref{thm:blockdiffsversusmappingspace} prove the following. 

\begin{cor}\label{cor:finishpreliminarycalculuatin}
The algebra map 
\[
H^* (\hq{\map_\partial ((U_{g,1}^n)_\bQ;BO_\bQ)^0}{\hAut_\partial((U_{g,1}^n)_\bQ)_\bZ}) \to H^* (B \blockdiff_\partial (U_{g,1}^n))
\]
is surjective if $n \geq 3$, and the kernel is the ideal generated by the classes $\kappa_{L_m}$, $m \in \bN$.\qed
\end{cor}

Hence we need to calculate the cohomology of $\hq{\map_\partial ((U_{g,1}^n)_\bQ;BO_\bQ)^0}{\hAut_\partial((U_{g,1}^n)_\bQ)_\bZ}$, and we shall approach this through the spectral sequences for the two fibrations 
\begin{equation}\label{eqn:firstkeyfibresequence}
\hq{\map_\partial ((U_{g,1}^n)_\bQ;BO_\bQ)^0}{\hAut_\partial((U_{g,1}^n)_\bQ)^{\id}} \to \hq{\map_\partial ((U_{g,1}^n)_\bQ;BO_\bQ)^0}{\hAut_\partial((U_{g,1}^n)_\bQ)_\bZ} \to B \GL(N(g)_\bZ),
\end{equation}
(where we identified the base space using $\pi_0 (\hAut_\partial((U_{g,1}^n)_\bQ)_\bZ) \cong \GL(N(g)_\bZ)$ which is valid when $ n \geq 3$ by Lemma \ref{lem:pinullhautrelbound}), and
\begin{equation}\label{eqn:secondkeyfibresequence}
\map_\partial ((U_{g,1}^n)_\bQ;BO_\bQ)^0 \to \hq{\map_\partial ((U_{g,1}^n)_\bQ;BO_\bQ)^0}{\hAut_\partial((U_{g,1}^n)_\bQ)^{\id}} \to B \hAut_\partial((U_{g,1}^n)_\bQ)^{\id}. 
\end{equation}

Combining our work so far, the $E_2$-term of the spectral sequence of \eqref{eqn:secondkeyfibresequence} is readily computed in a range of degrees. Let us first dispose of the special case $n=4$, which is now easy.

\begin{proof}[Proof of Theorem \ref{thm:blockdiffs} in the case $n=4$]
For $n=4$, Theorem \ref{thm:blockdiffs} just asserts that $H^1 (B  \blockdiff_\partial (U_{g,1}^4);\bQ)=0$ for large enough $g$. Corollary \ref{cor:cohomology-mappingspace} shows that $H^1 (\map_\partial (U_{g,1}^4);\bQ)=0$. Hence by the spectral sequence for \eqref{eqn:secondkeyfibresequence}, we obtain $H^1 (\hq{\map_\partial ((U_{g,1}^n)_\bQ;BO_\bQ)^0}{\hAut_\partial((U_{g,1}^n)_\bQ)^{\id}};\bQ)=0$. As $H^1 (B \GL_g (\bZ);\bQ)=0$ for all $g$, the spectral sequence for \eqref{eqn:firstkeyfibresequence} proves the claim.
\end{proof}

The case $n \geq 5$ is more complicated; let us introduce some notation. 
\begin{notation}\label{notation.algebraicpicturespectralsequence}
\begin{enumerate}
\item We denote 
\[
L^2 (N(g)):= 
\begin{cases}
S^2 (N(g)) & n\; \text{even}\\
\Lambda^2 (N(g)) & n \; \text{odd}.
\end{cases}
\]
\item We fix an integer $M$, and define graded vector spaces 
\[
V(n):= \bigoplus_{4m-2n-1>0, m \leq M} \bQ[4m-2n-1], 
\]
\[
U(n) := \bigoplus_{4m-n-1>0, m \leq M} \bQ [4m-n-1]
\]
and 
\[
W(n):= \bigoplus_{4m-n>0, m \leq M} \bQ[4m-n]. 
\]
\end{enumerate}
\end{notation}

The precise choice of $M$ does not play a role as long as it is large enough; see Proposition \ref{prop:spectralsequence-isomorphism} below. We do not allow $M=\infty$ since we need $V(n)$, $U(n)$ and $W(n)$ to be finite-dimensional in order to apply the results of \S \ref{subsec:specialinvariantcalculation}. The role these vector spaces play is explained by the following construction.

\begin{defn}
Let $n \geq 6$ be even. We define a map of bigraded algebras
\begin{equation}\label{eqn:spectralsequenceneven}
S^* (S^2 (N(g))) \otimes \Lambda^* (V(n)) \otimes S^* (N(g) \otimes W(n)) \otimes \Lambda^* (N(g)^\vee \otimes U(n)) \to E^{*,*}_2
\end{equation}
to the $E_2$-term of the spectral sequence of \eqref{eqn:secondkeyfibresequence}, as follows. 
\begin{itemize}
\item On $S^2 (N(g))$, which is in bidegree $(2,0)$, it is the isomorphism $S^2 (N(g)) \cong H^2 (B \hAut_\partial ((U_{g,1}^n)_\bQ)^\id$ from Proposition \ref{prop:rationalhomotopyhautug1n}. 
\item On $V(n)_{4m-n-1} =\bQ$, which is in bidegree $(0,4m-n-1)$, it is the map which sends $1$ to the class $\lambda_{[U_{g,1}^n],L_m} \in H^{4m-2n-1}(\map_{\partial} ((U_{g,1}^n)_\bQ;BO_\bQ)^0) = E_2^{0,4m-2n-1}$ (using the notations introduced in \S \ref{subsec:lowdimpistar} and Definition \ref{defn:lambdaclass}). 
\item On $N(g) \otimes W(n)_{4m-n}= N(g)$, which is in bidegree $(0,4m-n)$, it is the map which sends $a \in N(g)$ to $\lambda_{a,L_m} \in H^{4m-n}(\map_{\partial} ((U_{g,1}^n)_\bQ;BO_\bQ)^0) = E_2^{0,4m-n}$. 
\item On $N(g)^\vee \otimes U(n)_{4m-n-1}=N(g)^\vee$, which is in bidegree $(0,4m-n-1)$, it sends $b \in N(g)^\vee$ to $\lambda_{b,L_m} \in H^{4m-n-1}(\map_{\partial} ((U_{g,1}^n)_\bQ;BO_\bQ)^0) = E_2^{4m-n-1}$. 
\end{itemize}
We similarly define for odd $n\geq 5$ a map 
\begin{equation}\label{eqn:spectralsequencenodd}
S^* (\Lambda^2 (N(g))) \otimes \Lambda^* (V(n)) \otimes \Lambda^* (N(g) \otimes W(n)) \otimes S^* (N(g)^\vee \otimes U(n)) \to E^{*,*}_2
\end{equation}
by the analogous formulas.
\end{defn}

\begin{prop}\label{prop:spectralsequence-isomorphism}
The maps \eqref{eqn:spectralsequenceneven} and \eqref{eqn:spectralsequencenodd} are $\pi_0 (\hAut_\partial ((U_{g,1}^n)_\bQ)) \cong \GL(N(g))$-equivariant. If $n \geq 5$, they are isomorphisms in bidegrees $(p,q)$ with $p \leq n-3$ and $q \leq 4M-2n+2$.
\end{prop}

\begin{proof}
Equivariance follows from Corollary \ref{cor:actiononcohomologymappingspace}, from the naturality of the $\lambda$-classes that was recorded after Definition \ref{defn:lambdaclass}, and from Proposition \ref{prop:rationalhomotopyhautug1n}.

For the isomorphism statement, first observe that the dual of the equivariant isomorphism $H_2 (B \hAut_\partial((U_{g,1}^n)_\bQ)^{\id};\bQ) \cong L^2 (N(g))^\vee $ from Proposition \ref{prop:rationalhomotopyhautug1n} extends to an equivariant algebra map 
\[
S^* (L^2 (N(g))) \to H^* (B \hAut_\partial((U_{g,1}^n)_\bQ)^{\id};\bQ)
\]
which is an isomorphism up to degree $n-3$, also by Proposition \ref{prop:rationalhomotopyhautug1n} and the well-known computation of the cohomology of rational Eilenberg--Mac-Lane spaces. This observation, together with Corollary \ref{cor:cohomology-mappingspace}, implies that the maps in question are isomorphisms.
\end{proof}

\begin{prop}\label{prop:the:E2differential}
Assume $n \geq 5$ and $g \geq 2$. With respect to the isomorphism of Proposition \ref{prop:spectralsequence-isomorphism}, the $d_2$-differential in the spectral sequence of the fibre sequence \eqref{eqn:secondkeyfibresequence} is given on generators as follows.
\begin{enumerate}
\item On $L^2 (N(g))$, it is zero. 
\item On $V(n)$, it is zero. 
\item On $N(g)^\vee \otimes U(n)$, it is zero. 
\item On $N(g) \otimes W(n)$, it is of the form 
\[
N(g) \otimes W(n) \stackrel{p \otimes S}{\to} L^2 (N(g)) \otimes N(g)^\vee \otimes U(n)
\]
where $p: N(q) \to L^2 (N(g)) \otimes N(g)^\vee$ is adjoint to the projection $N(g)^{\otimes 2} \to L^2 (N(g))$ and $S: W(n) \to U(n)$ is a degree $-1$ map whose restriction $W(n)_{4m-n} \to U(n)_{4m-n-1}$ is an isomorphism unless $4m-n-1=0$.
\end{enumerate}
\end{prop}

\begin{proof}
Item (1) holds for degree reasons. For item (2), note that the inclusion map $\map_\partial ((U_{g,1}^n)_\bQ;BO_\bQ)^0 \to \hq{\map_\partial ((U_{g,1}^n)_\bQ;BO_\bQ)^0}{\hAut_\partial ((U_{g,1}^n)_\bQ)^\id}$ pulls back $\kappa_{L_m}$ to the class $\lambda_{[U_{g,1}^n],L_m}$, by Lemma \ref{lem:kappaclassaslambdaclass}. Therefore, the latter class is a permanent cycle, which verifies (2).

For the other two claims, we use the $\hAut_\partial ((U_{g,1}^n)_\bQ)$-equivariant map
\[
\map_\partial ((U_{g,1}^n)_\bQ;BO_\bQ)^0 \to \prod_{m \geq 1} \map_* ((U_{g,1}^n)_\bQ;K(\bQ,4m))^0
\]
given by the cohomology classes $(L_m)_{m \geq 1}$ and a forgetful map. Since the latter map, combined with the projection onto the $m$th factor, pulls back $\lambda_{a,u_m} \in H^* (\map_\partial ((U_{g,1}^n)_\bQ;K(\bQ,4m))^0;\bQ)$ to $\lambda_{a,L_m}$, the proposition follows from the next Lemma.
\end{proof}
\begin{lem}\label{claim:differentialidentificationlemmaproof}
The $E_2$-term of the spectral sequence of the fibration 
\begin{equation}\label{eqn:toymodel-mappingspacefibration}
\hq{\map_* ((U_{g,1}^n)_\bQ;K(\bQ,k))^0}{\hAut_\partial ((U_{g,1}^n)_\bQ)^\id} \to B \hAut_\partial ((U_{g,1}^n)_\bQ)^\id
\end{equation}
is (in the columns up to degree $n-3$, by Proposition \ref{prop:rationalhomotopyhautug1n}) given by 
\[
\bF (L^2 (N(g)))[2,0] \oplus N(g)[0,k-n]\oplus N(g)^\vee[0,k-n-1]) 
\]
(if $k-n-1>0$; if $k-n-1=0$, the last summand is dropped and if $k-n \leq 0$, the last two summands are dropped). The differential $d_2$ vanishes on $N(g)^\vee[0,k-n-1]$, and on $N(g)[0,k-n]$, it is, up to a sign, the map adjoint to the projection $ N(g)^{\otimes 2} \to L^2 (N(g))$, if $k-n-1>1$.
\end{lem}

\begin{proof}
The differential vanishes on $N(g)^\vee[0,k-n-1]$ for degree reasons, and because the fibration \eqref{eqn:toymodel-mappingspacefibration} has a section. The claim about $N(g)[0,k-n]$ is true for the same reason if $k-n-1=0$, and so we may suppose $k-n-1>0$. We pair the spectral sequence of \eqref{eqn:toymodel-mappingspacefibration} with the spectral sequence of the universal fibration 
\begin{equation}\label{eqn:universlfirbationsasamodule}
\hq{(U_{g,1}^n)_\bQ}{\hAut_\partial ((U_{g,1}^n)_\bQ)^\id} \to B \hAut_\partial ((U_{g,1}^n)_\bQ)^\id . 
\end{equation}
We write $\tilde{E}_{*}^{*,*}$ for the rational cohomological spectral sequence of the latter. The $E_2$-term is of the form
\[
\tilde{E}_2 = S^* (L^2 N(g)) \otimes H^* (U_{g,1}^n;\bQ)
\]
(in the columns up to degree $(n-3)$, by Proposition \ref{prop:rationalhomotopyhautug1n}). 
Let $x_{ij} \in L^2 (N_g)$ the image of $a_i \otimes a_j \in N_g^{ \otimes 2}$ under the projection map; these elements satisfy $x_{ji} = (-1)^n x_{ij}$. 
By Proposition \ref{prop:rationalhomotopyhautug1n}, the differential $\tilde{d}_2: \tilde{E}_2^{0,n+1} \to \tilde{E}_2^{2,n}$ is given on basis elements by the formulas
\begin{equation}\label{eqn:d2betai}
\tilde{d}_2 (\beta_i)= \sum_j  x_{ij} \otimes \alpha_j
\end{equation}
and
\begin{equation}\label{eqn:d2alphai}
\tilde{d}_2 (\alpha_i)=0. 
\end{equation}
To transfer this knowledge about $\tilde{E}_*^{*,*}$ to information about $E_{*}^{*,*}$, observe that the evaluation map $\ev: U_{g,1}^n \times \map_\partial (U_{g,1}^n;K(\bQ,k))^0 \to K(\bQ,k)$ is $\hAut_\partial (U_{g,1}^n)$-equivariant and hence induces a map 
\[
\hq{( U_{g,1}^n \times \map_\partial (U_{g,1}^n;K(\bQ,k))^0)}{\hAut_\partial (U_{g,1}^n)} \to K(\bQ,k). 
\]
It follows that $\ev^* u_k \in H^k (U_{g,1}^n \times \map_\partial (U_{g,1}^n;K(\bQ,k))^0)$ lies in the image of the map induced by the inclusion 
\[
 U_{g,1}^n \times \map_\partial (U_{g,1}^n;K(\bQ,k))^0 \to \hq{( U_{g,1}^n \times \map_\partial (U_{g,1}^n;K(\bQ,k))^0)}{\hAut_\partial (U_{g,1}^n)}
\]
of the fibre in the total space. This forces $\ev^* u_k \in \hat{E}_2^{0,k}$ to be a permanent cycle in the spectral sequence $\hat{E}^{*,*}_*$ of 
\[
\hq{( U_{g,1}^n \times \map_\partial (U_{g,1}^n;K(\bQ,k))^0 )}{\hAut_\partial ((U_{g,1}^n)_\bQ)^\id} \to B \hAut_\partial ((U_{g,1}^n)_\bQ)^\id,
\]
so $\hat{d}_2 (\ev^* u_k)=0$. By formula \eqref{eqn:formulaevz}
\[
\ev^* u_k = \sum_{i=1}^g \alpha_i\times \lambda_{a_i,u_k}  + \sum_{i=1}^g \beta_i \times \lambda_{b_i,u_k}. 
\]
Using \eqref{eqn:d2betai} and \eqref{eqn:d2alphai}, and that $d_2 (\lambda_{b_i,u_k})=0$ for degree reasons, we obtain 
\[
0= \sum_{i=1}^g \hat{d}_2 (\alpha_i \times \lambda_{a_i,u_k})  + \hat{d}_2 (\beta_i \times \lambda_{b_i,u_k}) = 
\]
\[
=\sum_{i=1}^g (d_2 \alpha_i) \times  \lambda_{a_i,u_k}  +  \sum_{i=1}^g (-1)^n \alpha_i \times  (d_2 \lambda_{a_i,u_k})+ 
\sum_{i=1}^g (\tilde{d}_2 \beta_i ) \times \lambda_{b_i,u_k} +  \sum_{i=1}^g (-1)^{n+1}\beta_i \times (\tilde{d}_2 \lambda_{b_i,u_k})  = 
\]
\[
= \sum_{j=1}^g (-1)^n \alpha_j \times (d_2 \lambda_{a_j,u_k})+  \sum_{i,j=1}^g x_{ij} \otimes (\alpha_j\times\lambda_{b_i,u_k}). 
\]
Comparing coefficients leads to the identity
\[
d_2 \lambda_{a_j,u_k} = (-1)^{n+1}\sum_{i=1}^g x_{ij} \otimes \lambda_{b_i,u_k}
\]
which finishes the proof.
\end{proof}

Propositions \ref{prop:spectralsequence-isomorphism} and \ref{prop:the:E2differential} were the goals of this section. Before we can use these to evaluate the two spectral sequences for \eqref{eqn:firstkeyfibresequence} and \eqref{eqn:secondkeyfibresequence}, we need to switch gears and introduce some more algebraic background. 

\section{A representation-theoretic calculation}\label{sec:representationtheory}

\subsection{Generalities}\label{subsec:representationtheory}

\subsubsection*{Rational representations}
Let $\bK$ be a field of characteristic $0$, and let $N$ be a finite-dimensional $\bK$-vector space. Recall that a representation $\rho:\GL (N) \to \GL(W)$ on some other $\bK$-vector space is \emph{rational} if the matrix entries (after choice of a basis) of $\rho(g)$ are polynomial functions of the matrix entries of $g$ and of $\det(g)^{-1}$. A similar definition applies to representations of $\SL(N)$. 

Let $\bK \subset \bL$ be a field extension, let $N$ and $W$ be $\bK$-vector spaces and let $\rho: \GL(N) \to \GL(W)$ be rational. Then $\rho$ extends to a rational representation $\rho_\bL: \GL(N_\bL) \to \GL(W_\bL)$. A similar statement is true for $\SL(N)$. 

If there is $w \in \bZ$ such that each element $\lambda \in \bK^\times\subset \GL(N)$ in the centre acts by multiplication with $\lambda^w$ on $W$, we say that $\rho$ has \emph{weight} $w$. 
We shall use the notation $T^{k,l}(N):= N^{\otimes k}\otimes (N^\vee)^{\otimes l}$. 

The rational representations are described in terms of Schur functors, whose definition we briefly recall. Let $\cP_n$ be the set of partitions of $n$, thought of as Young diagrams, as explained in e.g. \cite[\S 4.1]{FultonHarris}. If $\lambda$ is a partition of $n$, we also write $|\lambda|=n$. The \emph{height} $\height(\lambda)$ of $\lambda \in \cP_n$ is the number of rows of $\lambda$. To a partition $\lambda$ of $n$, there is associated the Young symmetrizer $c_\lambda \in \bQ[\Sigma_n] \subset \bK[\Sigma_n]$ and the associated irreducible $\Sigma_n$-representation $M_\lambda$ (over $\bK$; it is irreducible since it is irreducible when the scalars are extended to the algebraic closure $\overline{\bK}$). The tensor power $T^{n,0}(N)$ is a canonical $\GL(N) \times \Sigma_n$-representation, and the Schur functor $S_\lambda(N)$ is defined as the $\GL(N)$-representation 
\[
S_\lambda (N):= c_\lambda \cdot T^{n,0}(N). 
\]
The theory of rational representations can be summarized in the following result.

\begin{thm}\label{thm:representationtheoryGL(N)}
Let $N$ be a finite-dimensional $\bK$-vector space. 
\begin{enumerate}
\item As $\GL(N) \times \Sigma_n$-modules, we have
\[
T^{n,0}(N)= \bigoplus_{\lambda \in \cP_n} S_\lambda (N) \otimes M_\lambda. 
\]
If $\height(\lambda)> \dim (N)$, then $S_\lambda(N)=0$, and if $\height(\lambda) \leq \dim (N)$, $S_\lambda(N)$ is nonzero and irreducible. The Schur functors $S_\lambda(N)$ and $S_\mu(N)$ are isomorphic only if they both vanish or if $\lambda=\mu$. The Schur functors remain irreducible when viewed as $\SL(N)$-modules. 
\item After taking the tensor product with a suitable power of the determinant representation $\det (N)$, each $\GL(N)$-representation embeds into a direct sum $\bigoplus_j T^{n_j,0}(N)$ for suitable $n_j$. The same is true for $\SL (N)$-representations, without tensoring with $\det (N)$.
\item Rational $\GL(N)$- and $\SL(N)$-representations are completely reducible. The Schur functors give a complete list of the irreducible $\SL(N)$-representations, and the Schur functors, tensorized with powers of the determinant representation, a complete list of the irreducible $\GL(N)$-representations.
\item Let $N,V,W$ be $\bK$ vector spaces, let $\GL(N) \to \GL(V)$ and $\GL(N) \to \GL(W)$ be rational representations and let $\bK \subset \bL$ be a field extension. Then $W$ is irreducible if and only if $W_\bL$ is irreducible; $V$ and $W$ are isomorphic if and only if $V_\bL$ and $W_\bL$ are isomorphic, and furthermore 
\[
(W^{\GL(N)})_\bL=(W_\bL)^{\GL(N_\bL)}. 
\]
The same is true for $\SL$ in place of $\GL$.
\end{enumerate}
\end{thm}
\begin{proof}[References]
Statements (1)--(3) are well-known when $\bK$ is algebraically closed; \cite{Procesi} and \cite{FultonHarris} are convenient textbook references. We explain how to generalize the results to arbitrary $\bK$. 

(1) The decomposition is shown for algebraically closed fields (of characteristic $0$) in \cite[Theorem 9.3.1.4]{Procesi}, but it exists over any ground field of characteristic $0$ because the Young symmetrizers have rational coefficients. For a field extension $\bK \subset \bL$, we have $S_\lambda(N_\bL)\cong (S_\lambda(N))_\bL$ as $\GL(N_\bL)$-representations. Hence the other statements follow from the algebraically closed case, which is done in \cite[Theorem 6.3.1]{FultonHarris} and \cite[Theorem 9.8.1]{Procesi}.

(2) The proof of \cite[Lemma 7.1.4]{Procesi} does not use that $\bK$ is algebraically closed. 

(3) The first part follows from (2), since $T^{n,0}(N)$ is completely reducible, and by generalities on completely reducible representations, e.g. \cite[Proposition 3.1.4]{Etingof}. The second part follows from the algebraically closed case which is \cite[Theorem 9.8.1]{Procesi}. 

(4) This is clear from the other items. 
\end{proof}

\begin{defn}\label{defn:load}
Let $\rho: \SL(N) \to \GL(W)$ be a rational representation. We say that $\rho$ (or $W$) has \emph{load $\leq n$} if each irreducible summand of $W$ is a direct summand of $T^{k,l}(N)$, for some $k,l$ with $k+l \leq n$.
\end{defn}

\subsubsection*{Invariants on tensor spaces}
We need the fundamental theorem of invariant theory of $\GL(N)$. We can identify $T^{m,m}(N) \cong \End(N)^{\otimes m} \cong \End(N^{\otimes m})$. There is a natural map
\begin{equation}\label{eqn:fundamentalinvariantmap}
\sigma_{N,m}: \bK[\Sigma_m] \to \End(N^{\otimes m})^{\GL(N)} \cong T^{m,m}(N)^{\GL(N)}
\end{equation}
given by the $\Sigma_m$-action permuting the factors. In terms of a basis $(a_1,\ldots,a_g)$ of $N$ and the dual basis $(a^1,\ldots,a^g)$ of $N^\vee$, $\sigma_{N,m}$ is given by the formula
\begin{equation}\label{eqn:formula-fundmanetalinvariant}
s \mapsto \sum_{i_1, \ldots,i_m} a_{i_1} \otimes \ldots \otimes a_{i_m} \otimes a^{i_{s^{-1}(1)}} \otimes \ldots \otimes a^{i_{s^{-1}(m)}} \in T^{m,m}(N). 
\end{equation}

\begin{thm}\label{thm:fundamentaltheoreminvarianttheory}
Let $N$ be a finite-dimensional $\bK$-vector space. The map $\sigma_{N,m}$ is surjective, and it is also injective if $m \leq \dim (N)$.
\end{thm}

Surjectivity is the content of the first fundamental theorem which is proven in e.g. \cite[\S 9.1.2]{Procesi}, \cite[Theorem 9.1.2]{Loday} or \cite[Theorem 5.3.1]{GoodmanWallach} (and is a key ingredient for Theorem \ref{thm:representationtheoryGL(N)}). Injectivity is the second fundamental theorem. The treatment of that result in \cite{Procesi} or \cite{GoodmanWallach} has a slightly different layout; the version as stated above is shown in \cite[Theorem 9.1.3]{Loday}. 

\begin{lem}\label{lem:weightargument}
Let $N$ be a finite-dimensional $\bK$-vector space. Then 
\begin{enumerate}
\item $T^{k,l}(N)^{\GL(N)}=0$ unless $k=l$. 
\item $T^{k,l}(N)^{\SL(N)}=0$ unless $\dim (N)$ divides $k-l$. 
\item $T^{k,k}(N)^{\SL(N)} = T^{k,k}(N)^{\GL(N)}$. 
\end{enumerate}
\end{lem}

\begin{proof}
(1) is easy; just look at the action of a scalar matrix. (2) Use Theorem \ref{thm:representationtheoryGL(N)} (4) to replace $\bK$ by its algebraic closure. Then $\SL(N)$ contains a primitive $g$th root of unity $\zeta_g$, where $g=\dim (N)$, which acts by $\zeta_g^{k-l}$ on $T^{k,l}(N)$. Hence $T^{k,l}(N)^{\SL(N)}=0$ unless $g$ divides $k-l$.
(3) Use Theorem \ref{thm:representationtheoryGL(N)} (4) to replace $\bK$ by its algebraic closure. The canonical isomorphism $T^{k,k}(N) \cong \End(N^{\otimes k})$ identifies $T^{k,k}(N)^{\GL(N)}$ with the commutant algebra of the image of $\rho_{\GL}:\bK[\GL(N)] \to \End(N^{\otimes k})$.
Similarly, $T^{k,k}(N)^{\SL(N)}$ is the commutant of the image of $\rho_{\SL}:\bK[\SL(N)] \to \End(N^{\otimes k})$, so it is enough to prove that $\rho_\GL$ and $\rho_\SL$ have the same image. Each element $A \in \GL(N)$ can be written as $A= \lambda B$ with $\lambda \in \bK^\times$ and $B \in \SL(V)$ (here we are using that $\bK$ is algebraically closed). It follows that $\rho_\GL (A) = \lambda^k \rho_\SL(B) \in \im (\rho_\SL)$. 
\end{proof}

\subsection{A special invariant calculation}\label{subsec:specialinvariantcalculation}

In this section, we carry out a representa\-tion-theoretic calculation that will be used later on. The main ideas for the proof were communicated to us by Jerzy Weyman, and we thank him for allowing to reproduce his argument here.

Let $\bK$ be a field of characteristic $0$, and let $N$, $W$ and $U$ be finite-dimensional $\bK$-vector spaces. We let $g:=\dim (N)$ and fix a basis $(a_1,\ldots,a_g)$ of $N$, with dual basis $(a^1,\ldots,a^g)$. We consider the algebras
\[
A:= S^* (S^2 (N)) \otimes S^* (N \otimes W) \otimes \Lambda^* (N^\vee \otimes U)
\]
and 
\[
C:= S^* (\Lambda^2(N)) \otimes \Lambda^* (N \otimes W) \otimes S^* (N^\vee \otimes U).
\]
These algebras have obvious trigradings and actions of the group $\GL(N) \times \GL(W) \times \GL(U)$. We want to determine the algebras $A^{\GL(N)}$ and $C^{\GL(N)}$ of $\GL(N)$-invariants. There are some obvious invariants. We define 
\[
\varphi_A: \Lambda^2 (U) \to (S^1 (S^2 (N)) \otimes S^0 (N \otimes W) \otimes \Lambda^2 (N^\vee \otimes U))^{\GL(N)} =A_{1,0,2}^{\GL(N)} 
\]
by 
\[
u_1 \wedge u_2 \mapsto \sum_{i,j=1}^g (a_i \cdot a_j)\otimes 1 \otimes ((a^i \otimes u_1) \wedge (a^j \otimes u_2)), 
\]
and we define
\[
\psi_A: W \otimes U \to (S^0 (S^2 (N)) \otimes S^1 (N \otimes W) \otimes \Lambda^1 (N^\vee \otimes U))^{\GL(N)}= A_{0,1,1}^{\GL(N)}  
\]
by 
\[
w \otimes u \mapsto \sum_{i=1}^g 1 \otimes (a_i \otimes w) \otimes (a^i \otimes u). 
\]
For $x_0,x_1 \in \Lambda^2 (U)$ and $y_0,y_1 \in W \otimes U$, the relations 
\[
\varphi_A (x_0)\varphi_A (x_1) = \varphi_A (x_1)\varphi_A (x_0),
\]
\[
\psi_A (y_0) \psi_A (y_1)=-\psi_A (y_1) \psi_A (y_0)
\]
and 
\[
\varphi_A (x_i) \psi_A (y_j)= \psi_A (y_j)\varphi_A (x_i) 
\]
hold, and these imply that $\varphi_A \oplus \psi_A$ extends to an algebra map 
\[
G: S^* (\Lambda^2 (U))\otimes \Lambda^* (W \otimes U) \to A^{\GL(N)}.
\]

Similarly let us define 
\[
\varphi_C: \Lambda^2 (U) \to (S^1 (\Lambda^2 (N)) \otimes \Lambda^0 (N \otimes W) \otimes S^2 (N^\vee \otimes U))^{\GL(N)} =C_{1,0,2}^{\GL(N)} 
\]
by 
\[
u_1 \wedge u_2 \mapsto \sum_{1 \leq i <j\leq g} (a_i \wedge a_j)\otimes 1 \otimes ((a^i \otimes u_1) \cdot (a^j \otimes u_2))
\]
and 
\[
\psi_C: W \otimes U \to (S^0 (\Lambda^2 (N)) \otimes \Lambda^1 (N \otimes W) \otimes S^1 (N^\vee \otimes U))^{\GL(N)} = C_{0,1,1}^{\GL(N)}  
\]
by 
\[
w \otimes u \mapsto \sum_{i=1}^g 1 \otimes (a_i \otimes w) \otimes (a^i \otimes u). 
\]
A similar argument as above shows that $\varphi_C \oplus \psi_C$ extends to an algebra map 
\[
H: S^* (\Lambda^2 (U)) \otimes \Lambda^* (W \otimes U)  \to C^{\GL(N)}. 
\]
\begin{prop}\label{prop:keylemma-invarianttheory}
We have $A_{p,q,r}^{\GL(N)}= C_{p,q,r}^{\GL(N)}=0$ unless $2p+q-r=0$, and $A_{p,q,r}^{\SL(N)}= C_{p,q,r}^{\SL(N)}=0$ unless $g:= \dim (N)$ divides $2p+q-r$. The maps 
\[
S^p (\Lambda^2 (U)) \otimes \Lambda^q (W \otimes U) \stackrel{G}{\to} A_{p,q,2p+q}^{\GL(N)} \subset A_{p,q,2p+q}^{\SL(N)}
\]
and 
\[
S^p (\Lambda^2 (U)) \otimes \Lambda^q (W \otimes U) \stackrel{H}{\to} C_{p,q,2p+q}^{\GL(N)} \subset  C_{p,q,2p+q}^{\SL(N)}
\]
are surjective, and isomorphisms if $2p+q \leq g$. 
\end{prop}

We use Proposition \ref{prop:keylemma-invarianttheory} in conjunction with Proposition \ref{prop:invariantsSLZ} below. To state it, let us assume that $\bK=\bQ$ and that $N$ has an \emph{integral form}, i.e. a subgroup $N_\bZ \subset N$ such that $N_\bZ \otimes \bQ= N$. In that case, we have the subgroup $\SL(N_\bZ)\subset \SL(N)$ of automorphisms preserving $N_\bZ$; note that $\SL(N_\bZ)\cong \SL_g (\bZ)$. 

\begin{prop}\label{prop:invariantsSLZ}
If $\bK=\bQ$ and $N$ has an integral form $N_\bZ$, the inclusions 
\[
A_{p,q,2p+q}^{\SL(N)} \subset A_{p,q,2p+q}^{\SL(N_\bZ)}
\]
and 
\[
C_{p,q,2p+q}^{\SL(N)} \subset C_{p,q,2p+q}^{\SL(N_\bZ)}
\]
are equalities.
\end{prop}

\begin{proof}[Proof of Proposition \ref{prop:keylemma-invarianttheory}, surjectivity]
The first sentence is a straightforward application of Lemma \ref{lem:weightargument}. It also follows from Lemma \ref{lem:weightargument} that $A_{p,q,2p+q}^{\GL(N)} = A_{p,q,2p+q}^{\SL(N)}$ and $C_{p,q,2p+q}^{\GL(N)} = C_{p,q,2p+q}^{\SL(N)}$. So we must only show that $G$ and $H$ are isomorphisms, and we start with surjectivity.

Consider the case of $G$ first. We establish a commutative diagram
\begin{equation}\label{diag:invarianttheorysurjectiviy}
\xymatrix{
\bQ[\Sigma_{2p+q}] \otimes W^{\otimes q} \otimes U^{\otimes 2p+q} \ar[r]^-{F} \ar[d]^{Q} & (N^{\otimes 2p+q} \otimes (N^\vee)^{\otimes 2p+q}\otimes W^{\otimes q} \otimes U^{\otimes 2p+q})^{\GL(N)} \ar[d]^{S}\\
S^p (\Lambda^2 (U)) \otimes \Lambda^q (W \otimes U) \ar[r]^{G} & A_{p,q,2p+q}^{\GL(N)},
}
\end{equation}
and show that $F$ and $S$ are surjective. The map $S$ is the restriction of the ($\GL(N)$-equivariant) quotient map 
\[
N^{\otimes 2p+q} \otimes (N^\vee)^{\otimes 2p+q}\otimes W^{\otimes q} \otimes U^{\otimes 2p+q} \cong N^{\otimes 2p} \otimes (N \otimes W)^{\otimes q}\otimes (N^\vee\otimes U)^{\otimes 2p+q}\to A_{q,p,2p+q}
\]
to the invariant subspace. As the quotient map is surjective and as the category of rational representations of $\GL(N)$ is semisimple, $S$ is surjective. 

The map $F$ is the tensor product of $\sigma_{N,2p+q}$ and the identity on $W^{\otimes q} \otimes U^{\otimes 2p+q}$; it sends the element
\[
s \otimes w_1 \otimes \ldots \otimes w_q  \otimes u_1 \otimes \ldots \otimes u_{2p+q} \in \bQ[\Sigma_{2p+q}] \otimes W^{\otimes q} \otimes U^{\otimes 2p+q}
\]
to 
\[
\sum_{i_1, \ldots, i_{2p+q}}a_{i_1} \otimes \ldots \otimes a_{i_{2p+q}}  \otimes a^{i_{s^{-1}(1)}} \otimes\ldots \otimes  a^{i_{s^{-1}(2p+q)} } \otimes w_1 \otimes \ldots \otimes w_q  \otimes u_1 \otimes \ldots \otimes u_{2p+q} ,
\]
and is clearly $\GL(N)$-invariant. Formula \eqref{eqn:formula-fundmanetalinvariant} shows that (upon identification of its target) $F$ is the tensor product of the map $\sigma_{N,2p+q}$ defined in \eqref{eqn:fundamentalinvariantmap} and the identity on $W^{\otimes q} \otimes U^{\otimes 2p+q}$. Therefore, by Theorem \ref{thm:fundamentaltheoreminvarianttheory}, $F$ is surjective.

We define the map $Q$ by 
\[
Q(s \otimes w_1 \otimes \ldots \otimes w_q \otimes u_1 \otimes \ldots \otimes u_{2p+q}):= 
\]
\[
 \sgn (s) (u_{s(1)} \wedge u_{s(2)})  \cdots (u_{s(2p-1)} \wedge u_{s(2p)}) \otimes (w_1 \otimes u_{s(2p+1)}) \wedge \ldots \wedge (w_q \otimes u_{s(2p+q)}). 
\]
It remains to be proven that \eqref{diag:invarianttheorysurjectiviy} commutes, but this follows from 
\[
SF (s \otimes w_1 \otimes \ldots \otimes w_q  \otimes u_1 \otimes \ldots \otimes u_{2p+q} )= 
\]
\[
\sum_{i_1, \ldots, i_{2p+q}} (a_{i_1}\cdot a_{i_2}) \cdots (a_{i_{2p-1}}\cdot a_{i_{2p}}) \otimes ( a_{i_{2p+1}} \otimes w_1) \cdots ( a_{i_{2p+q}} \otimes w_q)  \otimes (a^{i_{s^{-1}(1)}} \otimes u_1) \wedge \ldots \wedge (a^{i_{s^{-1}(2p+q)}} \otimes u_{2p+q})
\]
and 
\[
GQ (s \otimes w_1 \otimes \ldots \otimes w_q  \otimes u_1 \otimes \ldots \otimes u_{2p+q} )= 
\]
\[
\sgn(s) \sum_{i_1,\ldots,i_{2p+q}} (a_{i_1} \cdot a_{i_2}) \cdots (a_{i_{2p-1}} \cdot a_{i_{2p}}) \otimes (a_{i_{2p+1} } \otimes w_1 ) \cdots  (a_{i_{2p+q} } \otimes w_q ) \otimes (a^{i_1} \otimes u_{s (1)}) \wedge \ldots \wedge (a^{i_{2p+q}} \otimes u_{s (2p+q)}). 
\]

This finishes the proof that $G$ is surjective. The case of $H$ is almost identical. In that case, we consider 
\begin{equation}\label{diag:invarianttheorysurjectiviy2}
\xymatrix{
\bQ[\Sigma_{2p+q}] \otimes W^{\otimes q} \otimes U^{\otimes 2p+q} \ar[r]^-{F} \ar[d]^{Q'} & (N^{\otimes 2p+q} \otimes (N^\vee)^{\otimes 2p+q}\otimes W^{\otimes q} \otimes U^{\otimes 2p+q})^{\GL(N)} \ar[d]^{S'}\\
S^p (\Lambda^2 (U)) \otimes \Lambda^q (W \otimes U) \ar[r]^{H} & C_{p,q,2p+q}^{\GL(N)}.
}
\end{equation}
The map $F$ is the same map as before. The map $Q'$ is defined just as $Q$, the only difference being that the factor $\sgn(s)$ in front of the definition of $Q$ is dropped, and the map $S'$ is again the quotient map. 
\end{proof}

For the proof of injectivity of $G$ and $H$, we need some classical results of invariant theory. Here is some notation: the conjugate partition (see \cite[\S 4.1]{FultonHarris} for this term) to $\lambda\in \cP_n$ is denoted $\tilde{\lambda} \in \cP_n$. By $\cP_{2p}^{\evrow}$, we denote the set of partitions of $2p$ with even rows, and by $\cP_{2p}^{\evcol}$ the set of partitions of $2p$ with even columns. 
The first ingredient we shall use are the Cauchy formulas which state that \cite[\S 9.6.3, p. 271]{Procesi}
\begin{equation}\label{eqn:cauchy1}
S^q (V \otimes W)= \bigoplus_{\lambda \in \cP_q} S_\lambda(V) \otimes S_\lambda(W),
\end{equation}
\cite[\S 9.8.4]{Procesi} 
\begin{equation}\label{eqn:cauchy2}
\Lambda^r (V \otimes W)= \bigoplus_{\lambda \in \cP_r} S_\lambda(V) \otimes S_{\tilde{\lambda}} (W)
\end{equation}
and \cite[\S 11.4.5]{Procesi}
\begin{equation}\label{eqn:cauchy3}
S^p (S^2 (V)) = \bigoplus_{\lambda \in \cP_{2p}^{\evrow}} S_\lambda(V), 
\end{equation}
as well as \cite[\S 11.4.5]{Procesi}
\begin{equation}\label{eqn:cauchy4}
S^p (\Lambda^2 (V)) = \bigoplus_{\lambda \in \cP_{2p}^{\evcol}} S_\lambda(V). 
\end{equation}
Furthermore, we need the formula \cite[\S 12.5.1]{Procesi}
\begin{equation}\label{eqn:littlewoodrichardson}
S_\lambda (V) \otimes S_\mu(V) = \bigoplus_{|\kappa|= |\lambda|+|\mu|} c^\kappa_{\lambda,\mu} S_\kappa (V). 
\end{equation}
The coefficients $c^\kappa_{\lambda,\mu}\in \bN_0$ are the well-known \emph{Littlewood--Richardson coefficients}. These are the structure constants of the ring $\bigwedge$ of symmetric functions (over the integers) when one takes the Schur functions $s_\lambda$, $\lambda \in \cP$, as a basis. From this, it follows that the coefficients in \eqref{eqn:littlewoodrichardson} do not depend on $V$. The symmetry
\[
c^\kappa_{\lambda,\mu} = c^\kappa_{\mu,\lambda}; 
\]
of the Littlewood--Richardson coefficients is obvious; we also need to know the relation 
\begin{equation}\label{eqn:symmetry-littlewood-richardson}
c^\kappa_{\lambda,\mu}=c^{\tilde{\kappa}}_{\tilde{\lambda},\tilde{\mu}}.
\end{equation}
To see \eqref{eqn:symmetry-littlewood-richardson}, let $\omega: \bigwedge \to \bigwedge$ be the involutive (ring) automorphism which is constructed in \cite[p.21]{Macdonald}; formula (3.8) on p.42 of \cite{Macdonald} shows that $\omega(s_\lambda)=s_{\tilde{\lambda}}$. Since the Littlewood--Richardson coefficients are the structure constants with respect to the Schur functions, \eqref{eqn:symmetry-littlewood-richardson} follows. An alternative proof of \eqref{eqn:symmetry-littlewood-richardson} can be found in \cite[p. 62]{FultonYoung}. 

\begin{proof}[Proof of Proposition \ref{prop:keylemma-invarianttheory}, injectivity]
Since we already saw that $G$ and $H$ are surjective, it suffices to show that the dimensions of the two vector spaces agree (degreewise, and in the range of degrees we claimed). It is therefore of no danger to write $S=S'$ for isomorphic representations $S$ and $S'$, and $nS := S^{\oplus n}$. 

We first turn to the map $G$. Its components are maps 
\[
S^p (\Lambda^2 U) \otimes \Lambda^q (W \otimes U) \to (S^p (S^2 (N)) \otimes S^q (N \otimes W) \otimes \Lambda^{2p+q} (N^\vee \otimes U))^{\GL(N)}. 
\]
As a $\GL(N) \times \GL(W) \times \GL(U)$-representation, we have by \eqref{eqn:cauchy3}, \eqref{eqn:cauchy1} and \eqref{eqn:cauchy2}
\[
S^p (S^2 (N)) \otimes S^q (N \otimes W) \otimes \Lambda^{2p+q} (N^\vee \otimes U) = 
\]
\[
\bigoplus_{\lambda \in \cP_{2p}^{\evrow},|\mu|=q,|\nu|=2p+q} S_\lambda(N) \otimes S_\mu(N) \otimes S_\mu(W)\otimes S_\nu(N^\vee) \otimes S_{\tilde{\nu}} (U). 
\]
By \eqref{eqn:littlewoodrichardson}, the latter is isomorphic to
\[
\bigoplus_{\lambda \in \cP_{2p}^{\evrow},|\mu|=q,|\nu|=|\kappa|=2p+q} c^{\kappa}_{\lambda,\mu} S_\kappa(N) \otimes S_\mu(W)\otimes S_\nu(N^\vee) \otimes S_{\tilde{\nu}} (U).
\]
Since the $S_\lambda(N)$ are mutually nonisomorphic irreducible $\GL(N)$-representations or trivial, we have, with $g:= \dim (N)$,
\begin{equation}\label{eqn:orthogonality-schurfunctor}
(S_\kappa(N) \otimes S_\nu(N^\vee))^{\GL(N)}\cong 
\begin{cases}
\bQ & \kappa=\nu\; \text{and} \; \height(\nu)=\height (\kappa) \leq g\\
0 & \kappa\neq \nu\; \text{or} \; \height(\nu) > g \; \text{or} \; \height(\kappa)>g.
\end{cases}
\end{equation}
Therefore
\[
(S_\kappa(N) \otimes S_\mu(W)\otimes S_\nu(N^\vee) \otimes S_{\tilde{\nu}} (U))^{\GL(N)} = 
\begin{cases}
S_\mu(W)\otimes  S_{\tilde{\nu}} (U) &  \kappa=\nu\; \text{and} \; \height(\nu)=\height (\kappa) \leq g\\
0 & \kappa\neq \nu\; \text{or} \; \height(\nu) > g \; \text{or} \; \height(\kappa)>g, 
\end{cases}
\]
and so 
\[
\begin{split}
(S^p (S^2 (N)) \otimes S^q (N \otimes W) \otimes \Lambda^{2p+q} (N^\vee \otimes U) )^{\GL(N)} & \\
 = \bigoplus_{\lambda \in \cP_{2p}^{\evrow},|\mu|=q,|\nu|=2p+q,\height(\nu)\leq g} c^{\nu}_{\lambda,\mu}  S_\mu(W)\otimes S_{\tilde{\nu}} (U) & \\
 = \bigoplus_{\lambda \in \cP_{2p}^{\evrow},|\mu|=q,|\nu|=2p+q,\height(\nu)\leq g} c^{\tilde{\nu}}_{\tilde{\lambda},\tilde{\mu}}  S_\mu(W)\otimes S_{\tilde{\nu}} (U) & \; \; \; (\text{by \eqref{eqn:symmetry-littlewood-richardson}}). 
\end{split}
\]
Under the hypothesis that $g \geq 2p+q$, $\height(\nu)\leq g$ holds for all $\nu \in \cP_{2p+q}$. Using \eqref{eqn:littlewoodrichardson} again, the latter is isomorphic to 
\[
\bigoplus_{\lambda \in \cP_{2p}^{\evrow},|\mu|=q }  S_\mu(W)\otimes S_{\tilde{\lambda}} (U) \otimes S_{\tilde{\mu}}(U), 
\]
and by \eqref{eqn:cauchy2}, this agrees with 
\[
\bigoplus_{\lambda \in \cP_{2p}^{\evrow}}  \Lambda^q (W \otimes U)\otimes  S_{\tilde{\lambda}} (U) = \Lambda^q (W \otimes U) \otimes  \bigoplus_{\lambda \in \cP_{2p}^{\evcol}} S_{\lambda} (U) = 
\]
\[
 = \Lambda^q (W \otimes U) \otimes S^p (\Lambda^2 (U)). 
\]

The proof for $H$ is almost identical. The components of $H$ are maps 
\[
S^p (S^2 U) \otimes \Lambda^q (W \otimes U) \to
 (S^p (\Lambda^2 (N)) \otimes \Lambda^q (N \otimes W) \otimes S^{2p+q} (N^\vee \otimes U))^{\GL(N)}. 
\]
We compute, by \eqref{eqn:cauchy4}, \eqref{eqn:cauchy1} and \eqref{eqn:cauchy2} and \eqref{eqn:littlewoodrichardson},
\[
S^p (\Lambda^2 (N)) \otimes \Lambda^q (N \otimes W) \otimes S^{2p+q} (N^\vee \otimes U) = 
\]
\[
=\bigoplus_{\lambda \in \cP_{2p}^{\evcol},|\mu|=q,|\nu|=2p+q} S_\lambda(N) \otimes S_\mu(N) \otimes S_{\tilde{\mu}}(W)\otimes S_\nu (N^\vee) \otimes S_{\nu} (U)=
\]
\[
=\bigoplus_{\lambda \in \cP_{2p}^{\evcol},|\mu|=q,|\nu|=|\kappa|=2p+q} c^\kappa_{\lambda,\mu} S_\kappa(N) \otimes S_{\tilde{\mu}}(W)\otimes S_\nu (N^\vee) \otimes S_{\nu} (U). 
\]
Taking $\GL(N)$-invariants and using \eqref{eqn:orthogonality-schurfunctor} yields 
\[
 (S^p (\Lambda^2 (N)) \otimes \Lambda^q (N \otimes W) \otimes S^{2p+q} (N^\vee \otimes U))^{\GL(N)} = 
\]
\[
=\bigoplus_{\lambda \in \cP_{2p}^{\evcol},|\mu|=q,|\nu|=2p+q,\height(\nu) \leq g} c^\nu_{\lambda,\mu}  S_{\tilde{\mu}}(W)\otimes S_{\nu} (U).\]
If $g \geq 2p+q$, \eqref{eqn:littlewoodrichardson} shows that this is equal to 
\[
\bigoplus_{\lambda \in \cP_{2p}^{\evcol},|\mu|=q}  S_{\tilde{\mu}}(W)\otimes S_{\lambda} (U) \otimes S_{\mu} (U)=
\]
\[
(\bigoplus_{\lambda \in \cP_{2p}^{\evcol}} S_{\lambda} (U)) \otimes (\bigoplus_{|\mu|=q}  S_{\tilde{\mu}}(W) \otimes S_{\mu} (U))\stackrel{\eqref{eqn:cauchy4}, \eqref{eqn:cauchy2}}{=}
\]
\[
S^p (\Lambda^2 (U)) \otimes \Lambda^q (W \otimes U)
\]
as claimed. 
\end{proof}

Finally, we give the short proof of Proposition \ref{prop:invariantsSLZ}. This is an immediate consequence of a more general result.

\begin{lem}\label{lem:boreldensity}
Let $N_\bZ$ be a finitely generated free abelian group, write $N:= N_\bZ \otimes \bQ$ and let $\rho: \SL(N)\to \GL(W)$ be a rational representation. Then 
\[
W^{\SL(N)} =W^{\SL(N_\bZ)} .
\]
\end{lem}

\begin{proof}
Assume $N= \bZ^g$. A special case of Borel's density theorem \cite{BorelDensity} states that $\SL_g (\bZ) \subset \SL_g (\bQ)$ is Zariski dense (a very short and elementary proof for the special linear group has been written down by Putman \cite{Putman}). For $v \in V^{\SL_g (\bZ)}$ and $\ell \in V^*$, the function $\SL_g (\bQ) \to \bQ$, $A \mapsto \ell (v-Av)$, is polynomial and vanishes on $\SL_g (\bZ)$, hence on $\SL_g (\bQ)$, whence $v=Av$ for all $A \in \SL_g (\bQ)$. 
\end{proof}

\section{The cohomology of the block diffeomorphism space}\label{sec:coh:blockdiffspecial}

\subsection{Using invariant theory}

In this section, we finish our partial evaluation of the spectral sequence of the fibration
\begin{equation}\label{eqn:secondkeyfibresequence2334346}
\map_\partial ((U_{g,1}^n)_\bQ;BO_\bQ)^0 \to \hq{\map_\partial ((U_{g,1}^n)_\bQ;BO_\bQ)^0}{\hAut_\partial((U_{g,1}^n)_\bQ)^{\id}} \to B \hAut_\partial((U_{g,1}^n)_\bQ)^{\id}. 
\end{equation}
Before we state the result, let us fix some bounds that the various parameters have to fulfil. 
\begin{assumption}\label{ass:bounds}
\begin{enumerate}
\item We assume throughout that $n  \geq 5$. 
\item We pick $M$ in \eqref{notation.algebraicpicturespectralsequence} so that 
\[
4M \geq 3n-5.
\]
\item We furthermore choose $g$ large enough to satisfy
\[
g >n-3,
\]
which implies also that $g \geq 3$.
\end{enumerate}
\end{assumption}
Using the number $M$, we define as in \eqref{notation.algebraicpicturespectralsequence}
\[
V(n):= \bigoplus_{4m-2n-1>0, m \leq M} \bQ[4m-2n-1], 
\]
\[
U(n) := \bigoplus_{4m-n-1>0, m \leq M} \bQ [4m-n-1]
\]
and 
\[
W(n):= \bigoplus_{4m-n>0, m \leq M} \bQ[4m-n]. 
\]
We let $v_m \in V(n)_{4m-2n-1}$ and $w_m \in W(n)_{4m-n}$ be the obvious generators, and let $u_m := S(w_m) \in U(n)_{4m-n-1}$ be the image under the map $S$ from Proposition \ref{prop:the:E2differential}. 
To formulate the result we are aiming at, some more notation is necessary. 

\begin{defn}\label{defn:Kvecspace}
We let $K(n)$ be the following graded vector space. It has basis elements $k_m \in K(n)_{4m-2n-1}$ for $m \leq M$ and $4m-2n-1>0$, and it has basis elements $k_{m_0,m_1}\in K(n)_{4(m_0+m_1)-2n-1}$ for $m_0 \leq m_1\leq M$, $4m_0 \geq n+1$, $4(m_0+m_1)-2n-1>1$. 
\end{defn}
Note that all the generators of $K(n)$ are in odd degrees. The role that $K(n)$ plays is explained by the next result, which is our goal in this subsection.

\begin{prop}\label{evaluation-spectralsequence-invarianttheory}
The composition of the graded algebra homomorphism
\begin{equation}\label{eqn:defnximap}
\xi: \Lambda^* (K(n)) \to H^* (\hq{\map_\partial ((U_{g,1}^n)_\bQ;BO_\bQ)^0}{\hAut_\partial ((U_{g,1}^n)_\bQ)_\bZ})
\end{equation}
given by 
\[
\xi(k_m):= \kappa_{L_m}, \; \xi(k_{m_0,m_1}):= \kappa_{L_{m_0}L_{m_1}}
\]
with the pullback map 
\[
H^* (\hq{\map_\partial ((U_{g,1}^n)_\bQ;BO_\bQ)^0}{\hAut_\partial ((U_{g,1}^n)_\bQ)_\bZ};\bQ) \to H^* (\hq{\map_\partial ((U_{g,1}^n)_\bQ;BO_\bQ)^0}{\hAut_\partial ((U_{g,1}^n)_\bQ)^\id};\bQ)^{\SL(N(g))_\bZ } \subset H^* (\hq{\map_\partial ((U_{g,1}^n)_\bQ;BO_\bQ)^0}{\hAut_\partial ((U_{g,1}^n)_\bQ)^\id};\bQ)
\]
is an isomorphism
\[
\Lambda^*(K(n)) \to H^* (\hq{\map_\partial ((U_{g,1}^n)_\bQ;BO_\bQ)^0}{\hAut_\partial ((U_{g,1}^n)_\bQ)^\id};\bQ)^{\SL(N(g))_\bZ }
\]
in degrees $* \leq (n-3)$, provided that $M$ and $g$ satisfy the bounds from \eqref{ass:bounds}. 
\end{prop}

As a preparation for the proof of Proposition \ref{evaluation-spectralsequence-invarianttheory}, recall from the discussion before Corollary \ref{cor:actiononcohomologymappingspace} that the spectral sequence $E_*^{*,*}$ of \eqref{eqn:secondkeyfibresequence2334346} is a spectral sequence of $\GL(N(g))\cong \pi_0 (\hAut_\partial((U_{g,1}^n)_\bQ))$-modules. Let us elaborate this a little.

\begin{prop}\label{prop:applicationmargulis}
Let $n \geq 5$ and $g \geq 3$.
\begin{enumerate}
\item For $q \leq (n-3)$, $H^q (\hq{\map_\partial ((U_{g,1}^n)_\bQ;BO_\bQ)^0}{\hAut_\partial ((U_{g,1}^n)_\bQ)^\id};\bQ)$ is a rational representation of $\SL(N(g))$, of load $\leq q$. 
\item Let $E_*^{*,*}$ denote the spectral sequence of \eqref{eqn:secondkeyfibresequence2334346}, and define 
\[
\overline{E}_r^{p,q}:= (E_r^{p,q})^{\SL(N(g)_\bZ)}.
\]
Then $\overline{E}_*^{*,*}$ is a spectral sequence, and it converges to $H^* (\hq{\map_\partial ((U_{g,1}^n)_\bQ;BO_\bQ)^0}{\hAut_\partial ((U_{g,1}^n)_\bQ)^\id};\bQ)^{\SL(N(g))_\bZ }$. 
\end{enumerate}
\end{prop}

\begin{proof}
(1) We have to invoke a deep result by Bass--Milnor--Serre \cite[Corollary 16.6]{BassMilnorSerre}. The quoted result implies that a homomorphism $\SL(N(g)) \to \GL(V)$, where $V$ is a finite-dimensional $\bQ$-vector space and $g \geq 3$, is actually rational (see also the discussion on p. 63 f and p. 134 of loc.cit.). Hence $H^q (\hq{\map_\partial ((U_{g,1}^n)_\bQ;BO_\bQ)^0}{\hAut_\partial ((U_{g,1}^n)_\bQ)^\id};\bQ)$ is a rational $\SL(N(g))$-representation. 

From the description of the $E_2$-term given in Proposition \ref{prop:spectralsequence-isomorphism}, it is apparent that the $\GL(N(g))$-representation $E_2^{p,q}$, for $p \leq (n-3)$, is rational of load $\leq q$. Since subquotients of rational representations are again rational (and the load does not increase), the representations occuring in $E_\infty^{*,*}$ with total degree $q$ are all rational, of load $\leq q$. 

(2): Theorem \ref{thm:representationtheoryGL(N)} and Lemma \ref{lem:boreldensity} together imply that taking $\SL(N(g)_\bZ)$-invariants is an exact functor from rational $\GL(N(g))$-representations to $\bQ$-vector spaces. Hence $\overline{E}_*^{*,*}$ is also a spectral sequence, and the claim about convergence also follows easily.
\end{proof}

Hence in order to prove Proposition \ref{evaluation-spectralsequence-invarianttheory}, we can focus our attention completely on $\overline{E}_*^{*,*}$. 

\begin{defn}
Let $D^{*,*}$ be the bigraded algebra
\[
D^{*,*}:= \Lambda^* (V(n)\oplus W(n) \otimes U(n)) \otimes S^* (\bQ[2,0]\otimes \Lambda^2 (U(n))) 
\]
where the graded vector spaces $W(n),U(n)$ and $V(n)$ are as in \eqref{notation.algebraicpicturespectralsequence}, but sit in bidegrees $(0,*)$. Let $\delta: D^{*,*} \to D^{*,*}$ be the differential of degree $(2,-1)$ which is a derivation, and is given on the generators as follows:
\begin{enumerate}
\item $\delta|_{V(n)} =0$, 
\item $\delta|_{W(n) \otimes U(n)}$ is the map $W(n) \otimes U(n) \stackrel{S\otimes 1}{\to} U(n) \otimes U(n) \to \Lambda^2 (U(n))\cong \bQ[2,0] \otimes \Lambda^2 (U(n))$,
\item $\delta|_{\bQ[2,0]\otimes \Lambda^2 (U(n))}=0$.
\end{enumerate}
\end{defn}
We define a graded algebra homomorphism 
\begin{equation}\label{deqn:etahom}
\eta:D^{*,*} \to \overline{E}^{*,*}_2= (E^{*,*}_2)^{\SL(N_\bZ(g))}
\end{equation}
to the invariant part of the spectral sequence of \eqref{eqn:secondkeyfibresequence2334346}, by sending
\begin{align*}
v_m \mapsto & \lambda_{[U_{g,1}^n],L_m} \in \overline{E}_2^{0,4m-2n-1} \\
w_{m_0} \otimes u_{m_1}\mapsto & \sum_i \lambda_{a_i,L_{m_0}} \lambda_{b_i,L_{m_1}} \in \overline{E}_2^{0,4m_0+4m_1-2n-1} \\
u_{m_0} \wedge u_{m_1} \mapsto & \sum_{ij} x_{ij} \otimes \lambda_{b_i,L_{m_0}} \lambda_{b_j,L_{m_1}} \in \overline{E}_2^{2,4m_0+4m_1-n-2} .\\
\end{align*}
Here $x_{ij} \in E_2^{2,0}$ are the generators in $E_2^{2,0} \cong L^2 (N(g))$.

\begin{prop}\label{invariantofspectralsequeE2}
The map $\eta$ is a map of differential bigraded algebras, and it is an isomorphism in total degrees $\leq (n-3)$, provided that the bounds of \ref{ass:bounds} are satisfied.
\end{prop}

\begin{proof}
That $\eta$ is an isomorphism in the indicated range of degrees follows from Propositions \ref{prop:spectralsequence-isomorphism}, \ref{prop:keylemma-invarianttheory} and \ref{prop:invariantsSLZ}. The first two conditions of \ref{ass:bounds} are needed for \ref{prop:spectralsequence-isomorphism}, and the third for \ref{prop:keylemma-invarianttheory}. 

Proposition \ref{prop:the:E2differential} and the definition of the maps in Proposition \ref{prop:keylemma-invarianttheory} show that $\eta$ is compatible with the differential.
\end{proof}

Let us next compute the cohomology of the differential graded algebras $D^{*,*}$. This is in terms of the following well-known construction. 

\begin{defn}
Let $F:Y \to X$ be a linear map of $\bQ$-vector spaces. The \emph{Koszul complex} of the map $F$ is the graded commutative differential graded algebra
\[
D_F := \Lambda^* (Y) \otimes S^* (X) ,
\]
where $S^1(X)$ has degree $2$, $\Lambda^1(Y)$ has degree $1$, and the differential $d_F: D_F \to D_F$ is the derivation of degree $+1$ given by the condition that $d_F|_{S^1(X)}=0$ and that $d_F: \Lambda^1 (Y) \to S^1 (X)$ is the map $F$. 
\end{defn}

Up to a different grading, $D^{*,*}$ is $D_F$, where $F$ is the map
\begin{equation}\label{eqn:mapforkoszulcomplex}
V(n)\oplus W(n) \otimes U(n) \stackrel{0 \oplus S \otimes 1_{U(n)}}{\to} U(n)^{\otimes 2} \to \Lambda^2(U(n)). 
\end{equation}
In order to compute $H^*(D^{*,*})$ and hence $(E_3^{*,*})^{\SL(N_\bZ(g))}$ in a range of degrees, we compute $H^* (D_F)$ in general. There are obvious linear maps 
\[
\ker (F) \cong \ker \Bigl(d_F: \Lambda^1 (Y) \to S^1 (X)\Bigr) \to H^1 (D_F)
\]
and 
\[
\coker (F) \cong \coker \Bigl(d_F: \Lambda^1 (Y) \to S^1 (X)\Bigr) \to H^2 (D_F). 
\]
They give a map of graded commutative algebras
\[
\eta_F: \Lambda^* (\ker (F)) \otimes S^* (\coker (F)) \to H^* (D_F).
\]
\begin{lem}\label{lem:cohomology-in-model}
The map $\eta_F$ is an isomorphism. 
\end{lem}

\begin{proof}
It follows from the K\"unneth formula that if $\eta_{F_0}$ and $\eta_{F_1}$ are isomorphisms, then so is $\eta_{F_0 \oplus F_1}$. 

By elementary linear algebra, we can write $F$ as a direct sum of an isomorphism and a zero map, and so it suffices to treat these two cases separately. If $F$ is a zero map, the claim is obvious. If $F$ is an isomorphism, we can assume without loss of generality that $F$ is an identity map. In that case, $D_{\id_V}$ is the Koszul complex of the vector space $V$ which is well-known to be acyclic (since we are over a field of characteristic $0$), see e.g. \cite[\S 3.1]{GuilStern}. 
\end{proof}

\begin{proof}[Proof of Proposition \ref{evaluation-spectralsequence-invarianttheory}]
The map in \eqref{eqn:mapforkoszulcomplex} is surjective. Therefore Lemma \ref{lem:cohomology-in-model} and Proposition \ref{invariantofspectralsequeE2} prove that 
\[
\overline{E}_3^{p,q} =0
\]
if $p+q \leq n-3$ and $p \neq 0$. Hence, using Proposition \ref{prop:applicationmargulis}, the edge homomorphism
\[
A:  H^q (\hq{\map_\partial ((U_{g,1}^n)_\bQ;BO_\bQ)^0}{\hAut_\partial ((U_{g,1}^n)_\bQ)^\id};\bQ)^{\SL(N(g))_\bZ } \to \overline{E}^{0,q}_\infty \stackrel{\subset}{\to} \overline{E}_3^{0,q}
\]
is an isomorphism if $q \leq n-3$ (to see that this is also true for $q=n-3$, note that the fibration \eqref{eqn:secondkeyfibresequence2334346} has a section). It is therefore sufficient to show that $A \circ \xi: \Lambda^*(K(n))\to \overline{E}_3^{0,*}$ is an isomorphism in that range of degrees. By the definition of $\xi$ and by Lemma \ref{lem:kappaclassaslambdaclass}, $A \circ \xi$ is given by 
\[
k_m \mapsto \lambda_{[U_{g,1}^n],L_m}
\]
and 
\[
k_{m_0,m_1} \mapsto \lambda_{[U_{g,1}^n],L_{m_0}L_{m_1}}. 
\]
We must therefore check that $\overline{E}_3^{0,*}$ is (in degrees $*\leq n-3$) the free graded-commutative algebra on the listed generators. Lemma \ref{lem:cohomology-in-model} tells us how to do that. 

We now have to distinguish two cases, which our system of notation has hidden so far. If $n \not \equiv 3 \pmod 4$, the map $S : W(n) \to U(n)$ is an isomorphism, while for $n \equiv 3 \pmod 4$, $S$ has a nontrivial kernel. Hence bases for the kernel of \eqref{eqn:mapforkoszulcomplex} will be given by slightly different sets of generators. 
It is not a coincidence that $n \equiv 3 \pmod 4$ is the case where Proposition \ref{prop:vanishing-degree1} provides an additional relation for the tautological classes, see the remark after Theorem \ref{mainthm:main}

If $n \not \equiv 3 \pmod 4$, the map $S : W(n) \to U(n)$ is an isomorphism. We deduce that the following set is a basis for the kernel of \eqref{eqn:mapforkoszulcomplex}:
\begin{equation}\label{eqn:basisforkernelkoszul}
\{ v_m, ( w_{m_0} \otimes u_{m_1} + w_{m_1} \otimes u_{m_0})  \vert \; 4m-2n-1>0, m \leq M, m_0 \leq m_1 \leq M, 4m_0-n-1>0 \}. 
\end{equation}
Under the map $\eta$ from \eqref{deqn:etahom}, the element $v_m$ is mapped to $\lambda_{[U_{g,1}^n],L_m}$. The element $ w_{m_0} \otimes u_{m_1} + w_{m_1} \otimes u_{m_0}$ is mapped to 
\[
\begin{split}
 \sum_i \lambda_{a_i,L_{m_0}} \lambda_{b_i,L_{m_1}} + \lambda_{a_i,L_{m_1}} \lambda_{b_i,L_{m_0}}   & \\
 =  \sum_i \lambda_{a_i,L_{m_0}} \lambda_{b_i,L_{m_1}} +  \lambda_{b_i,L_{m_0}} \lambda_{a_i,L_{m_1}}  & (\text{for degree reasons})\\
 =  \lambda_{[U_{g,1}^n],L_{m_0}L_{m_1}}  & (\text{by \eqref{eqn:lambdaproductug1n}}).
\end{split}
\]
This completes the proof if $n \not \equiv 3 \pmod 4$. 

If $n \equiv 3 \pmod 4$, the kernel of \eqref{eqn:mapforkoszulcomplex} is larger: one obtains a basis by adding the elements
\[
w_{\frac{n+1}{4}} \otimes u_m, \; 4m-n-1>0, m \leq M
\]
to the list of elements in \eqref{eqn:basisforkernelkoszul}. Using \eqref{deqn:etahom} and \eqref{eqn:lambdaproductug1n} again, these elements go to $\lambda_{[U_{g,1}^n],L_{\frac{n+1}{4}}L_{m}}$. So the proof is complete.
\end{proof}

\subsection{Using Borel's vanishing theorem} 

We now look at the spectral sequence of the fibration 
\begin{equation}\label{eqn:firstkeyfibresequence2132344567576879}
\hq{\map_\partial ((U_{g,1}^n)_\bQ;BO_\bQ)^0}{\hAut_\partial((U_{g,1}^n)_\bQ)^{\id}} \to \hq{\map_\partial ((U_{g,1}^n)_\bQ;BO_\bQ)^0}{\hAut_\partial((U_{g,1}^n)_\bQ)_\bZ} \to B \GL(N(g)_\bZ). 
\end{equation}
Recall the graded vector space $K(n)$ from Definition \ref{defn:Kvecspace} and from \eqref{eqn:defnximap} the map 
\[
\xi: \Lambda^* (K(n)) \to H^* (\hq{\map_\partial ((U_{g,1}^n)_\bQ;BO_\bQ)^0}{\hAut_\partial((U_{g,1}^n)_\bQ)_\bZ}; \bQ).
\]
Let moreover $B$ be the graded vector space $B:=\bigoplus_{k \geq 1} \bQ[4k+1]$. Mapping the generators to the Borel classes gives furthermore a map 
\[
\beta: \Lambda^* (B) \to H^* (\hq{\map_\partial ((U_{g,1}^n)_\bQ;BO_\bQ)^0}{\hAut_\partial((U_{g,1}^n)_\bQ)_\bZ};\bQ).
\]

\begin{prop}\label{prop:cohomology-tangentialhomotopyautomorphisms}
Assume the bounds from \eqref{ass:bounds} (which in particular means $n\geq 5$), but strengthened by
\[
g  \geq 2n-4. 
\]
Then the map 
\[
\xi\otimes \beta: \Lambda^* (K(n)  \oplus B) \to H^* (\hq{\map_\partial ((U_{g,1}^n)_\bQ;BO_\bQ)^0}{\hAut_\partial((U_{g,1}^n)_\bQ)_\bZ};\bQ)
\]
is an isomorphism in degrees $* \leq (n-3)$. 
\end{prop}

When combined with Theorem \ref{thm:blockdiffsversusmappingspace}, we obtain the following result as a corollary. This establishes Theorem \ref{thm:blockdiffs} from the introduction. 

\begin{cor}\label{cor:cohomologyblockdiffs}
Assume the bounds of Proposition \ref{prop:cohomology-tangentialhomotopyautomorphisms}. Then in degrees $\leq (n-4)$, $H^* (B \blockdiff_\partial (U_{g,1}^n);\bQ)$ is the exterior algebra on the Borel classes and on the tautological classes $\kappa_{L_{m_0} L_{m_1}}$ with $m_0 \leq m_1$, $4m_0-n>0$ and $4(m_0+m_1)-2n-1>0$. \qed
\end{cor}

The proof of Proposition \ref{prop:cohomology-tangentialhomotopyautomorphisms} relies on Borel's vanishing theorem \cite{Borel2} that we shall state first. 

\begin{thm}[Borel]\label{thm:borelvanishing}
Let $V$ be a rational representation of $\SL_g (\bQ)$ of load at most $n$. Then the map 
\[
H^p (\SL_g (\bZ);\bQ) \otimes V^{\SL_g(\bZ)} = H^p (\SL_g (\bZ);\bQ) \otimes V^{\SL_g(\bQ)} \to H^p (\SL_g (\bZ);V^{\SL_g (\bQ)}) \to H^p (\SL_g (\bZ);V)
\]
is an isomorphism, provided that $ 2p+2 \leq g-n $ (the first equation holds by Lemma \ref{lem:boreldensity}).
\end{thm}

\begin{proof}[References]
This is essentially due to Borel \cite[Theorem 4.4]{Borel2}, but the ranges are not made explicit enough for our needs in Borel's work, so some more words need to be said here. By complete reducibility, we can assume that $V \subset T^{k,l}(\bQ^g)$ with $k+l\leq n$, and finally suppose that $V= T^{k,l}(\bQ^g)$. Moreover, it is enough to prove that statement for $\bQ$ replaced by $\bR$ and $\bQ^g$ replaced by $\bR^g$. 

Borel proved in loc.cit. that
\begin{equation}\label{eqn:borelmap1}
H^p (\SL_g (\bZ);\bR) \otimes T^{k,l}(\bR^g)^{\SL_g(\bR)}\to H^p (\SL_g (\bZ);T^{k,l}(\bR^g))
\end{equation}
is an isomorphism provided that $p \leq \min (M(\SL_g(\bR),(k,l)),C(\SL_g(\bR),(k,l)))$, where $M(\SL_g(\bR),(k,l))$ and $C(\SL_g(\bR),(k,l))$ are constants which can be read off from the root system of $\mathfrak{sl}_g$ and the weights of $T^{k,l}(\bQ^g)$. In loc.cit., Borel showed that $M(\SL_g(\bR),(k,l)) \geq g-2$, but left $C(\SL_g(\bR),(k,l))$ implicit. 
A relatively naive counting argument given in the proof of \cite[Theorem 7.3]{KrannRW} shows that $C(\SL_g(\bR),(k,l)) \geq \frac{1}{8} g^2-\max(k,l)-1$, which implies that \eqref{eqn:borelmap1} is an isomorphism when $p \leq f_{k,l}(g)$, and $f_{k,l}$ is a function with $\lim_{g \to \infty} f_{k,l}(g)=\infty$. 

To get at the range claimed by us, we use Van der Kallen's work on homological stability, more precisely \cite[Theorem 5.6]{VanderKallen}. The latter result implies that for $2q+2\leq g-(k+l)$ and all $h \geq g$, the map $H_q (\SL_g (\bZ); T^{k,l}(\bR^g)) \to H_q (\SL_h (\bZ); T^{k,l}(\bR^h))$ is an isomorphism. This implies a cohomological statement by an instance of the universal coefficient theorem \cite[Lemma 3.5]{ORWOutFn}, and we can pick $h$ large enough so that $f(h) \geq q$. 
\end{proof}

\begin{rem}
The referee pointed us to the paper \cite{LiSun} which gives a range for Theorem \ref{thm:borelvanishing} that only depends on $g$, not on $V$. Using this, we would achieve some minor simplification in our proofs. Tshishiku \cite{Tshishiku} showed by a careful analysis of root systems that the corresponding result is true for $\SO_{g,g}(\bZ)$ and $\Sp_{2g}(\bZ)$; we have been informed by him that the analogous procedure for $\SL_g (\bZ)$ does \emph{not} lead to a range independent of $V$. 
\end{rem}

\begin{proof}[Proof of Proposition \ref{prop:cohomology-tangentialhomotopyautomorphisms}]
In order to accommodate to the fact that Theorem \ref{thm:borelvanishing} is about the special linear group rather than the general linear group, we let $\mathrm{S}\hAut_\partial((U_{g,1}^n)_\bQ)_\bZ \subset \hAut_\partial((U_{g,1}^n)_\bQ)_\bZ$ consists of those homotopy automorphisms whose action on $H_n (U_{g,1}^n)$ is by maps of determinant $1$. The natural map 
\[
\hq{\map_\partial ((U_{g,1}^n)_\bQ;BO_\bQ)^0}{\mathrm{S}\hAut_\partial((U_{g,1}^n)_\bQ)_\bZ} \to \hq{\map_\partial ((U_{g,1}^n)_\bQ;BO_\bQ)^0}{\hAut_\partial((U_{g,1}^n)_\bQ)_\bZ}
\]
is, up to homotopy, a $2$-fold covering; hence it induces an injective map in rational cohomology by a general argument \cite[Proposition 3G.1]{Hatcher}. Therefore, it is enough to give the proof with $\hq{\map_\partial ((U_{g,1}^n)_\bQ;BO_\bQ)^0}{ \mathrm{S}\hAut_\partial((U_{g,1}^n)_\bQ)_\bZ}$ in place of $\hq{\map_\partial ((U_{g,1}^n)_\bQ;BO_\bQ)^0}{\hAut_\partial((U_{g,1}^n)_\bQ)_\bZ}$. 

Let $E_*^{*,*}$ be the spectral sequence for the fibre sequence
\begin{equation}\label{eqn:firstkeyfibresequenceSL}
\hq{\map_\partial ((U_{g,1}^n)_\bQ;BO_\bQ)^0}{\hAut_\partial((U_{g,1}^n)_\bQ)^{\id}} \to \hq{\map_\partial ((U_{g,1}^n)_\bQ;BO_\bQ)^0}{\mathrm{S}\hAut_\partial((U_{g,1}^n)_\bQ)_\bZ} \to B \SL(N(g)_\bZ);
\end{equation}
it has $E_2$-term
\[
E_2^{p,q} = H^p (\SL(N(g)_\bZ);H^q (\hq{\map_\partial ((U_{g,1}^n)_\bQ;BO_\bQ)^0}{\hAut_\partial((U_{g,1}^n)_\bQ)^{\id}};\bQ)). 
\]
By Proposition \ref{prop:applicationmargulis}, the coefficient module is a rational representation of load $\leq q$. Hence if $2p+q+2\leq g$, we may invoke Theorem \ref{thm:borelvanishing} and see that 
\begin{equation}\label{eqn:spectralseqzuenceproduct}
E_2^{p,q} = H^p (\SL(N(g)_\bZ);\bQ) \otimes H^q (\hq{\map_\partial ((U_{g,1}^n)_\bQ;BO_\bQ)^0}{\hAut_\partial((U_{g,1}^n)_\bQ)^{\id}};\bQ))^{\SL(N(g)_\bZ)}. 
\end{equation}
Under our bound on $g$, this holds for all $p,q$ with $p+q \leq n-3$.

It follows from Proposition \ref{evaluation-spectralsequence-invarianttheory} and \eqref{eqn:defnximap} that the map 
\[
H^*( \hq{\map_\partial ((U_{g,1}^n)_\bQ;BO_\bQ)^0}{\mathrm{S}\hAut_\partial((U_{g,1}^n)_\bQ)_\bZ};\bQ) \to
H^* (\hq{\map_\partial ((U_{g,1}^n)_\bQ;BO_\bQ)^0}{\hAut_\partial((U_{g,1}^n)_\bQ)^{\id}};\bQ)^{\SL(N(g)_\bZ)}
\]
is surjective if $* \leq n-3$, but that can be identified with the edge homomorphism of the spectral sequence, so that all differentials starting in the zeroeth column vanish (in degrees $* \leq n-3$). The isomorphism \eqref{eqn:spectralseqzuenceproduct} implies that the map $E_2^{*,0} \otimes E_2^{0,*} \to E_2^{*,*}$ of bigraded vector spaces is an isomorphism in total degrees $\leq n-3$. As the spectral sequence is a spectral sequence of algebras, we conclude that all differentials starting in a term $E_r^{p,q}$ with $p+q \leq n-3$ are trivial. 

An application of Theorem \ref{thm:Borel} and the Leray--Hirsch theorem shows that the map $\xi \otimes \beta$ is an isomorphism in degrees $\leq n-3$. 
\end{proof}

\section{The endgame: from block diffeomorphisms to actual diffeomorphisms}\label{sec:endgame}

In this rather short section, we compare the rational cohomologies of $B \Diff_\partial (U_{g,1}^n)$ and $B \blockdiff_\partial (U_{g,1}^n)$ in a range of degrees and thereby complete the proof of Theorem \ref{mainthm:main}. 
We phrase the argument in a way that is largely independent of the calculations for $H^* (B \blockdiff_\partial (U_{g,1}^n);\bQ)$ which we carried out in the previous sections. To formulate the result, let us introduce the following notations: 
\[
\cD_g :=  \Diff_\partial (U_{g,1}^n); \; \tilde{\cD}_g := \blockdiff_\partial (U_{g,1}^n). 
\]
There are stabilization maps $B\cD_g \to B\cD_{g+1}$ and $B\tilde{\cD}_g \to B\tilde{\cD}_{g+1}$ defined by gluing in $S^n \times S^{n+1}$ with two discs removed and extending block diffeomorphisms by the identity, and we let $B \cD_\infty:= \hocolim_{g \to \infty} B \cD_g$ and define $B \tilde{\cD}_\infty$ analogously. 
The aim of this section is the following result, which together with Theorem \ref{thm:blockdiffs} and \cite[Corollary 1.3.2]{Perlmutter} immediately implies Theorem \ref{mainthm:main} for $n \geq 5$ (the case $n=4$ is easier).

\begin{thm}\label{thm:comparisondiffblockdiff}\mbox{}
\begin{enumerate}
\item Let $n \geq 3$. Then $H^* (B \tilde{\cD}_\infty;\bQ)$ is a free-graded commutative algebra which is degreewise finitely generated.
\item Let $n \geq 4$. Then $H^* (B \cD_\infty;\bQ)$ is a free-graded commutative algebra which is degreewise finitely generated.
\item Let $n \geq 3$. Then the map 
\[
B\rho_n: B \tilde{\cD}_\infty \to B \GL_\infty (\bZ)
\]
given by the action on $H_n(\_;\bZ)$ induces an injective map in rational cohomology, and $H^*(B \tilde{\cD}_\infty)$ is isomorphic to the tensor product of $H^*(B \GL_\infty (\bZ);\bQ)$ and a free graded-commutative algebra. 
\item Let $n \geq 5$. Then the comparison map 
\[
BI_\infty: B \cD_\infty \to B \tilde{\cD}_\infty
\]
induces a surjection in rational cohomology in degrees $* \leq 2n-5$, and in this range of degrees, the kernel of $(BI_\infty)^*$ is the ideal generated by the image of $(B \rho_n)^*: H^{*>0}(B \tilde{\cD}_\infty;\bQ) \to H^{*>0}(B\cD_\infty;\bQ)$. 
\end{enumerate}
\end{thm}

\subsection{Using Morlet's Lemma and Farrell-Hsiangs theorem}

The first step of the proof of Theorem \ref{thm:comparisondiffblockdiff} is to get a hold of the rational homotopy groups of the fibre of $BI_\infty$ in a range of degrees. The stabilization maps fit into a commutative diagram 
\begin{equation}\label{morletdiagram}
\xymatrix{
B \cD_g\ar[r] \ar[d]^{BI_g} & B\cD_{g+1} \ar[d]^{BI_{g+1}}\\
B\tilde{\cD}_g \ar[r] & B\tilde{\cD}_{g+1}
}
\end{equation}
and by passing to the colimits, we obtain
\begin{equation}\label{morletdiagraminfinite}
\xymatrix{
B \cD_0\ar[r] \ar[d]^{BI_0} & B\cD_{\infty} \ar[d]^{BI_\infty}\\
B\tilde{\cD}_0 \ar[r] & B\tilde{\cD}_{\infty}.
}
\end{equation}

\begin{lem}\label{morlet-lemma-of-disjunction}
For $n \geq 2$, the diagrams \eqref{morletdiagram} and \eqref{morletdiagraminfinite} are $(2n-4)$-cartesian; i.e. the induced map on vertical (or equivalently horizontal) homotopy fibres is $(2n-4)$-connected. 
\end{lem}

\begin{proof}
The so-called \emph{Morlet's lemma of disjunction}, more precisely \cite[Theorem 3.1 on page 28]{BLR}, immediately says that \eqref{morletdiagram} is $(2n-4)$-cartesian for each $g$, and the claim for \eqref{morletdiagraminfinite} follows by passage to the colimit. 
\end{proof}

\begin{lem}\label{lem:cerf}
The map 
\[
BI_g: B \cD_g \to B \tilde{\cD}_g
\]
induces an isomorphism on fundamental groups, if $n \geq 2$. 
\end{lem}

\begin{proof}
Surjectivity of $\pi_1 (BI_g)$ follows from the very definition of the block diffeomorphism group, and injectivity follows from Cerf's theorem \cite[Th\'eor\`eme 0]{CerfIsotopy}. 
\end{proof}

\begin{lem}\label{lem:farrellhsiang}
The homotopy fibre $\tilde{\cD}_0/\cD_0$ of $BI_0$ is connected, and its rational homotopy groups in low degrees are given by 
\[
\pi_k (\tilde{\cD}_0/\cD_0)\otimes \bQ =
\begin{cases}
\bQ & 1 \leq k \leq 2n-3, \; k \equiv 0 \pmod 4\\
0 & 1 \leq k \leq 2n-3, \; k \not\equiv 0 \pmod 4\\
\end{cases}
\]
when $n \geq 5$; when $n=4$, the above isomorphisms hold for $1 \leq k \leq 4$.
\end{lem}

\begin{proof}
Lemma \ref{lem:cerf} shows that $\tilde{\cD}_0/\cD_0$ is connected and has abelian fundamental group. Note that $U_{0,1}^n=D^{2n+1}$. The homotopy group $\pi_k (B \blockdiff_\partial (D^{2n+1}))$ (for $k \geq 1$) can be identified with the group $\Theta_{2n+1+k}$ of homotopy spheres, which is of course finite by \cite{KervaireMilnor}. It follows that $\tilde{\cD}_0/\cD_0 \to B \cD_0$ induces an isomorphism on rational homotopy groups. 

The rational homotopy groups of $B \Diff_\partial (D^{2n+1})$ are famously related to algebraic $K$-theory, originally by \cite{FarrellHsiang}; the range stated was given by Krannich and Randal--Williams in \cite[Theorem A]{KrannRW}.
\end{proof}

The results surveyed above suffice to deal with the case $n=4$ of our main result.

\begin{proof}[Proof of Theorem \ref{mainthm:main} for $n=4$]
It follows immediately from Theorem \ref{thm:blockdiffs}, Lemma \ref{morlet-lemma-of-disjunction}, Lemma \ref{lem:cerf} and Lemma \ref{lem:farrellhsiang} that 
\[
H^1 (B \Diff_\partial (U_{g,1}^4);\bQ)=0
\]
for all $g$, which shows Theorem \ref{mainthm:main} for $n=4$.
\end{proof}

As we shall see below, the rational cohomology of $\tilde{\cD}_0/\cD_0$ is an easy consequence of Lemma \ref{lem:farrellhsiang}. In order to make use of this for our homological calculation, we need to overcome the problem that the fundamental group of $B \cD_\infty$ is highly nontrivial. The relevant argument is developed in the next subsection. 

\subsection{Getting the mapping class group under control}

The spaces 
\[
B\cD := \coprod_{g \geq 0} B \cD_g \; \text{and} \; B\tilde{\cD} := \coprod_{g \geq 0} B \tilde{\cD}_g
\]
carry $E_{2n+1}$-structures, analogous to the $E_2$-structure described in \cite[\S 3]{GalKupRW} for diffeomorphisms of surfaces; we refrain from giving any more details here. As a consequence of May's recognition principle \cite{MayLoop}, the group completions $\Omega B B \cD$ and $\Omega B B \tilde{\cD}$ have the structures of $(2n+1)$-fold loop spaces. 
The comparison map $B\cD \to B\tilde{\cD}$ is a map of $E_{2n+1}$-spaces, and hence the induced map on group completions is a map of $(2n+1)$-fold loop spaces. 

\begin{lem}\label{lem:plusconstruction}
For $n \geq 1$, the natural maps 
\begin{equation}\label{eqn:groupcompletionmap1}
\bZ \times B \cD_\infty \to \Omega B B \cD
\end{equation}
and 
\begin{equation}\label{eqn:groupcompletionmap2}
\bZ \times B \tilde{\cD}_\infty \to \Omega B B \tilde{\cD}
\end{equation}
are acyclic (i.e. their homotopy fibres have the integral homology of a point). The commutator subgroups of $\pi_1 (B \cD_\infty)$ and $\pi_1 (B \tilde{\cD}_\infty)$ are perfect. The maps \eqref{eqn:groupcompletionmap1} and \eqref{eqn:groupcompletionmap2} identify their targets with the Quillen plus construction on their source (here the Quillen plus construction is performed one component at a time, on the maximal perfect normal subgroup of the fundamental group).
\end{lem}

\begin{proof}
A straightforward application of the group completion theorem \cite{McDuffSegal} shows that the two maps are integral homology equivalences. An improved version of the group completion theorem, namely \cite[Theorem 1.1]{ORWgcthm}, proves the (stronger) claim of acyclicity. Perfectness of the commutator subgroups follows from \cite[Proposition 3.1]{ORWgcthm}, and the statement about the Quillen plus construction is a consequence: by \cite[Proposition 3.1]{HausHus}, acyclic maps out of a given space are classified up to homotopy equivalence by the kernels of their induced maps on fundamental groups.
\end{proof}

Our goal in this subsection is the following statement.

\begin{prop}\label{prop:comparisondiagramcartesian}
For $n \geq 3$, the diagram 
\[
\xymatrix{
B \cD_0 \ar[d] \ar[r] & B \cD_\infty^+ \ar[d]\\
B \tilde{\cD}_0 \ar[r] & B \tilde{\cD}_\infty^+
}
\]
is $(2n-5)$-cartesian.
\end{prop}

The proof of Proposition \ref{prop:comparisondiagramcartesian} is a sequence of lemmas. We let
\[
B \cD_g \to \cE_g \to B\tilde{\cD}_g
\]
be the $(2n-4)$th stage of the Moore--Postnikov tower of the natural map $B \cD_g \to B \tilde{\cD}_g$ 
(we use the same indexing convention for Moore--Postnikov towers as \cite[p.414]{Hatcher}, so the first map is $(2n-4)$-connected and the second one $(2n-4)$-coconnected). Using naturality of the Moore--Postnikov tower, we obtain the $(2n-4)$-th Moore--Postnikov stage
\[
B \cD_\infty \to \cE_\infty \to B \tilde{\cD}_\infty
\]
in the colimit. 
\begin{lem}\label{lem:moorepostnikov}
The diagrams
\[
\xymatrix{
B \cD_g \ar[r] \ar[d] & \cE_g \ar[d]\\
B \tilde{\cD}_g \ar@{=}[r] & B \tilde{\cD}_g
}
\]
are $(2n-4)$-cartesian for all $0 \leq g\leq \infty$. The diagrams 
\[
\xymatrix{
\cE_g \ar[r] \ar[d] & \cE_{g+1} \ar[d]\\
B\tilde{\cD}_g \ar[r] & B\tilde{\cD}_{g+1}
}
\]
are homotopy cartesian, for each $0 \leq g< \infty$.
\end{lem}

\begin{proof}
If $X \stackrel{f}{\to} Y \stackrel{g}{\to} Z$ are two maps and $f$ is $k$-connected, then 
\[
\xymatrix{
X \ar[d]^{g \circ f}\ar[r]^{f} & Y \ar[d]^{g}\\
Z \ar@{=}[r] & Z
}
\]
is $k$-cartesian. Apply this observation to the definition of $\cE_g$ to get the first claim. The second claim follows from the $(2n-4)$-cartesianness of \eqref{morletdiagram}, together with the following general fact: if 
\[
\xymatrix{
X_0 \ar[d]^{f_0} \ar[r] & X_1 \ar[d]^{f_1} \\
Y_0 \ar[d]^{g_0} \ar[r] & Y_1 \ar[d]^{g_1} \\
Z_0 \ar[r] & Z_1
}
\]
is a commutative diagram of spaces, the maps $f_i$ are $k$-connected, and the maps $g_i$ are $k$-coconnected, and the large rectangle is $k$-cartesian, then the lower square is homotopy cartesian. 
\end{proof}

\begin{lem}\label{morlet-lemma-abeliannesstrick}
For each $1 \leq g \leq \infty$, the fibration $\cE_g \to  \tilde{\cD}_g$ is ``very simple'' in the following sense: if $\cF_g$ denotes its homotopy fibre, then the monodromy action $\mu(\gamma): \cF_g \to \cF_g$, for each $\gamma \in \pi_1 (B \tilde{\cD}_g)$, is homotopic to the identity.
\end{lem}

\begin{proof}
We follow the outline of a very similar argument contained in \cite[\S 5.3]{Krannich}. It is enough to deal with the case of finite $g$, the case $g=\infty$ follows by passage to the colimit. We present the core argument first. Note that the homotopy fibre of $B \tilde{\cD}_g \to B \cD_g$ is the homogeneous space $\tilde{\cD}_g /\cD_g$. Let 
\[
\sigma: \tilde{\cD}_0/\cD_0 \to \tilde{\cD}_g/\cD_g
\]
be the stabilization map and let $\gamma \in \tilde{\cD}_g$. Note that $\sigma$ is given by gluing in (block) diffeomorphisms in a fixed disc, and note that $\gamma$ can be isotoped to a diffeomorphism that fixes this disc. It follows that 
\begin{equation}\label{eqn:monodromyargument}
L_\gamma \circ \sigma \sim \sigma: \tilde{\cD}_0/\cD_0 \to \tilde{\cD}_g/\cD_g,
\end{equation}
where $L_\gamma$ denotes the left translation by $\gamma$ on $\tilde{\cD}_g/\cD_g$. Note that the monodromy action of $\gamma$ on $\tilde{\cD}_g/\cD_g$ is exactly $L_\gamma$.

By definition, $\pi_*(\cF_g)=0$ if $* \geq 2n-4$, and there is a $(2n-4)$-connected map $q: \tilde{\cD}_g/\cD_g\to \cF_g$, equivariant with respect to the two monodromy actions. We may also assume that $\cF_g$ is a CW complex.

Since $\pi_k(\cF_g)=0$ if $k \geq 2n-4$, two maps $f_0,f_1: K \to \cF_g$ from a CW complex are homotopic if their restrictions $f_j|_{K^{(2n-4)}}$ are homotopic. We apply this to $f_0=\id_{\cF_g}$ and $f_1 = \mu(\gamma)$. The inclusion of the $(2n-4)$-skeleton $\cF_g^{(2n-4)} \to \cF_g$ can be factored through maps 
\[
\cF_g^{(2n-4)} \stackrel{h}{\to} \tilde{\cD}_0/\cD_0 \stackrel{\sigma}{\to} \tilde{\cD}_g/\cD_g\stackrel{q}{\to} \cF_g. 
\]
For $\gamma \in \tilde{\cD}_g$, we therefore have
\[
\mu(\gamma)|_{\cF_g^{(2n-4)}} \sim \mu(\gamma) \circ q \circ \sigma\circ h \sim  q \circ L_\gamma \circ \sigma\circ h\stackrel{\eqref{eqn:monodromyargument}}{\sim}  q \circ \sigma\circ h = \id|_{\cF_g^{(2n-4)}},
\]
so that $\mu(\gamma) \sim \id$, as desired.
\end{proof}

\begin{lem}\label{lem:result-of-obstructionargument}
The commutator subgroup of $\pi_1 (\cE_\infty)$ is perfect, and the diagram 
\[
\xymatrix{
\cE_\infty \ar[d] \ar[r] & \cE_\infty^+\ar[d]\\
 B \tilde{\cD}_\infty \ar[r] & B \tilde{\cD}_\infty^+
}
\]
is homotopy cartesian. 
\end{lem}

\begin{proof}
First note that the composition
\[
B \cD_\infty \to \cE_\infty \to B \tilde{\cD}_\infty
\]
induces an isomorphism on fundamental groups by Lemma \ref{lem:cerf}, and the first map is $(2n-4)$-connected by definition, so that both maps induce isomorphisms on fundamental groups. By Lemma \ref{lem:plusconstruction}, the commutator subgroup of $\cE_\infty$ is hence perfect. 

For the proof that the square is homotopy cartesian, we use a theorem by Berrick \cite[Theorem 1.1]{Berrick}.

The fact that $\pi_1 (\cE_\infty) \to \pi_1 (B \tilde{\cD}_\infty)$ is an isomorphism has three consequences: firstly, $\cF_\infty$ is connected; secondly $\pi_1 (\cF_\infty)$ is abelian (because it is a quotient of $\pi_2 (B \tilde{\cD}_\infty)$); and thirdly $\cF_\infty$ is nilpotent ($\cF_\infty$ is also the homotopy fibre of the map $\widetilde{\cE_\infty} \to  \widetilde{B \tilde{\cD}_\infty}$ induced on universal coverings, so it is the homotopy fibre of a map of $1$-connected spaces; it is a general fact that such homotopy fibres are nilpotent, see e.g. \cite[Proposition 4.4.1]{MayPonto}). Hence $\cF_\infty^+ = \cF_\infty$ is nilpotent, which is part of hypothesis (b) of \cite[Theorem 1.1]{Berrick}. The rest of the hypothesis is that the commutator subgroup of $\pi_1 (B \tilde{\cD}_\infty)$ acts trivially on the homology of $\cF_\infty$, and this follows from Lemma \ref{morlet-lemma-abeliannesstrick}. It now follows from \cite[Theorem 1.1]{Berrick} that 
\[
\cF_\infty\to \cE_\infty^+ \to B \tilde{\cD}_\infty^+
\]
is a fibre sequence, which is exactly the statement that the square is homotopy cartesian. 
\end{proof}

To finish the proof of Proposition \ref{prop:comparisondiagramcartesian}, we need a general property of the Quillen plus construction. 

\begin{lem}\label{lem:quillenplus}
Let $f: X \to Y$ be a $n$-connected map of connected spaces, $n \geq 2$, let $P \subset \pi_1 (X)=\pi_1 (Y)$ be a perfect normal subgroup of the common fundamental group, and let $X \to X^+$, $Y \to Y^+$ be the Quillen plus constructions on $P$. Then $f^+: X^+ \to Y^+$ is $n$-connected.
\end{lem}

\begin{proof}
If $P=\pi_1 (X)$, there is not much say, besides quoting Hurewicz' theorem. In the general case, let $\tilde{X} \to X$, $\tilde{Y} \to Y$ be the coverings with fundamental group $P$. Now $X^+$ can be realized as the homotopy pushout 
\[
\xymatrix{
\tilde{X} \ar[r] \ar[d] & \tilde{X}^+\ar[d]\\
X \ar[r] & X^+,
}
\]
see e.g. \cite[p.374]{Hatcher}, and the claim follows.
\end{proof}

\begin{proof}[Proof of Proposition \ref{prop:comparisondiagramcartesian}]
Lemma \ref{morlet-lemma-of-disjunction}, Lemma \ref{lem:moorepostnikov} and Lemma \ref{lem:result-of-obstructionargument} imply that the large rectangle in
\[
\xymatrix{
B \cD_0 \ar[d] \ar[r] & B \cD_\infty \ar[d] \ar[r] & \cE_\infty  \ar[d] \ar[r] & \cE_\infty^+ \ar[d] \\
B \tilde{\cD}_0 \ar[r] & B \tilde{\cD}_\infty \ar@{=}[r] & B \tilde{\cD}_\infty \ar[r] & B \tilde{\cD}_\infty^+
}
\]
is $(2n-4)$-cartesian. This is the same as the rectangle in 
\[
\xymatrix{
B \cD_0 \ar[d] \ar[r] & B \cD_\infty^+ \ar[d] \ar[r] & \cE_\infty^+ \ar[d] \\
B \tilde{\cD}_0 \ar[r] & B \tilde{\cD}_\infty^+ \ar@{=}[r] & B \tilde{\cD}_\infty^+  . 
}
\]
The map $B \cD_\infty^+ \to \cE_\infty^+$ is $(2n-4)$-connected by Lemma \ref{lem:quillenplus} which applies as $n \geq 3$, and so the right square is $(2n-4)$-cartesian. Hence the left square is $(2n-5)$-cartesian, as required. 
\end{proof}

\subsection{Computation of the cohomology}

We now let $\cG$ be the homotopy fibre of the map $\Omega BB\cD \to \Omega BB\tilde{\cD}$ at the basepoint. By Lemma \ref{lem:plusconstruction}, this is also the homotopy fibre of $B \cD_\infty^+ \to B \tilde{\cD}_\infty^+$, so that there is a fibre sequence 
\[
\cG \stackrel{j}{\to} B \cD_\infty^+ \to B \tilde{\cD}_\infty^+. 
\]

\begin{lem}\label{lem:useofpastingtheorem}\mbox{}
\begin{enumerate}
\item $\cG$ has the homotopy type of a connected $(2n+1)$-fold loop space, and its rational homotopy groups in degrees $k \leq 2n-5$ are, when $n \geq 4$, given by 
\[
\pi_k (\cG)\otimes \bQ =
\begin{cases}
\bQ & 1 \leq k \leq 2n-5, \; k \equiv 0 \pmod 4\\
0 & 1 \leq k \leq 2n-5, \; k \not\equiv 0 \pmod 4.\\
\end{cases}
\]
\item The map $j$ induces the zero map on (reduced) rational homology up to degree $2n-5$, if $n \geq 5$.
\end{enumerate}
\end{lem}

\begin{proof}
Being the homotopy fibre of a map of $(2n+1)$-fold loop spaces at the basepoint, it is clear that $\cG$ has the homotopy type of a $(2n+1)$-fold loop space. By Proposition \ref{prop:comparisondiagramcartesian}, there is a $(2n-5)$-connected map 
\[
\hq{\blockdiff_\partial (D^{2n+1})}{\Diff_\partial (D^{2n+1})} \to \cG,
\]
and so (1) follows from Lemma \ref{lem:farrellhsiang}. 

For (2), consider the diagram
\[
\xymatrix{
\tilde{\cD}_0/\cD_0 \ar[d] \ar[r] & \cG \ar[d]^{j}\\
B \cD_0 \ar[r] \ar[d]  & B \cD_\infty^+ \ar[d] \\
B \tilde{\cD}_0 \ar[r]  & B \tilde{\cD}_\infty^+. 
}
\]
The middle horizontal map in the diagram induces the zero map on (reduced) rational homology up to degree $2n-3$, because the same is true for the map $B \cD_0 \to B \cD_\infty$ before taking the plus construction. This was proven by the first named author in \cite[Theorem 1.7]{EbertDisc}, based on \cite{BotPerl} and \cite{PerlStab}.

The top map in the diagram is $(2n-5)$-connected by Proposition \ref{prop:comparisondiagramcartesian}, and as $B \tilde{\cD}_0$ has finite homotopy groups, $\tilde{\cD}_0/\cD_0 \to B \cD_0$ is a rational homotopy equivalence. Putting these facts together, triviality of $H_* (j;\bQ)$ follows.
\end{proof}

\begin{proof}[Proof of Theorem \ref{thm:comparisondiffblockdiff}]
(1) We first prove that $H^* (B \blockdiff_\partial (U_{g,1}^n);\bQ)$ is degreewise finitely generated; this implies via the homological stability theorem of \cite{Grey} that $H^* (B \tilde{\cD}_\infty;\bQ)$ is degreewise finitely generated. Proposition \ref{prop:rationalized-blockdiffs} reduces this to the question whether $H^*( \hq{\map_* ((U_{g,1}^n)_\bQ;BO_\bQ)^0}{\hAut_\partial((U_{g,1}^n)_\bQ)^{\cong}};\bQ)$ is degreewise finitely generated. Using Lemma \ref{lem:pinullhautrelbound} (3), the spectral sequence for \eqref{eqn:firstkeyfibresequence} and the well-known fact that $H^* (B \GL_g (\bZ);V)$ is finite-dimensional for each finite-dimensional $\GL_g (\bQ)$-representation $V$, one reduces the problem to the question whether $H^* (\hq{\map_\partial ((U_{g,1}^n)_\bQ;BO_\bQ)^0}{\hAut_\partial((U_{g,1}^n)_\bQ)^{\id}};\bQ)$ is degreewise finitely generated. This is clear from Corollary \ref{cor:cohomology-mappingspace} and the fact that $H^* (B \hAut_\partial ((U_{g,1}^n)_\bQ)^\id;\bQ)$ is degreewise finitely generated, which is easily deduced from the description in terms of mapping spaces that we gave during the course of the proof of Lemma \ref{lem:lowdimhomotopyhomoauto1} and the fibre sequence \eqref{fibresequence}. 

By Lemma \ref{lem:plusconstruction}, $H^* (B \tilde{\cD}_\infty;\bQ) \cong H^* (\Omega_0 B B \cD;\bQ)$. As $\Omega_0 B B \cD$ is a connected $(2n+1)$-fold loop space, it is rationally equivalent to a product of Eilenberg--Mac-Lane spaces, from which the claim about the structure of $H^* (B \tilde{\cD}_\infty;\bQ)$ follows. 

(2) Let $U_g^n:= U_{g,1}^n \cup_{\partial} D^{2n+1}$. Kupers' finiteness theorem\footnote{All that is needed for the proof of Theorem \ref{mainthm:main} is that $H^* (B \cD_\infty;\bQ)$ is finite-dimensional for $* \leq n-3$, and this follows from (1) and the argument we shall give for (4). So the appeal to the deep results of \cite{Kupers} is not really necessary for our main purposes.} \cite[Corollary C]{Kupers} shows that $B \Diff^+ (U_g^n)$ is of homologically finite type if $n \geq 4$. There is a fibre sequence $\Fr^+ (TU_g^n) \to B\Diff_\partial (U_{g,1}^n) \to B \Diff^+ (U_g^n)$, and a straightforward spectral sequence argument proves that $B\Diff_\partial (U_{g,1}^n)$ is of homologically finite type. Together with \cite[Corollary 1.3.2]{Perlmutter}, this establishes the finiteness claim. The rest of the proof is parallel to that of (1). 

(3) In degrees $\leq n-3$, injectivity of $(B \rho_n)^*$ is contained in the statement of Theorem \ref{thm:blockdiffs} (and this is all what is truely needed for the proof of Theorem \ref{mainthm:main}). To get the claim in full generality, we use recent work of Stoll \cite{Stoll}. Theorem A of loc.cit. implies that 
\[
\hocolim_g B \hAut_\partial ((U_{g,1}^n)_\bQ)_\bZ \to B \GL_\infty (\bZ)
\]
induces an injective map in rational cohomology and that $H^*(\hocolim_g B \hAut_\partial ((U_{g,1}^n)_\bQ)_\bZ;\bQ)$ is a free $H^*(B \GL_\infty (\bZ);\bQ)$-module. On the other hand, the forgetful map 
\[
\hq{\map_* ((U_{g,1}^n)_\bQ;BO_\bQ)^0}{\hAut_\partial((U_{g,1}^n)_\bQ)^{\cong}} \to B\hAut_\partial((U_{g,1}^n)_\bQ)^{\cong}
\]
has a right inverse given by the constant map $U_{g,1}^n \to BO$. Using that $\hAut_\partial ((U_{g,1}^n)_\bQ)_\bZ = \hAut_\partial ((U_{g,1}^n)_\bQ)^{\cong}$ (Lemma \ref{lem:pinullhautrelbound} (4)), the claim on the injectivity of $B\rho_n^*$ therefore follows from Proposition \ref{prop:rationalized-blockdiffs}.

Now the map $B \tilde{\cD}_\infty^+\to B \GL_\infty (\bZ)^+$ is a map of $H$-spaces between connected $(2n+1)$-fold loop spaces of finite rational type. For such spaces, the rational cohomology is naturally the free graded-commutative algebra on the dual space of the rational homotopy. 
From the injectivity in rational cohomology that we just established, we deduce surjectivity in rational homotopy, and this proves the claim about the structure of $H^* (B \tilde{\cD}_\infty;\bQ)$. 

(4) For the purpose of computing cohomology, we can replace $BI_\infty$ by its plus-construction or equivalently by the map $p: \Omega_0 BB \cD \to \Omega_0 BB \tilde{\cD}$. The three spaces in the fibre sequence 
\[
\cG \stackrel{j}{\to} \Omega_0 BB \cD \stackrel{p}{\to} \Omega_0 BB \tilde{\cD}
\]
are connected double loop spaces, and $j$ and $p$ are loop space maps. Hence by the Milnor--Moore theorem \cite{MilnorMoore}, the rational homology algebras of the three spaces (with the Pontrjagin product) are the free graded commutative algebras on the rational homotopy, $H_* (p;\bQ)$ is the algebra homomorphisms induced by $\pi_* (p) \otimes \bQ$ and similarly for $j$. It follows from Lemma \ref{lem:useofpastingtheorem} that the map $\pi_* (j)\otimes \bQ$ is also the trivial map up to degree $(2n-5)$. Hence $\pi_* (p) \otimes \bQ$ is injective up to degree $(2n-5)$, so $H_*(p;\bQ)$ is injective, and $H^*(p;\bQ)$ is surjective, both in degrees $\leq (2n-5)$. 
Using the knowledge about the rational homotopy of $\cG$, we see that $\coker (\pi_* (p) \otimes \bQ)$ is concentrated in degrees $4k+1$, $k \geq 1$ and in these degrees has dimension $1$ (again in degrees $\leq (2n-5)$).

Hence the kernel of $H^*(p;\bQ)$ must, in degrees $\leq (2n-5)$, be an ideal generated by classes in each degree $4k+1$, $k \geq 1$.
However, we know by Proposition \ref{prop:vanishingdww} that the ideal generated by the Borel classes lies in the kernel of $H^*(p;\bQ)$, and by a dimension count, justified by items (1),(2),(3), must be equal to the kernel. This finishes the evaluation of $H^*(\Omega_0 BB \cD;\bQ)$, which by Lemma \ref{lem:plusconstruction} gives $H^* (B \cD_\infty;\bQ)$.
\end{proof}

\bibliographystyle{plain}
\bibliography{diffalgK}

\end{document}